\documentclass{amsart}
\usepackage{hyperref}
\usepackage{bbm}
\usepackage{stmaryrd}

\usepackage{amssymb}
\usepackage{tikz}
\usetikzlibrary{cd}
\usetikzlibrary{matrix,arrows}

\makeatletter

\newfont{\gothic}{eufm10 scaled 1100}

\theoremstyle{plain}    
\newtheorem{thm}{Theorem}[section]
\numberwithin{equation}{section} 
\numberwithin{figure}{section} 
\theoremstyle{plain}    
\newtheorem{cor}[thm]{Corollary} 
\theoremstyle{plain}    
\theoremstyle{plain}    
\theoremstyle{plain}
\newtheorem{lem}[thm]{Lemma} 
\theoremstyle{plain}    
\newtheorem{prop}[thm]{Proposition} 
\theoremstyle{plain}    
\newtheorem{Def}[thm]{Definition} 
\theoremstyle{remark}
\newtheorem{rem}[thm]{Remark}
\theoremstyle{remark}    

\theoremstyle{remark}    
\newtheorem{exm}[thm]{Example}

\usepackage[arrow,matrix,curve,ps]{xy}
\xyoption{dvips}
\CompileMatrices

\makeatother

\begin{document}

\title{Grothendieck's Equality vs Voevodsky's Equality}


\author{Thomas Eckl}




\address{Thomas Eckl, Department of Mathematical Sciences, The University of Liverpool, Mathematical
               Sciences Building, Liverpool, L69 7ZL, England, U.K.}

\email{thomas.eckl@liv.ac.uk}

\urladdr{http://pcwww.liv.ac.uk/~eckl/}

\maketitle

\begin{abstract}
We discuss how canonical and universal constructions, properties and characterizations interact with equality in the framework of Homotopy Type Theory, comparing it with Grothendieck's use of equality and shedding further light on (efficient) formalisation of mathematics. This is achieved by investigating examples that range from monoids, groups, rings and modules to cohomology theories in the category of modules over commutative rings and culminate in a cohomological criterion of flatness. 
\end{abstract}


\pagestyle{myheadings}
\markboth{THOMAS ECKL}{GROTHENDIECK VS VOEVODSKY}

\setcounter{section}{-1}

\section{Introduction}

Unsurprisingly, the ongoing massive collaborative efforts to formalize mathematics, especially in the Mathlib \cite{Mathlib} based on the interactive theorem prover Lean \cite{Lean} and the projects derived from it, reveal glitches in the standard ways developed in the past century to ensure mathematical rigour. Not only the soundness of certain practices is scrutinized, but also their efficiency when it comes to formalization -- the very name of Lean alludes to considerations of the latter sort, and the trend away from set-theoretic foundations like the ZFC axioms towards type theory can be motivated along those lines. But the jury is still out on that matter, and is also so on the question which flavour of type theory is best suited for formalization (see e.g. \cite{BPL22}), and further investigations are needed.

In his recent paper \cite{Buz24} (see also \cite{BHL+22}) Buzzard reflects on his experiences when formalising more and more advanced chunks of mathematics belonging to very specific theories, culminating in the construction of Scholze's perfectoid spaces \cite{BCM19} (but by no means stopping there \cite[Fermat's Last Theorem]{BuzXena}). His most striking case study, however, looks at a rather basic topic in Commutative Algebra: the localisation of rings. It already exhibits several problems which will be important in the formalisation of many parts of mathematics, as they seem to require a deviation from well-established proof practices, and thus tie the discussion back to foundational issues:
\begin{itemize}
\item Formal proofs of properties of objects satisfying a universal property, like the localisation of a ring in a multiplicatively closed set, are inefficient or even impossible by just using the universal property. Instead, a specific construction of the object is needed.
\item If a property is shown for a specifically constructed object with a universal property it can be difficult to transfer it to other objects  satisfying the same universal property. The standard way of using the unique isomorphism between these objects induced by the universal property often requires a lot of supporting technical lemmas.
\item Grothendieck resolved this issue by simply identifying two objects satisfying the same universal property, calling the unique isomorphism between them ``canonical". This word pervades all of EGA (\cite{EGAIn}, ...) and all the literature based on it, but is nowhere properly defined and quite often used in an outright wrong way. 
\item Furthermore, this identification is not possible in the foundational framework of the Lean Mathlib \cite{Mathlib}. What proved to be efficient for localisations was rewriting the statements on localisation such that they apply to all rings satisfying a carefully chosen property proposed by Neil Strickland, which implies the usual universal property. Buzzard suggests that this will be the best way to deal with problems caused by universal objects in many cases: Find the most suitable characterisation of the universal objects which even may involve new mathematical ideas.
\item In \cite[Sec.~8]{Buz24} Buzzard also discusses the abuse, or simply incorrect use of the word ``canonical" for constructions involving sign choices, for example when defining boundary maps in homological algebra. He argues that there not only are several such choices, but none of them is preferable to the others without further context. Indeed, it seems best to record the choices used for a construction, as did Grothendieck by including the boundary maps as data of the universal $\delta$-functors. 
\end{itemize}

In the following, I will discuss these problems in the foundational framework of Homotopy Type Theory (for short: HoTT), using mostly the same examples as in \cite{Buz24}. At a first glance (see for example \cite{HoTTBook, Rij22}), one of the main ingredients of HoTT seems to offer a ready-made formal analogue of Grothendieck's identification practice: Univalence (\cite{Voe06}, \cite[Introduction \& Ch.4]{HoTTBook} and Sec.~\ref{univ-ssec}) postulates that the type of equalities is equivalent to the type of equivalences between two types, and equivalences between two types sit at the core of isomorphisms between objects in many categories, like rings, modules, schemes etc (see~\cite[Ch.9]{HoTTBook} and Sec.~\ref{cat-ssec}). Unfortunately, we will see that matters are more complicated; in short, that is due to univalence being more about attaching equivalences to equalities than about identifying equivalent types. Nevertheless, I will argue that the concepts and notions of Homotopy Type Theory allow to formulate the problems above in a clearer way, lead to recipes and guidelines on how to solve them and also provide a better understanding of why some of these approaches have limitations, whereas others are more efficient:

\begin{itemize}
\item In Homotopy Type Theory, universal properties are induction principles, usually of higher inductive types. Thus, universal properties are associated to (higher) inductive constructions, immediately providing objects satisfying the universal properties. 
\item They can also be used to characterize such objects, up to unique equality = canonical isomorphism, by univalence. In general, there are two strategies to produce such characterisations: Either, we require functions corresponding to the constructors and make sure that the constructors deliver all elements of the objects, and in a unique way, or we extract the inputs of constructors in a unique way from all the elements. Usually, an actual object satisfying the universal properties is needed to show the equivalence of the characterisations.

 Typical examples of following either strategy are discussed in Sec.~\ref{set-prod-ssec} on Cartesian products, Sec.~\ref{nat-type-ssec} on natural number types, and Sec.~\ref{set-quot-ssec} on set quotients, a higher inductive type. Strickland's characterisation of localized rings in Sec.~\ref{ring-mod-ssec} falls into the first kind, and the equational criterion for tensor products presented in Sec.~\ref{tensor-ssec} is another instance.
\item Characterisations of the first kind can often be simplified by condensing the equations ensuring the uniqueness of inputs needed to produce a given output by using a constructor. This is what happens in Strickland's characterization of localizations, but can also be seen with group quotients (see Prop.~\ref{quot-group-def}) and tensor products (Prop.~\ref{eq-char-tensor-prop} and~\ref{tensor-eq-crit-prop}).
\item Choices in universal constructions do not destroy uniqueness, but just lead to equivalent, and therefore by univalence equal objects. In Section~\ref{choice-sec}, we discuss these phenomena extensively on possible sign choices when constructing boundary maps in homological algebra, and how they influence the construction of universal $\delta$-functors. Choices can be either constructive in all cases, or they rely on the Axiom of Choice in general, but still can be constructed in special cases, for example when producing liftings required in the definition of projective $R$-modules (see~\ref{proj-mod-res-ssec}).
\item Most of the time, the presences of choices does not necessitate cumbersome checks of their equivalence when calculating with the objects produced by them. By just requiring their mere existence, as an indication that we do not want to fix a construction, we can still prove theorems, that is, propositions in HoTT terminology: their contents is completely encapsulated in their statements, as proofs are all equal. Then the mere existence can be resolved by induction on propositional truncation (see~Sec.~\ref{prop-set-ssec}), even leaving us the choice among the possibly several constructions.

We demonstrate this procedure in Sec.~\ref{calc-cohom-sec}, when proving a flatness criterion for (affine) hypersurfaces. 
\end{itemize}

The article is written from the viewpoint of a ``working mathematician"; in particular, the ``we" refers to their community. One consequence of this decision is that this paper does not systematically study concepts of foundations of mathematics. For example, we will not present the basic formal types of inductive constructions like $W$-types (see \cite[5.3]{HoTTBook}). Instead, by extensively exemplifying on how to produce different characterisations of inductively defined types and discussing their merits in terms of efficiency in (formalised) proofs, we keep the contents of the article closer to the practice of many mathematicians.

However, the exact characteristics of such a ``working mathematician" can vary a lot and may even be elusive, so in case of doubt, they boil down to just my own perceptions. But since I cannot consider myself to be an original genius I at least hope that many of my assumptions on the thinking of mathematicians are shared by those with an education similar to mine. If not, the resulting misrepresentations are completely my fault. 

In view of the rising use of Large Language Models and other reasoning models together with automatic formalization to Lean code allowing proof certification, in the race to IMO gold medals \cite{BuzXena2}, high scores in the Putnam Benchmark \cite{LogInt1} and more and more actual mathematical research results (see e.g. \cite{BuzXena3}; they are admittedly just in a restricted number of mathematical fields, but also see \cite{Gauss}), the emphasis of this article on how well certain ways of formalization fit to human mathematical practice may seem outdated. But understanding the characteristics and workings of ``human mathematics" as opposed to ``formal mathematics", that is the totality of all possible formal deductions, may very well be essential for AI agents to actually produce research level mathematics in more than just a few fields.

The terms above were coined by Aksenov, Bodia, Freedman and Mulligan in their paper \cite{ABFM26}, where they argue, and underpin by a test of the Lean Mathlib \cite{Mathlib}, that the nested hierarchy of definitions, lemmas and theorems (they forgot to mention examples) allows the human mind to encompass an exponentially growing collection of mathematical expressions, by using compression at a logarithmic rate. The authors do not draw an even more striking consequence of their findings: Any mathematics produced by a sort of ``calculator", human or not, will fall into the category of ``human mathematics" if all other parts of mathematics are computationally too complex to be discovered -- that is, if they are NP-hard to find (see Widgerson's book \cite{Wid19} on these all-important notions from the Theory of Computation). And if this is the case, exhibiting the successful strategies of mathematicians to produce new mathematical results is essential for achieving anything interesting with AI tools, as powerful as they may be: if $\mathrm{P} \neq \mathrm{NP}$ they cannot beat the complexity barrier.    

The article is written in an unusually detailed style, on the surface just regurgitating arguments well-known from standard text books and monographs. However, this close inspection is necessary to ensure that Homotopy Type Theory can adapt the proofs, and the close resemblance is a reassurance that Homotopy Type Theory can faithfully follow standard mathematical practices. Finally, that way the article can serve as an at least first sketch of a ``blueprint" established as a standard step towards formalization of mathematical material (see e.g.\ Tao's report in \cite{TaoBlog23}).

As the proof of considerations on formalization is formalizing the exemplifying mathematical material, the contents up to subsection 3.4 and some parts beyond can be found as Lean code in \cite{Lean-HoTT}. The tex-file of the article can be searched for tags of definition names in the Lean code. The project is based on Ebner's HoTT 3 Library \cite{HoTT3} using (an early version of) Lean 3, so must be ported to Lean 4 as soon as possible if it should be useful for future purposes. Fortunately, there already exist HoTT Libraries based on Lean 4, for example \cite{HoTT4}, which implements the restrictions imposed by HoTT on the Lean foundations in the same way as Ebner's library by introducing a large eliminator checker.

\section{Equalities in Homotopy Type Theory}   

In many type theories, not just dependent type theories or Homotopy Type Theory, types can be introduced by \textit{inductive constructions} (see \cite[5.1]{HoTTBook}). They consist of \textit{constructors} postulating the existence of objects and determining how objects can be produced from objects of other types or recursively from objects of the same type, as long as the recursion does not create infinite loops. From these constructors, a \textit{recursion} resp.\ \textit{induction principle} telling how to produce a function from the inductively defined type to a (dependent) type can be automatically extracted, together with the \textit{calculation rules} showing the evaluation of such functions on the constructed objects. Note that the function values can be propositions, and that in dependent type theories like HoTT, their type may depend on objects of the inductively defined type.

The easiest inductive definitions just contain constructors postulating the existence of an object of the type. For example, the type $\mathbbm{1}$ has one constructor $\ast: \mathbbm{1}$, whereas the type $\mathbbm{2}$ has two constructors $0_\mathbbm{2} : \mathbbm{2}$ and $1_\mathbbm{2} : \mathbbm{2}$. Functions on $\mathbbm{1}$ and $\mathbbm{2}$ are constructed by simply prescribing the function values of $\ast$ and $0_\mathbbm{2}, 1_\mathbbm{2}$, respectively. An extreme example is the inductive type $\mathbbm{O}$ that has no constructor at all. In turn, this means that we can construct a function $\mathbbm{O} \rightarrow T$ to any type $T$. Thus, $\mathbbm{O}$ acts like the false proposition from which we can deduce everything: Ex falso quodlibet sequitur. We also denote $\mathbbm{O}$ by $\mathbf{False}$.


We will discuss more advanced inductive types whose definition also uses recursion, in the next sections. In this section, we concentrate on the type of equalities between objects of a given type.

\subsection{Identity Types}

In Homotopy Type Theory, identity types $a =_A b$ of two objects $a, b : A$ are inductively defined: For given $a$, only for $a =_A a$ we can construct an equality, called \textit{reflexivity} or for short, $\mathtt{refl}$. This seems to restrict a lot the possible equalities, but already the constructor comprises more than visible at first glance: In Homotopy Type Theory, or  Type Theory in general, objects (or terms) are \textit{judgmentally} equal if they can be reduced to the same terms when using the calculation rules of inductive types (or function evaluation). In that sense, $2 + 2$, $2 \times 2$ and $4$ are all judgmentally equal, and the equalities are all identical to $\mathtt{refl}(4)$. We sometimes denote judgmental equalities by $2 + 2 \equiv 4$ etc. 

Furthermore, the existence of only one constructor does not prevent an identity type to have different inhabitants, so that there are more than one way that two objects of certain types can be equal, at least if we assume univalence -- we will discuss this in Sec.~\ref{univ-ssec} in more details.
In particular, Homotopy Type Theory does not postulate that equalities always are unique, or a special kind of ``sort", different from types, as does the ``logic enriched" type theory underlying the Lean Mathlib \cite{Mathlib}.

The characterisation of identity types between arbitrary objects of a given type requires some work. By induction, the only identity type of objects in $\mathbbm{1}$ to consider is $\ast =_{\mathbbm{1}} \ast$, and it is inhabited by $\mathtt{refl}(\ast)$. With univalence (see Sec.~\ref{univ-ssec}), we can show even more: $\ast =_{\mathbbm{1}} \ast$ and $\mathbbm{1}$ are equivalent, hence equal, so there is only one equality of type $\ast =_{\mathbbm{1}} \ast$. For the type $\mathbbm{2}$, we can use the \textit{encode-decode-method} to determine equalities: The \textit{code} maps $(0_\mathbbm{2}, 0_\mathbbm{2})$ and $(1_\mathbbm{2}, 1_\mathbbm{2})$ to $\mathbbm{1}$ and $(0_\mathbbm{2}, 1_\mathbbm{2})$ and $(1_\mathbbm{2}, 0_\mathbbm{2})$ to $\mathbbm{O}$. By double induction, we can \textit{encode} equalities $a =_{\mathbbm{2}} b$ by mapping them to $\ast : \mathbbm{1}$ and \textit{decode} codes to equalities by mapping $\ast$ to $\mathtt{refl}(0_\mathbbm{2})$ and  $\mathtt{refl}(1_\mathbbm{2})$, respectively, or deduce anything from $n : \mathbbm{O}$. Encoding and decoding are inverse to each other, so $n =_{\mathbbm{2}} m$ and $\mathtt{code}(n,m)$ are equivalent, hence by univalence equal, and furthermore, equalities between two objects of type $\mathbbm{2}$, if they exist, are equal.


The encode-decode method and its more explicit variants like the \textit{structure identity principle} (see \cite[11.6]{Rij22}) can be applied to every inductively constructed type (see \cite[2.13]{HoTTBook} for the natural numbers). The case of \textit{dependent pairs}, or \textit{$\Sigma$-types}, shows how to proceed when constructors have different types as arguments: Dependent pairs $(a_1, b_1), (a_2, b_2) : \Sigma_{a : A} B(a)$ over a type family $B : A \rightarrow \mathcal{U}, a \mapsto B(a)$ are equal if, and only if, we have an equality $p : a_1 =_A b_1$ and $b_1 : B(a_1)$ \textit{transported} to $B(a_2)$ is equal to $b_2$. The \textit{transport} in type families $B : A \rightarrow \mathcal{U}$ over equalities $p : a_1 =_A a_2$ can be defined by induction on the equality $p$ since $b_1$ should be transported to itself if $p \equiv \mathtt{refl}(a_1)$. Since most (algebraic) structures are nested $\Sigma$-types, this criterion allows to prove their equalities iteratively.


\subsection{Propositions and Sets} \label{prop-set-ssec}

Since identity types are types, there are identity types of equalities and identity types of equalities of equalities. Since it is also possible to invert and concatenate equalities, and the resulting equalities satisfy the expected laws, we can speak of the \textit{higher groupoid structure} of types. On the other hand, in many cases, all objects of a type are equal, or at least there is only one equality between two objects, and the ``higher" equalities are unique (see \cite[Lem.3.3.4]{HoTTBook} for a proof of one instance of this fact).
\begin{Def}
A type $P$ is a \textbf{proposition} if for all $x, y : P$ we have $x = y$.
\end{Def}  
\begin{Def}
A type $S$ is a \textbf{set} if for all $x, y : S$ and equalities $p, q : x = y$, we have $p = q$.
\end{Def}

Among the examples above, $\mathbbm{O}$ and $\mathbbm{1}$ are propositions. The results of the encode-decode method show that $\mathbbm{2}$ is a set, since there is at most one equality, encoded as $\ast$, between two objects of $\mathbbm{2}$. In the same way, we can show for example that the natural numbers $\mathbb{N}$ are a set (see \cite[Ex.3.1.4]{HoTTBook}).


The name ``proposition" is sometimes misleading because it alludes to only one of the possible uses of propositions: as statements which can have a proof or not, but the exact nature of the proof is \textit{irrelevant}. Another use of propositions is as types that encode \textit{uniquely determined choices}.

In Homotopy Type Theory, propositions are the natural kind of types on which we can perform logical operations: The negation of a proposition $P$ is a function $\mathtt{not\ } P :\equiv P \rightarrow \mathtt{False}$, and we can connect two propositions $P, Q$ to $P \mathtt{\ and\ } Q :\equiv P \times Q$, $P \mathtt{\ or\ } Q :\equiv P + Q$ and $P \mathtt{\ implies\ } Q :\equiv P \rightarrow Q$. If we introduce $\mathtt{not}$, $\mathtt{and}$, $\mathtt{or}$ and $\mathtt{implies}$ on arbitrary types, univalence will cause ``paradoxa", for example a contradiction to the Law of Excluded Middle (see \cite[Cor.3.2.7]{HoTTBook}). Of course, even the Law of Excluded Middle for propositions cannot be deduced from the basic type theory, but we can introduce it as an axiom, without causing inconsistencies:
\[ \mathrm{LEM}_{-1}\ :\ \mathrm{For\ all\ propositions\ } P,\  P + \neg P. \]   


For arbitrary types, it is possible to cut off all the higher equalities by \textit{propositional truncation}. This is actually a first example of a higher inductive construction: For each type $A$, we construct a type $\| A \|$ with a constructor $t : A \rightarrow \| A \|$ and a higher constructor of $t(|a|) = t(|b|)$ for each $|a|, |b| : \| A \|$, making $\| A \|$ into a proposition by \textit{truncating} the higher equalities. The recursion principle delivers a map $g : \| A \| \rightarrow B$ for each proposition $B$ with a map $f : A \rightarrow B$ such that $g(t(a)) = f(a)$, for all $a : A$. 
If we have an object $a : \| A \|$ we say that there \textit{merely} exists an object of type $A$.

From an object of a propositionally truncated type we cannot construct an object of the type itself, unless it is a proposition. But if we want to construct an object of a proposition and another object of a propositionally truncated type is available, we can apply induction on propositional truncation and continue to work with an object of the type, without truncation, and without checking independence from the choice. These kind of arguments will be at the core of our treatment of choices in the constructions of Sec.~\ref{choice-sec}. The problem solved there is not that we do not know how to construct objects of a specific type (which may or may not be the case), but that we do not know (or do not want to decide) which of the many possibilities to choose --- and we do not need to specify our choice as long as we only want to prove a proposition, or construct another uniquely determined object.

Iteratively we can also truncate arbitrary types $A$ to sets $\| A \|_0$, by propositionally truncating the identity types $a =_A b$ to $\| a =_A b \|$, for all $a, b : A$. Objects of $\| A \|_0$ are also objects of $A$, but equalities $a =_A b$ merely exist, for all $a, b : A$.

Propositional truncation is essential when formulating the Axiom of Choice. The naive translation to Homotopy Type Theory for a set $X$ and a family of sets $Y(x)$ over $X$,
\[ \mathrm{AC}_\infty\ :\ \mathrm{If\ for\ all\ } x : X\ \mathrm{there\ is\ } y : Y(x)\ \mathrm{then\ there\ is\ a\ function\ } \Pi_{x : X}, Y(x), \]
is a triviality, because functions in Homotopy Type Theory are always considered to be explicit, so they automatically make a choice for each argument. What the the Axiom of Choice really states is
\[ \mathrm{AC}_0\ :\ \mathrm{If\ for\ all\ } x : X\ \mathrm{there\ is\ } y : \| Y(x) \|\ \mathrm{then\ there\ is\ } f : \| \Pi_{x : X}, Y(x) \|, \]
that is, if we merely know the existence of objects in each $Y(x)$, without any rule to produce such elements, the axiom allows us to make a simultaneous choice for all $x : X$, but again merely knowing the existence of the resulting function.


\subsection{Univalence and Function Extensionality} \label{univ-ssec}

Up to now, we mainly discussed equalities in inductively defined types. This does not tell us how to construct and characterise equalities of types themselves, or of functions between types.

Univalence deals with equalities $A = B$ of arbitrary types $A, B$ in the same universe $\mathcal{U}$. Univalence postulates the equivalence of identity types and equivalence types, for short:
\[ \mathrm{UA}\ : \ (A =_\mathcal{U} B) \simeq (A \simeq B). \]
One way to define \textit{equivalences} $A \simeq B$ are as pairs of functions $f : A \rightarrow B$ and $g : B \rightarrow A$ inverse to each other, checked by equalities
\[ (f \circ g)(b) = b\ \mathrm{for\ all\ } b : B\ \mathrm{and\ } (g \circ f)(a) = a\ \mathrm{for\ all\ } a : A, \]
together with additional higher equalities between these equalities (see \cite[Def.4.2.1]{HoTTBook} for the details). 

For practical uses of univalence, it is most of the times enough to present the inverse maps $f$ and $g$, as \textit{adjointification} (\cite[Thm.4.2.3]{HoTTBook}) allows to construct a proper equivalence from these data without changing $f$ and $g$ (but possibly some of the equalities checking inverseness). If the types $A$ and $B$ are sets, the higher equalities are irrelevant (since uniquely determined) anyway, so the inverseness of $f$ and $g$ is sufficient. Therefore, univalence completely characterizes equalities of sets:
\begin{prop}   
Equalities $p : A = B$ of sets $A, B$ are in one-to-one correspondence with bijective maps $f : A \rightarrow B$. \hfill $\Box$
\end{prop}
Thus, there always will be a multitude of equalities between sets of the same cardinality $> 1$. This concept certainly contrasts with the definition of equality of sets in Zermelo-Fraenkel set theory, by the Axiom of Extensionality: Two sets are equal if they contain the same elements. In Homotopy Type Theory, on the other hand, the equivalences $=$ bijective maps between sets allow to jumble the elements freely and still produce an equality. This in turn seems to contradict mathematical instincts when applied for example to the set of natural numbers: An equality of $\mathbb{N}$ with itself should better lead to an equivalence identifying $0$ with itself. 

However, the mathematical instinct here reacts to the additional semiring structure on natural numbers, not present when we only discuss $\mathbb{N}$ as a set. Objects of a set are not incorporated in any additional structure and therefore cannot be distinguished. That allows them to be arbitrarily identified by equalities between the sets.

If on the other hand we provide $\mathbb{N}$ with the standard commutative semiring structure $s : \mathtt{isCommSemiring}(\mathbb{N})$, any equality of the dependent pair $(\mathbb{N}, s)$ with itself automatically induces a bijective map $\mathbb{N} \rightarrow \mathbb{N}$ preserving the commutative semiring structure $s$ --- in other words, an automorphism of the commutative semiring $\mathbb{N}$. Part of the structure says that $\mathbb{N}$ has $0$ as a zero, so the underlying bijection $\mathbb{N} \rightarrow \mathbb{N}$ of the automorphism must in particular map $0$ to $0$. 

Univalence and in extension Homotopy Type Theory is usually promoted as a tool to transport structures and properties from one type to another type equal to it, without need of further constructions or proofs. While this claim is certainly true, the untrained eye may exaggerate its benefits: For example, if we want to compare the transported structure with another one already present on the second type, we need to execute the checks we thought to have avoided, namely that the equivalence induced by the equality preserves the structures. 

In \cite[2.1]{VMA21}, a telling example of this problem can be found (see also \cite[5.2]{HoTTBook}): The inductive definition of natural numbers by introducing $0$ and the successor function is very efficient when proving the laws of calculations collected in the semiring structure for $\mathbb{N}$. On the other hand, it does not allow efficient calculations --- for this purpose, binary (or alternatively decimal, \ldots) numbers, call them $\mathrm{Bin}\mathbb{N}$, are much more suitable; however, proving the calculation laws for these representations of natural numbers is very tedious. In many arguments involving calculations with natural numbers you may want to use both laws and actual calculations, so to do this efficiently we would need to switch back and forth between the two representations of natural numbers. Now, it is easy to show that $\mathbb{N}$ and $\mathrm{Bin}\mathbb{N}$ are equal as sets, and we could construct many such equalities. But for the aim outlined above, you actually need the (unique) equality preserving addition and multiplication, together with $0$ and $1$, so you end up with constructing an isomorphism of semirings. We will further discuss this example at the end of Sec.~\ref{ind-cons-sec}.


In the end, univalence makes equalities of structures and structural isomorphisms literally and in a very precise technical sense \textit{equivalent}, so it is to be expected that using univalence instead of isomorphisms may not necessarily simplify arguments. In the example above, univalence will become really beneficial to calculations if it produces a \textit{judgmental} equality between the different presentations of the natural numbers and their semiring structures because then the switch between $\mathbb{N}$ and $\mathrm{Bin}\mathbb{N}$ will be performed automatically as a term reduction. \textit{Cubical Type Theory} provides such a construction, turning the Univalence Axiom $\mathrm{UA}$ into a constructively provable statement, and its effects are further described in \cite{VMA21}. Apart from some remarks in the final section, we will however not pursue these developments further in this article.

The final type of equalities we need to discuss are those between functions $f : A \rightarrow B$ from type $A$ to $B$. We want them to be equal if their function values $f(a)$ and $g(a)$ are equal for all inputs $a : A$. This principle is called \textit{Function Extensionality}. As with Univalence, in most presentations of Homotopy Type Theory this is still considered to be an axiom; however, it can also be deduced from Univalence (see \cite[4.9]{HoTTBook}, and Cubical Type Theory contains it as a constructively provable statement (see \cite{CCHM18}). 

\subsection{Categories} \label{cat-ssec}

Following \cite[Ch.~9]{HoTTBook} we call a type of objects together with a family of homomorphism \textit{sets} between two objects, provided with identity homomorphisms and a composition of homomorphisms satisfying the usual laws, as a \textit{precategory}. If we relax the requirement of homomorphisms between two given objects being a set to allowing higher equalities, we obtain $2$- or higher (pre)categories. When objects have underlying sets, and homomorphisms have an underlying map of sets, homomorphisms usually form sets, as maps between the underlying sets are a set again. For example, this happens for categories of algebraic objects like monoids, groups and rings. Thus, this definition of categories should not be seen as restrictive, but as a useful starting point already covering many interesting examples, and allowing a systematic extension to more complicated kinds of categories.

A precategory is called a \textit{category} if isomorphisms and equalities between two given objects are \textit{equivalent}. Sometimes, such a category is also called a \textit{univalent} \mbox{($1$-)category}, because the equivalence between isomorphisms and equalities can be seen as a generalisation of univalence. If objects of the category have underlying sets and homomorphisms have an underlying map of sets, the equivalence is usually indeed derived from univalence. As with univalence, the identification of equalities with isomorphisms should not be seen as destroying the structure given by isomorphisms, but actually as providing equalities with this structure. 

\section{Universal properties from inductive constructions} \label{ind-cons-sec}

The standard ingredients of an inductive construction hide another important feature of an inductively defined type (which always can be proven using the inductive principle): The constructors \textit{freely generate} the objects of the type. First of all, this means that every object is the result of exactly one of the constructors, and in a unique way. But if some constructors are recursive, even more is true: there is a unique way to produce an object by iteratively applying constructors. We discuss how to turn this into characterisations of inductive types on two examples, Cartesian products of two and more factors, and the natural numbers.

\subsection{Cartesian products} \label{set-prod-ssec}

The only constructor of the Cartesian product $A \times B$ of two types delivers a pair $(a, b)$ from two given objects $a$ of type $A$ and $b$ of type $B$. The recursion principle then associates a function $g : A \times B \rightarrow C$ to a function $f : A \rightarrow B \rightarrow C$ (that is, a function that gives a function $B \rightarrow C$ for each object $a$ of type $A$), and the calculation rule states that $g((a,b)) = f(a)(b)$ (and similarly for the induction principle). This may look less familiar than induction on $\mathbb{N}$, first of all because the difference between $g((a,b))$ and $f(a)(b)$ seems insignificant. Even worse, we are used to the universal property of a Cartesian product delivering a (unique) map $\phi$ from a type $C$ \textit{to} $A \times B$, for given maps $f: C \rightarrow A$, $g : C \rightarrow B$ and factorising $f$, $g$ by the (inductively defined) natural projections $p_A : A \times B \rightarrow A$, $p_B : A \times B \rightarrow B$, not the other way round. But how do we prove this universal property? We use the free generatedness of the Cartesian product, that is, each $p : A \times B$ is equal to the pair $(p_A(p), p_B(p))$, and for all $a : A$, $b : B$ there is exactly one object $p : A \times B$ such that $p_A(p) = a$ and $p_B(p) = b$! 

We will spell this out by first using the free generatedness of its constructor as a characterisation of Cartesian products of types:
\begin{Def} \label{prod-cons-char} 
A type $P$ is a \textbf{Cartesian product} of two types $A$ and $B$ if
\begin{itemize}
\item[(i)] there is a function $f : A \rightarrow B \rightarrow P$ (called \textbf{pair} constructor), and
\item[(ii)] for all $p : P$, there are unique $a : A$, $b : B$  such that $p = f(a)(b)$. 
\end{itemize}
\end{Def}

Induction immediately shows that the inductively defined product $A \times B$ satisfies this characterisation.
Condition \ref{prod-cons-char}(ii) allows the construction of \textit{projection maps} $p_A : P \rightarrow A$,
$p_B : P \rightarrow B$ such that $p_A(f(a)(b)) = a$ and $p_B(f(a)(b)) = b$, and $f(p_A(p),p_B(p)) = p$ for all $p : P$.   

From the characterisation we can also deduce the induction principle and the computation rules, and vice versa:
\begin{prop}  \label{cons-ind-prod-Def} 
A type $P$ is a Cartesian product of two types $A$, $B$ with a pair constructor $f : A \rightarrow B \rightarrow P$ if, and only if, for all types $Q$ with a map $g : A \rightarrow B \rightarrow Q$, there is a unique map $h : P \rightarrow Q$ such that $g = h \circ f$.
\end{prop}
\begin{proof}
If $P$ is a Cartesian product of $A, B$ and $Q$ a type with a map $g : A \rightarrow B \rightarrow Q$, we define $h(p)$ to be $g(p_A(p), p_B(p))$. Obviously, $g = h \circ f$, which also implies uniqueness of $h$ using $f(p_A(p),p_B(p)) = p$ and function extensionality. In particular, $A \times B$ satisfies the induction principle.

The converse is also true: If the induction principle holds, we can apply it on the maps $A \rightarrow B \rightarrow A$ and $A \rightarrow B \rightarrow B$ and obtain projection maps $p_A : P \rightarrow A$ and $p_B : P \rightarrow B$ such that $p_A(f(a)(b)) = a$ and $p_B(f(a)(b)) = b$. If furthermore $f(p_A(p),p_B(p)) = p$ for all $p : P$, we have shown Def.~\ref{prod-cons-char}(i), and (ii) follows from the computation rules of the projection maps. But the universal property also gives us maps $P \rightarrow A \times B$ and $A \times B \rightarrow P$ whose composition maps $p : P$ to $f(p_A(p),p_B(p))$. By uniqueness, this composition must be the identity map, hence the desired equality.  
\end{proof}

We can interpret the characterisation in this proposition as the universal property of an initial object in the category of \textit{pair types} associated to types $A$ and $B$, whose objects consist of types $Q$ and a \textit{pair map} $A \rightarrow B \rightarrow P$, and whose morphisms are maps $P \rightarrow Q$ factorizing the pair maps.

The two characterisations of a Cartesian product thus imply each other, but are also quite different: One requires a quantification over all types (in a universe), the other states the free generatedness of an inductive construction, thus decribing the objects of the inductive type directly, and therefore looking more efficient. 

Another method to obtain characterisations of a Cartesian product uses the projections retrieving the arguments with which objects in the Cartesian product are constructed and their calculation rules, or alternatively their universal properties, thus making it into a terminal object of a category of ``types with projections":

\begin{prop} \label{prod-proj-calc-char} 
A type $P$ is a Cartesian product of two types $A$ and $B$ if and only if
\begin{itemize}
\item[(i)] there are functions $p_A : P \rightarrow A$, $p_B : P \rightarrow B$ (called \textit{projections}), and
\item[(ii)] for all $a : A$, $b : B$ there is a unique $p : P$ such that $p_A(p) = a$ and $p_B(p) = b$. 
\end{itemize}
\end{prop}
\begin{proof}
If $P$ is a Cartesian product as in Def.~\ref{prod-cons-char}, we can take the projection functions $p_A, p_B$ constructed after the definition which also satisfy (ii).

Conversely, (ii) provides a pair constructor $f : A \rightarrow B \rightarrow P$, and $f(a)(b) = f(a^\prime)(b^\prime)$ implies $a = a^\prime$, $b = b^\prime$, by applying the projections.
\end{proof}

\begin{prop}  
A type $P$ is a Cartesian product of two types $A$, $B$ with projections $p_A : P \rightarrow A$, $p_B : P \rightarrow B$ if, and only if, for all types $Q$ with maps $q_A : Q \rightarrow A$, $q_B : Q \rightarrow B$ there is a unique map $f : Q \rightarrow P$ such that $p_A \circ f = q_A$ and $p_B \circ f = q_B$.
\end{prop}
\begin{proof}
We show that the universal property is equivalent to the characterisation in Prop.~\ref{prod-proj-calc-char}.

If $P$ is a Cartesian product of $A, B$, for all $q : Q$ we define $f(q)$ as the unique $p$ such that $p_A(p) = q_A(q)$ and $p_B(p) = q_B(q)$. By function extensionality these equations imply $p_A \circ f = q_A$ and $p_B \circ f = q_B$, and the uniqueness of $f$ follows from the uniqueness of $p$.

The converse is also true: Just apply the standard universal property to the (inductively defined) unit type $\mathbbm{1}$ with just one object $\ast$ and the two maps $\mathbbm{1} \rightarrow A$, $\mathbbm{1} \rightarrow B$ given by $\ast \mapsto a$ and $\ast \mapsto b$, for any given $a : A$, $b : B$, to obtain the unique object $p : P$ with $p_A(p) = a$ and $p_B(p) = b$.
\end{proof}

Once again, the second characterisation requires a quantification over all types (in a universe), whereas the first decribes the objects of the inductive type directly, and therefore looks more efficient. As seen above, even the characterisation by constructor and free generatedness needs the projections, so we will indeed base our proofs later on only on this characterisation.

The usual yoga of universal objects tells us that for each type $P$ together with two maps $q_A : P \rightarrow A$ and $q_B : P \rightarrow B$ that satisfy either of the two previous characterisations, we have unique functions $P \rightarrow A \times B$ and $A \times B \rightarrow P$ that are inverse to each other and factorize the projection maps. ``Adjointifying" (see \cite[Thm.4.2.3]{HoTTBook} turns these inverse constructions into an equivalence of products respecting their projection maps, on which we can apply univalence to show that they are equal.

However, in Homotopy Type Theory it does not follow that the equality is ``canonical", that is, that two such equalities always must be equal: such a ``higher" equality for example involves equality of the equalities verifying factorization of the projection maps, and this may fail in arbitrary types. But they are automatically equal if  we restrict to types that are \textit{sets}, in the HoTT sense (see \ref{prop-set-ssec}). It is also straightforward to prove that the Cartesian product inductively defined above is a set if the two factors are sets (see \cite[2.6]{HoTTBook}). 
Thus, the Cartesian product inductively constructed above from two sets and any set satisfying either of the two universal properties for the same two factor sets are always (uniquely) equal in HoTT, justifying Grothendieck's practice for the set product. 

Note that this literally means that two products \textit{together with their universal properties} are equal, thus yielding much more than just equations (equivalently, bijections) between the product sets. We can show the existence by using the Structure Identity Principle (\cite[11.6]{Rij22}), which provides an equivalent  characterisation of the identity type of products with their universal property. This requires the choice of a characterisation that is more easily to handle than the original identity type; in the case of Cartesian products equality of types and of projection maps will do.

Finally, uniqueness follows from the upward closedness of truncatedness (\cite[Prop.12.4.3]{Rij22}): If a type is a proposition, it is also a set, or more explicitly, if any two objects of this type are equal, equalities between two objects are equal, too.

The same approaches can be used for a Cartesian product of $n$ sets: We can characterise them by their constructors and their free generatedness, by the induction principle, by the projection maps and their calculation rules, or by the universal property of projections, show that these characterisation imply each other, and that such products of $n$ sets together with one of the characterisations are unique,
via the Structure Identity Principle. 

At the same time, we can construct such products of $n$ sets iteratively from products with less factors.
\begin{prop}   
If $P_n$ and $P_m$ are products of $n$ sets $A_1, \ldots, A_n$ and $m$ sets $B_1, \ldots, B_m$, respectively, then $P_n \times P_m$ is a product of the $n+m$ sets $A_1, \ldots, A_n, B_1, \ldots, B_m$.
\end{prop} 
\begin{proof}
We use the characterisation of products of $n$ sets by the free generatedness of their constructor. That is, we want to show that every element in $P_n \times P_m$ can be constructed from a uniquely determined $(n+m)$-tuple $(a_1, \ldots, a_n, b_1, \ldots, b_m)$, with $a_i : A_i$, $i = 1, \ldots, n$, $b_j : B_j$, $j = 1, \ldots, m$. Of course, the type of the $(n+m)$-tuples $(a_1, \ldots, b_m)$ already is a Cartesian product of $n+m$ sets, a fact disguised by the currying of functions in the case of two factors.

The unique construction of $p : P_n \times P_m$ from $(a_1, \ldots, b_m)$ proceeds by splitting up the $(n+m)$-tuple into the uniquely determined $n$-tuple $(a_1, \ldots, a_n)$ and the $m$-tuple $(b_1, \ldots, b_m)$, then using the characterisation above for the products $P_n$ and $P_m$, and finally using the analogous characterisation for the $2$-factor product $P_n \times P_m$ in Def.~\ref{prod-cons-char}.
\end{proof}

The uniqueness properties of products of $n$ sets together with this proposition show that any two ways of constructing products of more than two sets by iteratively taking the product of exactly two sets yield canonically equal sets. In particular, no pentagon identity as in monoidal categories is needed for the identifications.

\subsection{Natural types} \label{nat-type-ssec}

Cartesian products are rather simple inductive types, with just one constructor and no recursion. The paradigmatic example of a recursively defined type are the natural numbers $\mathbb{N}$, with two constructors: One postulates the existence of $0$ in $\mathbb{N}$, and the other constructor is the recursive successor operation $\mathtt{succ} : \mathbb{N} \rightarrow \mathbb{N}$. The inductive principle guarantees a (unique) function value for each natural number, if we know the value of $0$ and the value of the successor given the function value of any natural number. The calculation rules simply state what the function value of $0$ is and how to determine the function value of the successor from the value of any given natural number. 

Because of the recursive constructor, free generation means more than just every $n : \mathbb{N}$ being either $0$ or the successor of a uniquely determined $m : \mathbb{N}$. If we only assume that condition it would also cover the disjoint union of $\mathbb{N}$ and the integers $\mathbb{Z}$, with the standard successor function on $\mathbb{Z}$. Instead, free generatedness prescribes a unique number of iterations of the successor function to produce a given number $n : \mathbb{N}$ from $0$, and thus leads to the following characterisation:
\begin{Def} \label{nat-cons-def}  
A type $N$ is \textbf{natural} if
\begin{itemize}
\item[(i)] there is an object $n_0 : N$ and a function $s : N \rightarrow N$, and
\item[(ii)] for each $n : N$ there is a unique natural number $i : \mathbb{N}$ such that $n = s^{\circ i}(n_0)$, the function $s$ applied $i$ times on $n_0$.
\end{itemize}
\end{Def}

It is straightforward to show that the natural numbers $\mathbb{N}$ form such a natural type, 
and that natural types satisfy the analogue of the induction principle for natural numbers. This leads to a categorical characterisation identifying a natural type with an initial object (see \cite[5.4]{HoTTBook}):
\begin{prop} \label{nat-ind-def}  
A type $N$ with object $n_0$ and map $s : N \rightarrow N$ is natural if, and only if, for all types $A$ with an object $a_0 : A$ and a map $a_s : A \rightarrow A$ there is a unique map $f : N \rightarrow A$ such that $f(n_0) = a_0$ and $f(s(n)) = a_s(f(n))$, for all $n \in \mathbb{N}$. 
\end{prop}
\begin{proof}
If $N$ is natural and $A$ is a type with $a_0 : A$ and $a_s : A \rightarrow A$, we can define the map $f : N \rightarrow A$ by setting $f(s^{\circ i}(n_0)) = a_s^{\circ i}(a_0)$. Then $f$ certainly satisfies 
$f(n_0) := a_0$ and $f(s(m)) := a_s(f(m))$, and by induction on $i$ these are the only possible values of such a function $f$. Hence $f$ is unique, by function extensionality. Since $\mathbb{N}$ is natural, this implies that there will always be a unique map $g : \mathbb{N} \rightarrow N$ such that $g(0) = n_0$ and $g(\mathtt{succ}(n)) = s(g(n))$ for all natural numbers $n$.

So conversely, we can use the unique map $f : N \rightarrow \mathbb{N}$ such that $f(n_0) = 0$ and $f(s(n)) = \mathtt{succ}(f(n))$ together with $g$ to show Def.~\ref{nat-cons-def}(ii): By uniqueness, $f \circ g = \mathrm{id}_{\mathbb{N}}$ and $g \circ f = \mathrm{id}_N$. Then $n = s^{\circ f(n)}(n_0)$ holds, because $g(m) = s^{\circ m}(n_0)$ for all $m : \mathbb{N}$, by induction on $m$, and we can replace $n$ with $g(f(n))$. The inverseness of $f$ and $g$ also implies that $g$ is injective, hence the uniqueness part of
Def.~\ref{nat-cons-def}(ii).
\end{proof} 

The second part of the argument uses the inductively defined instance of a natural type, to show a property for all natural types. This also ocurred in the proof of Prop.~\ref{cons-ind-prod-Def}, when we discussed Cartesian products.

Starting with the number of iterations of the successor function needed to produce an element of the type when applied on the zero object, gives another characterisation of natural types:
\begin{prop} \label{ext-nat-char}   
A type $N$ with a function $i : N \rightarrow \mathbb{N}$ is natural if, and only if, for all $m : \mathbb{N}$ there exists a unique $n$ such that $i(n) = m$.
\end{prop}
\begin{proof}
If $N$ is natural, the condition Def.~\ref{nat-cons-def}(ii) provides the function $i : N \rightarrow \mathbb{N}$. In the proof of Prop.~\ref{nat-ind-def}, we already constructed an inverse map to $i$, hence the existence and uniqueness.

Conversely, existence and uniqueness mean nothing else than $i$ being invertible. So we can extract $n_0$ and $s : N \rightarrow N$, together with their required properties, from $\mathbb{N}$ using $i$ and its inverse map. 
\end{proof}

The existence of an inverse of a map from $N$ to $\mathbb{N}$ automatically implies that $N$ is a set because $\mathbb{N}$ is a set (see \cite[Ex.12.3.2 \& Prop.12.4.5]{Rij22}). 
So to ensure unique equalities between natural types there is no need to restrict them to sets, as was the case for Cartesian products.

The categorical version of this characterisation identifies a natural type with a terminal object in the category of types with a map to $\mathbb{N}$:
\begin{prop}   
A type $N$ with a function $i : N \rightarrow \mathbb{N}$ is natural if, and only if, for all types $A$ with a function $j : A \rightarrow \mathbb{N}$ there is a unique function $f : A \rightarrow N$ such that $j = i \circ f$.
\end{prop}
\begin{proof}
If $N$ satisfies the previous characterisation, $i$ is invertible, hence $f = i^{-1} \circ j$ is the unique function with $j = i \circ f$, for given type $A$ with $j : A \rightarrow \mathbb{N}$. Conversely, the usual yoga of universal properties shows that the function $f: N \rightarrow \mathbb{N}$ associated to $\mathrm{id} : \mathbb{N} \rightarrow \mathbb{N}$ is inverse to $i$. 
\end{proof}

It may look tautological to characterize natural types using the natural numbers $\mathbb{N}$ --- especially in the characterisation of Prop.~\ref{ext-nat-char}, which simply states that there is an invertible map from the type to $\mathbb{N}$. 
But the only way to describe the recursive generation of objects in a type is by a recursively constructed type. There are systematic recipes to find such types, for example the well-founded trees of the $W$-types (see \cite[5.3]{HoTTBook}). But in this article, we will supply them on an ad-hoc basis, and the simplest one for a natural type is $\mathbb{N}$. For our purposes, it is more important to realize that the resulting characterisations contain proper mathematical substance and have non-trivial consequences, as the example of unary and binary natural numbers in Sec.~\ref{univ-ssec} shows: Constructing a bijection between $\mathbb{N}$ and $\mathrm{Bin}\mathbb{N}$ that respects the zero object and successor function and other algebraic operations on both sets is necessary to use the more efficient calculation algorithms on binary natural numbers for statements on unary natural numbers.   

\section{Universal Properties from Higher Inductive Constructions}

The results of the previous section, as long as they did not use univalence, are valid for any foundations in dependent type theory with inductive constructions. But for more advanced concepts than the natural numbers and Cartesian products (of sets), and especially for algebraic ones, using a central feature of Homotopy Type Theory can pay off immensely: the so-called higher inductive types, for short HITs. 

\subsection{Set quotients}  \label{set-quot-ssec}

We start discussing HITs on one of the basic examples, set quotients. They are used for the fundamental construction of identifying all the elements in the same equivalence class of an equivalence relation on a set, with one element in a new set, the set quotient. The formalisation of this construction, and especially its efficiency, are a difficult issue in many formal proof systems (see for example \cite{Setoid}). 

In HoTT, \textit{relations} are given by maps $R : A \rightarrow A \rightarrow \mathtt{Prop}$ associating pairs of objects of a set $A$ to \textit{propositions}, that are types whose objects, if they exist, are all equal. \textit{Equivalence relations} require that $R$ is \textit{reflexive}, that is,  $R(a,a)$ is a \textit{true} (or inhabited) proposition, and that $R$ is \textit{symmetric} and \textit{transitive}. The \textit{equivalence class} of $a : A$ is the (dependent pair) type of all objects $b : A$ such that $R(a, b)$ is true. 

It is straightforward to characterize a set quotient of an equivalence relation:
\begin{Def} \label{set-quot-def}  
A set $Q$ is a \textbf{set quotient} of an equivalence relation $R : A \rightarrow A \rightarrow \mathtt{Prop}$ on a set $A$ if
\begin{itemize}
\item[(i)] there is \textit{quotient map} $q : A \rightarrow Q$, 
\item[(ii)] the map $q$ is surjective, that is, for all $\overline{a} : Q$ there (merely) exists $a : A$ such that $q(a) = \overline{a}$, and
\item[(iii)] for all $a, b : A$, $q(a) = q(b)$ if, and only if $R(a, b)$ is true. 
\end{itemize}
\end{Def}

Let us first clarify once again the meaning of ``merely" in Def.~\ref{set-quot-def}. It weakens the requirement of having a proper function $s : Q \rightarrow A$ such that $q(s(\overline{a})) = \overline{a}$ for all $\overline{a} : Q$ (in other words, a \textit{section} of $q$). Instead, it only postulates the existence of an $a$ such that $q(a) = \overline{a}$, without further specifying how to obtain this $a$. This is enough to state that the map $q$ is \textit{surjective}. As noted in \ref{prop-set-ssec}, postulating the ``mere existence" instead of outright existence is mainly done to avoid a choice, which would not matter anyway, and not because choices cannot be made. We just need to be aware that we only can use an explicit preimage if we want to show a proposition or construct a uniquely determined object.

In some special cases of equivalence relations it is easy to identify such set quotients directly.
\begin{exm}  
On a set $A$, the relation given by $R(a, b) := (a =_A b)$ is an equivalence relation because there is the reflexive path for any object $a : A$ and equality paths have inverses and can be concatenated. Obviously, $A$ is a set quotient by this equivalence relation. 
\end{exm} 

\begin{exm}    
If $q: A \rightarrow B$ is a map from a type $A$ to a set $B$ then the image $Q$ of $q$ is a set quotient of $A$ by the equivalence relation $R$ given by $R(a, a^\prime) := q(a) =_B q(a^\prime)$. Note that $b : B$ is in the image of $q$ if there \textit{merely} exists $a : A$ with $q(a) = b$, which allows to deduce Def.~\ref{set-quot-def}(ii).
\end{exm}

\begin{exm} \label{surj-equiv-exm}  
On a set $A$, the relation given by $R(a,b) := \mathbbm{1}$, the unit type with only object $\ast : \mathbbm{1}$, is obviously an equivalence relation, because $R(a,b)$ holds for all $a, b : A$. A set quotient of $A$ by $R$ is simply $\mathbbm{1}$, with quotient map $a \mapsto \ast$, if $A$ is (merely) inhabited by an object $a : A$.
\end{exm}

To construct a set quotient as characterized in Def.~\ref{set-quot-def} for arbitrary equivalence relations it is not enough to introduce a constructor $q : A \rightarrow Q$ as in a standard inductive construction since free generatedness would imply that $q(a) \neq q(b)$ as soon as $a \neq b$, even if $R(a, b)$ is true. A \textit{higher inductive type} also allows constructors for equalities between its objects, and for the set quotient we need a constructor which produces for each $a, b : A$ with true $R(a, b)$ an equality $q(a) = q(b)$. 

Since we still want the constructors of a HIT to satisfy free generation, the new equalities $q(a) = q(b)$ are at least not a priori equal to any other equality in the identity type $q(a) = q(b)$. For example, $R(a, a)$ is always true, hence we obtain a second equality besides the reflexive equality $\mathtt{refl}_a$ produced by Martin-L\"of's inductive construction of equality. If we do not want these equalities to differ, we must introduce a further higher constructor, which for any objects $\overline{a}, \overline{b} : Q$ and equalities $r, s : \overline{a} = \overline{b}$ produces an equality $r = s$. In other words, this constructor makes the set quotient into a set, by \textit{truncating} all the higher equalities. Therefore it is called the \textit{set truncation}.

The recursion principle of a higher inductive type delivers for any type equipped with the same structure a function from the HIT to the type, also mapping the constructors to that structure. If we denote the HIT constructed above from a set $A$ and an equivalence relation $R$ on $A$ by $A/R$, with quotient map $q_R : A \rightarrow A/R$, this means that for each set $B$ with a map $f : A \rightarrow B$ such that $f(a) = f(b)$ whenever $R(a, b)$ is true, there is a map $g : A/R \rightarrow B$ such that $f = g \circ q_R$ and if $R(a, b)$ is true the equality $f(a) = f(b)$ is obtained by applying $g$ on the equality $q_R(a) = q_R(b)$. We need that $B$ is a set because of the set truncation constructor, and thus the last property of $g$ above is automatically satisfied.

The induction principle is the analogue of the recursion principle for dependent types $B$ associating a set $B(\overline{a})$ to every $\overline{a} : A/R$. If we have a dependent function $f$ mapping $a : A$ to an object of $B(q(a))$ such that $f(a)$ is transported to $f(a^\prime) : B(q_R(a^\prime))$ along the equality $q_R(a) = q_R(a^\prime)$ holding if $R(a, a^\prime)$ is true then induction delivers a function $g$ mapping $\overline{a} : A/R$ to an object of $B(\overline{a})$ such that $f = g \circ q_R$. 

\begin{prop} \label{HIT-set-quot-prop}  
For a set $A$ and an equivalence relation $R$ on $A$, the HIT $A/R$ with quotient map $q_R : A \rightarrow A/R$ is a set quotient.
\end{prop}
\begin{proof}
We can apply induction to show surjectivity of $q_R$: For $a : A$ we obviously (merely) have $a : A$ such that $q_R(a) = q_R(a)$, and the transport condition is satisfied because the dependent types are propositions.

The proof of Def.~\ref{set-quot-def}(iii) is a little bit more involved; we follow \cite[Lem.10.1.8]{HoTTBook}. The first step to characterise equalities in $A/R$ is to descend the relation $R$ on $A$ to a relation $\widetilde{R} : A/R \rightarrow A/R \rightarrow \mathtt{Prop}$ on $A/R$. We can do this by double induction applied on the function $R$; if $R(a,a^\prime)$ and $R(b, b^\prime)$ hold then symmetry and transitivity of $R$ show that $R(a,b)$ and $R(a^\prime, b^\prime)$ are equivalent propositions and therefore equal, by univalence.

Similarly, we can show that for all $\overline{a}, \overline{b} : A$, the identity type $\overline{a} = \overline{b}$ is equivalent to $\widetilde{R}(\overline{a}, \overline{b})$. One direction follows by path induction and the reflexivity of $R$, hence $\widetilde{R}$. The other direction uses that all the types involved are propositions, hence by surjectivity of $q_R$ we only have to produce a map from $\widetilde{R}(q_R(a), q_R(b)) = R(a, b)$ to $q_R(a) = q_R(b)$ for all $a, b : A$, which is simply the higher constructor of the HIT $A/R$. 
\end{proof}

If we add uniqueness the recursion principle translates to a universal property of set quotients equivalent to Def.~\ref{set-quot-def}.
\begin{prop} \label{set-quot-univ-prop}  
A set $Q$ is a set quotient of a set $A$ by an equivalence relation $R$ on $A$ and with quotient map $q : A \rightarrow Q$ where $R(a, b)$ implies $q(a) = q(b)$ if, and only if for all sets $B$ with a map $f : A \rightarrow B$ such that $f(a) = f(b)$ whenever $R(a, b)$ is true, there is a unique map $g : Q \rightarrow B$ such that $f = g \circ q$.
\end{prop}
\begin{proof}
We will show first that for a set quotient $Q$ with quotient map $q : A \rightarrow Q$ and a set $B$ with $f : A \rightarrow B$ such that $f(a) = f(b)$ if $R(a, b)$ is true, there is at most one function $g : Q \rightarrow B$ such that $g(q(a)) = f(a)$ for all $a : A$. If we have two such functions $g, g^\prime$ then for a given $\overline{a} : Q$, the identity type $g(\overline{a}) = g^\prime(\overline{a})$ will be a proposition, since $B$ is a set. Surjectivity gives us $a : A$ such that $q(a) = \overline{a}$, hence 
\[ g(\overline{a}) = g(q(a)) = g^\prime(q(a)) = g^\prime(\overline{a}), \]
as requested (by function extensionality).

Next, we construct such a function $g : Q \rightarrow B$: For each $\overline{a} : Q$ we need $b : B$ such that there (merely) exists an $a : A$ with $q(a) = \overline{a}$ and $f(a) = b$. Such a $b$ is unique: If $b^\prime$ is another such object then we can show $b = b^\prime$ by using propositional truncation, as $B$ is a set. We can assume that we have $a : A$ with $q(a) = \overline{a}$, $f(a) = b$ and $a^\prime : A$ with $q(a^\prime) = \overline{a}$, $f(a^\prime) = b^\prime$. Def.~\ref{set-quot-def}(iii) implies $R(a, a^\prime)$ is true, hence also $b = f(a) = f(a^\prime) = b^\prime$. Applying once again propositional truncation we obtain an $a : A$ with $q(a) = \overline{a}$ and can choose $b := f(a)$.

Conversely, we can apply the recursion principle for $Q$ on the HIT $A/R$ with quotient map $q_R$, as constructed above, because by Prop.~\ref{HIT-set-quot-prop} condition \ref{set-quot-def}(iii) holds for $A/R$. We thus obtain a map $g : Q \rightarrow A/R$ with $q_R = g \circ q$. The first part of the proof and Prop.~\ref{HIT-set-quot-prop} yield a map $h : A/R \rightarrow Q$ such that $q = h \circ q_R$. As usual, uniqueness also implies that $h \circ g = \mathrm{id}_Q$ and $g \circ h = \mathrm{id}_{A/R}$. Consequently, $g$ is injective and surjective, hence $q$ satisfies Def.~\ref{set-quot-def}(ii) and (iii).
\end{proof}

\begin{prop} \label{set-quot-eq-prop}  
Two set quotients of a set $A$ by an equivalence relation $R$ are equal, with quotient maps equal after transport. 
\end{prop}
\begin{proof}
We apply univalence to the equivalence between the quotients obtained by the usual yoga on the universal property in Prop.~\ref{set-quot-univ-prop}. The construction of the equivalence in the proof of Prop.~\ref{set-quot-univ-prop}
shows equality of the quotient maps after transport.
\end{proof}

Note that the uniqueness statement above once again implies equality of the quotient sets together with (transported) equality of the structure of a quotient set  --- in the language of \cite[11.6]{Rij22}, equality of quotient sets and their quotient maps is a dependent identity system for quotient sets together with their structure. 
It is actually more efficient to apply the Structure Identity Principle on the structure of a quotient set given in Def.~\ref{set-quot-def}.

\begin{exm} \label{well-def-exm} 
The universal property of set quotients justifies the widespread construction of \textit{well-defined} functions: Given a surjective map $q: A \rightarrow B$ from a type $A$ to a set $B$ we can construct a unique map $g: B \rightarrow C$ to a set $C$ if we have a map $f: A \rightarrow C$ such that $q(a) = q(a^\prime)$ implies $f(a) = f(a^\prime)$. This is just the universal property applied to $B$ interpreted as the set quotient of $A$ by the equivalence relation defined in Ex.~\ref{surj-equiv-exm}.
\end{exm}

We can also extract generators from elements of the quotient set and thus interpret it as a terminal object:
\begin{prop}  \label{sec-quot-char-prop}
A set $Q$ is a set quotient of a set $A$ by an equivalence relation $R$ on $A$ and with quotient map $q : A \rightarrow Q$ where $R(a, b)$ implies $q(a) = q(b)$ if, and only if there merely exists a section $s : Q \rightarrow A$, that is, for all $a : A$ there is a $q$ such that $R(s(q), a)$ holds, and $R(s(q_1), s(q_2))$ always implies $q_1 = q_2$. 
\end{prop}
\begin{proof}
If $Q$ is a set quotient in the sense of Def.~\ref{set-quot-def}, the existence of a section is actually equivalent to the Axiom of Choice $\mathrm{AC}_0$ as discussed in \ref{prop-set-ssec}, applied on Def.~\ref{set-quot-def}(ii). 

Conversely, we can apply induction on propositional truncation, because Prop.~\ref{set-quot-eq-prop} states that set quotients in the sense of Def.~\ref{set-quot-def} are all equal.
\end{proof}

It is awkward but not surprising that we need the Axiom of Choice to produce a section because we indeed need to choose it among the possibly many existing. On the other hand, incorporating the datum of a particular section into the characterisation of a quotient set would destroy its uniqueness.

\begin{prop}  
A set $Q$ is a set quotient of a set $A$ by an equivalence relation $R$ on $A$ and with quotient map $q : A \rightarrow Q$ where $R(a, b)$ implies $q(a) = q(b)$ if, and only if there merely exists a map $s : Q \rightarrow A$ such that for each map $s_P : P \rightarrow A$ from a set $P$ there exists a unique map $f : P \rightarrow Q$ with 
$R(s_P(p), s(f(p))$ for all $p : P$.
\end{prop}
\begin{proof}
We just need to show that the characterisations in this proposition and Prop.~\ref{sec-quot-char-prop} imply each other. To show that $s$ is a section, consider suitable maps $s_\mathbbm{1} : \mathbbm{1} \rightarrow A$.  
\end{proof}

The higher inductive construction of a set quotient $A/R$ of a type $A$ also works when $R : A \rightarrow A \rightarrow \mathtt{Prop}$ is not an equivalence relation, but an arbitrary relation:
\begin{itemize}
\item Objects are generated by the quotient map $q : A \rightarrow A/R$, and
\item for all $a, b : A$, we get an equality $q(a) = q(b)$ if $R(a, b)$ is true.
\end{itemize}

This set quotient, with the same quotient map, is actually equal to the set quotient of $A$ by the \textit{equivalence relation} $\overline{R} : A \rightarrow A \rightarrow \mathtt{Prop}$ \textit{generated by the relation} $R$, which can be constructed inductively:
\begin{itemize}
\item $\overline{R}(a, b)$ holds if $R(a, b)$ is true, for all $a, b : A$;
\item $\overline{R}(a, a)$ holds, for all $a : A$;
\item if $\overline{R}(a, b)$ is true then also $\overline{R}(b, a)$, for all $a, b : A$; 
\item $\overline{R}(a, b)$ and $\overline{R}(b, c)$ imply $\overline{R}(a, c)$, for all $a, b, c : A$.
\end{itemize}

However, this does not \textit{freely} generate the relation $\overline{R}(a,b)$. For example, $\overline{R}(a,b)$ can also be produced by transitivity from $\overline{R}(a,a)$ and $\overline{R}(a,b)$. Thus, relations $\overline{R}(a,b)$ are characterized by the \textit{mere existence} of an inductive construction following the rules above.

The reason why $A/R$ also includes equalities induced by reflexive, symmetric and transitive relations, is that a higher inductive construction \textit{freely} generates the equalities that it introduces, but the inductive construction of equalities prescribed by the Martin-L\"of definition is still in place. Thus if $q : A \rightarrow A/R$ is the quotient map for $R$, the reflexive identity gives us $q(a) = q(a)$, reversing equalities yields $q(b) = q(a)$ from $q(a) = q(b)$, and transitivity is simply concatenating path equalities.

\begin{prop}  
Let $R : A \rightarrow A \rightarrow \mathtt{Prop}$ be a relation on a set $A$. Then the set quotient of $A$ by $R$ is equal to the set quotient of $A$ by $\overline{R}$, the equivalence relation generated by $R$, and the quotient maps are equal after transport. 
\end{prop}
\begin{proof}
By Prop.~\ref{HIT-set-quot-prop} and Prop.~\ref{set-quot-eq-prop}, it is enough to show that $A/R$ is a set quotient of $A$ by the equivalence relation $\overline{R}$. 

Surjectivity follows by (higher) induction on the objects of $A/R$, since surjectivity is defined by a proposition.

Def.~\ref{set-quot-def}(iii) follows for $A/R$ in the same way as it follows for $A/\overline{R}$: We descend the equivalence relation $\overline{R}$ to the relation $\widetilde{R} : A/R \rightarrow A/R \rightarrow \mathtt{Prop}$ and then construct an equivalence $(\overline{a} = \overline{b}) \simeq \widetilde{R}(\overline{a}, \overline{b})$ for all $\overline{a}, \overline{b} : A/R$.
\end{proof}

\begin{rem}
Note that despite this proposition it is not possible to relax the characterisation in Def.~\ref{set-quot-def}(iii) to $q(a) = q(b)$ if $R(a, b)$ is true. The uniqueness of quotients would get lost; for example, the trivial map to $\mathbbm{1}$ would always be a quotient map.
\end{rem}

\begin{rem}   
It is possible to define the set quotient $A \sslash R$ of a type $A$ by an equivalence relation as the set of equivalence classes, described as predicates $A \rightarrow \mathtt{Prop}$. It is equivalent to the set quotients above (see \cite[Thm.6.10.6]{HoTTBook}, but it requires a lift of the universe unless we allow \textit{propositional resizing}.
\end{rem}

\subsection{Subsets}

Before we continue to use HITs for constructing types with algebraic structure we need to discuss \textit{subsets} of sets, as the underlying type of algebraic subobjects.

Usually subsets are described as \textit{predicates} $A \rightarrow \mathtt{Prop}$, reflecting the fact that (a propositionally resized) $\mathtt{Prop}$ is a classifier in the category of sets (see \cite[Thm.10.1.12]{HoTTBook}). But alternatively, a subset can be described as a set $B$ together with an injective map $\iota : B \rightarrow A$, which corresponds to the definition of \textit{subobjects} in category theory. This point of view pays off when algebraic structures get involved, as I will explain below. 

Before showing that both definitions are coinciding, let us remove possible concerns about the second one: Of course, there are different ways to embed a set into another set without changing the image. However, two injections $\iota : B \rightarrow A$, $\iota^\prime : B^\prime \rightarrow A$ will have the same image if, and only if there is a bijection $f : B \rightarrow B^\prime$ such that $\iota = f \circ \iota^\prime$, and $f$ is uniquely determined by $\iota$ and $\iota^\prime$. The bijection therefore induces a unique equality between the pairs $(B, \iota)$ and $(B^\prime, \iota^\prime)$, by univalence.

\begin{prop} \label{set-equiv-def-prop}   
Let $A$ be a set. Then the type of predicates $A \rightarrow \mathtt{Prop}$ and the type of pairs of sets $B$ and injections $\iota : B \rightarrow A$ are equivalent. 
\end{prop}
\begin{proof}
Given a predicate $B : A \rightarrow \mathtt{Prop}$ all the pairs of $a : A$ and inhabitants (or proofs) of $B(a)$ form a set which we also call $B$, and the projection to $a$ is an injection, because $B(a)$ is a proposition and thus has at most one inhabitant, up to equality. Conversely, an injection $\iota : B \rightarrow A$ associates to every $a$ the type of pairs of $b : B$ and equalities $\iota(b) = a$, and injectivity guarantees that these types are propositions. Both constructions are easily shown to be inverse to each other. 
\end{proof}

By univalence, this implies that the two types describing subsets are actually equal. The equivalence inducing the equality actually shows that both ways of constructing a subset a priori require the same amount of work: To define a predicate $B : A \rightarrow \mathtt{Prop}$ we need to construct a type for each $a : A$ and check that all these types are propositions. For an injective map $\iota : B \hookrightarrow A$, we need to introduce a type $B$, associate an $a : A$ to each $b : B$, and check that the type of pairs consisting of an element $b : B$ and an equality $p : \iota(b) = a$ is a proposition for each $a : A$. Note that the proof of Prop.~\ref{set-equiv-def-prop} also shows how injectivity automatically implies that $B$ is a set, so we do not need to verify that fact. 

In specific cases one way of defining a subset may look more natural, considering previously provided information. But the equivalence of both ways
always allows us to denote a subset $B$ of a set $A$ by $B \subseteq A$, and for $a : A$ the element relation $a \in B$ checks whether the proposition $B(a)$ is true or equivalently, whether $a$ is in the image of $\iota$.

\subsection{Monoids}  \label{monoid-ssec}

As a first algebraic structure defined on sets $M$ we will discuss \textit{monoids}. They are sets provided with
\begin{itemize}
\item a composition map $\cdot : M \rightarrow M \rightarrow M$;
\item a neutral element $e_M : M$;
\item associativity: for all $m_1, m_2, m_3 : M$, we have 
\[ (m_1 \cdot m_2) \cdot m_3 = m_1 \cdot (m_2 \cdot m_3); \]
\item neutrality: for all $m : M$, we have $m \cdot e_M = e_M \cdot m = m$;
\end{itemize}

A monoid is \textit{commutative} if furthermore,
\begin{itemize}
\item $m \cdot n = n \cdot m$, for all $m, n : M$.
\end{itemize}
The most basic example of a (commutative) monoid is the unit monoid $\mathbbm{1}$, whose only object will be identified with the neutral element.

A homomorphism between monoids $M, N$ is a map $f : M \rightarrow N$ of the underlying sets satisfying
\begin{itemize}
\item $f(e_M) = e_N$, and
\item $f(m_1 \cdot m_2) = f(m_1) \cdot f(m_2)$ for all $m_1, m_2 : M$.
\end{itemize}

There always is the trivial homomorphism $u : M \rightarrow N$ between two monoids, mapping each element of $M$ to the neutral element $e_N : N$.

It is straightforward to show that the identity map is a homomorphism, that homomorphisms $f : N \rightarrow M$, $g : M \rightarrow L$ between monoids $N, M, L$ can be composed to a homomorphism $g \circ f : N \rightarrow L$, and that the composition of (identity) homomorphisms satisfies the usual laws: This makes monoids into a \textit{precategory}, in the sense discussed in \ref{cat-ssec}. 
The additional property turning it into a category, namely the identification of isomorphisms and equalities, is most efficiently shown by constructing it as an iterated \textit{concrete category} over sets, \textit{magmas} (sets with a binary operation) and \textit{semigroups} (magmas whose binary operations are associative). In each step we can derive an equality of isomorphic objects of the concrete category from the induced isomorphisms of the underlying objects in the base category, satisfying the additional properties of a homomorphism in the concrete category, and we start the process by producing an equality from isomorphisms between sets, that is an equivalence of the underlying types, using univalence.

Finally, homomorphisms of monoids are uniquely determined by the map of their underlying sets. In categorical terms: The forgetful functor from monoids to sets is faithful.

Rearranging the material in \cite[6.11]{HoTTBook} we now discuss \textit{free monoids} on sets. Their universal property can be used to characterise such free monoids.

\begin{Def}  
A monoid $F$ with a map $h : A \rightarrow F$ is a \textbf{free monoid} on a set $A$ if for all monoids $M$ and maps $h^\prime : A \rightarrow M$ there is a unique homomorphism $f : F \rightarrow M$ such that $h^\prime = f \circ h$.
\end{Def}

When interpreting the universal property as the higher recursive principle of a higher inductive type $F(A)$ on $A$, the constructors of the HIT will consist of
\begin{itemize}
\item a function $h : A \rightarrow F(A)$;
\item a function $m : F(A) \rightarrow F(A) \rightarrow F(A)$;
\item an element $e : F(A)$;
\item for each $x, y, z : F(A)$, an equality $m(x,m(y, z)) = m(m(x, y), z)$;
\item for each $x : F(A)$, equalities $m(x, e) = x$ and $m(e, x) = x$;
\item the $0$-truncation constructor: for any $x, y : F(A)$ and $p, p^\prime : x = y$, we have $p = p^\prime$.
\end{itemize} 

However, in the case of a free monoid, this higher inductive construction is not necessary (in contrast to free groups, see~\ref{groups-ssec}): There is a canonical way to present elements of the free monoid over a set $A$ by \textit{words}, with objects of $A$ as \textit{letters}, and ``canonical" means that equality between two words can be computably decided. Words can be encoded in the type $\mathrm{List}(A)$ of lists with entries in $A$: Objects $a : A$ are embedded into $\mathrm{List}(A)$ as one-letter words $h(a) := [a]$, two words are associatively multiplied by concatenating them, and the empty list $[]$ is a neutral element. It also easy to show that $\mathrm{List}(A)$ is a set if $A$ is a set.
\begin{prop}[{\cite[Lem.6.11.5]{HoTTBook}}] \label{list-free-monoid-prop}
$\mathrm{List}(A)$ is a free monoid over the set $A$.
\end{prop}
\begin{proof}
If $h^\prime : A \rightarrow M$ is a map from $A$ to a monoid $M$, we define the homomorphism $f : \mathrm{List}(A) \rightarrow M$ inductively by mapping $[]$ to $e_G$ and $[a, b, \ldots]$ to $m(h^\prime(a), f([b, \ldots]))$. The homomorphism laws are straightforward, and we can show by induction on lists that $h^\prime \mapsto f$ is right-inverse to $f \mapsto f \circ h$.
\end{proof}

We can characterize products of two monoids using products of the underlying sets.
\begin{Def}    
A monoid $P$ is called a product of two monoids $M$ and $N$ if the underlying set is a product of the underlying sets of $M$ and $N$ and the projection maps are underlying maps of monoid homomorphisms $p_M : P \rightarrow M$ and  $p_N : P \rightarrow N$.
\end{Def}
Then it is easy to check that the product set $M \times N$ with identity $(e_M, e_N)$, where $e_M$ and  $e_N$ are the identities of $M$ and $N$, and multiplication $(m_1, n_1) \cdot (m_2, n_2) = (m_1 \cdot m_2, n_1 \cdot n_2)$ for all $m_1, m_2 : M$, $n_1, n_2 : N$, is a product monoid. 

The adjointness of the construction of free monoids to the forgetful functor allows us to show the universal property of products in the category of monoids:
\begin{prop}  
A monoid $P$ is a product of two monoids $M$ and $N$ if, and only if there are two projection homomorphisms $p_M : P \rightarrow M$ and $p_N : P \rightarrow N$ and for all monoids $Q$ with monoid homomorphisms $q_M : Q \rightarrow M$ and  $q_N : Q \rightarrow N$ there exists a unique monoid homomorphism $f : Q \rightarrow P$ such that $q_M = p_M \circ f$ and $q_N = p_N \circ f$.
\end{prop}
\begin{proof}
If $P$ is a product of $M$ and $N$ and $Q$ is a monoid with homomorphisms $q_M : Q \rightarrow M$ and  $q_N : Q \rightarrow N$, then the map $f : Q \rightarrow P$ induced by the universal property of set products satisfies the properties of a monoid homomorphism:
\begin{itemize}
\item The identity $e_Q$ of $Q$ is mapped to the identities $e_M$, $e_N$ of $M$ and $N$ by $q_N$ and $q_M$, as is the identity $e_P$ of $P$ by $p_M$ and $p_N$, hence $e_Q$ is mapped to $e_P$ by $f$, using Prop.~\ref{prod-proj-calc-char}(iii).
\item For all $m_1, m_2 : Q$, the product $f(m_1) \cdot f(m_2)$ in $P$ is mapped to $q_M(m_1) \cdot q_M(m_2)$ in $M$ by $p_M$ and to $q_N(m_1) \cdot q_N(m_2)$ in $N$ by $p_N$, hence $f(m_1 \cdot m_2) = f(m_1) \cdot f(m_2)$, using again Prop.~\ref{prod-proj-calc-char}(iii).
\end{itemize}

Conversely, the underlying set of $P$ must be a product of the underlying sets of $M$ and $N$: If $Q$ is a \textit{set} with maps $q_M : Q \rightarrow M$, $q_N : Q \rightarrow N$ to (the underlying sets of) $M$ and $N$, the maps $q_M$ and $q_N$ induce monoid homomorphisms $q_M^\prime : F(Q) \rightarrow M$ and $q_N^\prime : F(Q) \rightarrow N$ from the free monoid $F(Q)$ on $Q$ to $M$ and $N$. The homomorphism $F(Q) \rightarrow P$ induced by the universal property, composed with the embedding $Q \rightarrow F(Q)$, shows the existence of a map $Q \rightarrow P$. The uniqueness follows from the uniqueness statements for free monoids and the universal property.
\end{proof}

A \textit{submonoid} of a monoid $M$ is a subobject of $M$ in the category of monoids, that is, a monoid $N$ together with a \textit{monomorphism} $\iota: N \rightarrow M$. 

\begin{Def} \label{monomorphism-def}  
A homomorphism $f : N \rightarrow M$ of monoids is called a monomorphism if for all monoid homomorphisms $g_1, g_2 : L \rightarrow N$ we have the implication 
\[ f \circ g_1 = f \circ g_2\ \Rightarrow\ g_1 = g_2. \] 
\end{Def}

Since the forgetful functor from monoids to sets is faithful, a monoid homomorphism will be a monomorphism if the underlying map of sets is a monomorphism, that is, if it is injective. 
The converse is also true:
\begin{prop}   \label{mono-mon-inj-prop}
The underlying map of sets of a monomorphism $f : N \rightarrow M$ of monoids $N, M$ is injective.
\end{prop}
\begin{proof}
Assume that $f(n_1) = f(n_2)$ for $n_1, n_2 : N$, and consider the homomorphisms $g_1, g_2 : F(\mathbbm{1}) \rightarrow N$ from the free monoid $F(\mathbbm{1})$ generated by one object $\ast$ to $N$ uniquely determined by $g_1(\ast) = n_1$ and $g_2(\ast) = n_2$, respectively. Then $f(g_1(\ast)) = f(g_2(\ast))$ shows that $f \circ g_1 = f \circ g_2$, hence $g_1 = g_2$, so $n_1 = g_1(\ast) = g_2(\ast) = n_2$.
\end{proof}

Alternatively, we can show that the underlying map of $f$ satisfies the universal property of a monomorphism in the category of sets, by factorizing maps $g_1, g_2 : A \rightarrow N$ from a set $A$ through the free monoid $\mathbb{F}(A)$.

Since a subobject of a set is a subset, we can characterize submonoids as subsets of a monoid that inherit the monoid structure:
\begin{prop}   \label{submon-subset-prop}  
A subset $N \subseteq M$ of a monoid $M$ is a \textbf{submonoid} of $M$ if
\begin{itemize}
\item[(i)] $e_M \in N$;
\item[(ii)] for all $n_1, n_2 \in N$, we have $n_1 \cdot n_2 \in N$. \hfill $\Box$
\end{itemize}
\end{prop}

\begin{exm}
The \textit{image} $\mathrm{im}(f)$ of a homomorphism $f : M \rightarrow N$ of monoids $M, N$ is the submonoid of elements in $N$ that have a preimage in $M$. It is obviously \textit{minimal} among all submonoids $L \subseteq N$ such that $f(M) \subseteq L$.

The submonoid $\langle L \rangle$ of a monoid $N$ \textit{generated by elements} of a subset $L \subseteq N$ is the image of the homomorphism $f_L : \mathbb{F}(L) \rightarrow N$ from the free monoid on $L$ to $N$, induced by the inclusion map $L \hookrightarrow N$ of the subset $L$ into $N$. It is also the minimal submonoid of $N$ containing the elements of $L$.
\end{exm}

\begin{exm}  
To each homomorphism $f : M \rightarrow N$ of monoids $M, N$ we can associate the \textit{kernel pair} $\ker f$, the submonoid of $M \times M$ consisting of all pairs $(m_1, m_2) : M \times M$ with $f(m_1) = f(m_2)$, that is, the elements of the fibered product $M \times_f M$. Obviously, $(1_M, 1_M) \in \ker f$, and also $f(m_1) = f(m_2)$ and $f(m_1^\prime) = f(m_2^\prime)$ imply 
\[ f(m_1 \cdot m_1^\prime) = f(m_1) \cdot f(m_1^\prime) = f(m_1) \cdot f(m_1^\prime) = f(m_2) \cdot f(m_2^\prime) = f(m_2 \cdot m_2^\prime). \]
\end{exm}

We characterize a quotient of a monoid by an equivalence relation using the quotient set of the underlying set:
\begin{Def}  \label{quot-monoid-def} 
A monoid $Q$ with a monoid homomorphism $\pi : M \rightarrow Q$ is called a quotient monoid of the monoid $M$ by an equivalence relation $R$ if the underlying set of $Q$ is a quotient set of the underlying set of $M$ by $R$ and the underlying map of $\pi$ is the projection map.
\end{Def}

Leading to a quotient monoid of a monoid $M$ has consequences for the equivalence relation $R$ on the underlying set of $M$:
\begin{Def}  
An equivalence relation $R$ on the underlying set of a monoid $M$ is called a congruence if for all $m_1, m_1^\prime, m_2, m_2^\prime : M$ with $m_1 \equiv_R m_1^\prime$ and $m_2 \equiv_R m_2^\prime$, 
\[ m_1 \cdot m_2 \equiv_R m_1^\prime \cdot m_2^\prime. \]
\end{Def} 

\begin{exm}   
If $f : M \rightarrow N$ is a homomorphism of monoids $M, N$, then the kernel pair $\ker f$ describes a congruence $R$ on the underlying set of $M$, by setting $m_1 \equiv_R m_2$ if and only if $(m_1, m_2) \in \ker f$ (that is, $f(m_1) = f(m_2)$). The relation $R$ is obviously reflexive, symmetric and transitive, and it satisfies the congruence condition because $f$ is a monoid homomorphism.
\end{exm}

\begin{prop}    \label{monoid-cong-quot-prop}
A monoid $Q$ with a monoid homomorphism $\pi : M \rightarrow Q$ is a quotient monoid of the monoid $M$ by an equivalence relation $R$ if, and only if, the underlying set of $Q$ is a quotient set of the underlying set of $M$ with projection map underlying $\pi$ and $R$ is a congruence on the underlying set of $M$.
\end{prop}
\begin{proof}
If $Q$ is a quotient monoid of $M$ by $R$, the projection map $\pi$ preserves the monoid structure of $M$ and $Q$ and the equivalence relation $R$. Therefore, if $m_1 \equiv_R m_1^\prime$ and $m_2 \equiv_R m_2^\prime$ for $m_1, m_1^\prime, m_2, m_2^\prime : M$ we have $\pi(m_1 \cdot m_2) = \pi(m_1^\prime \cdot m_2^\prime)$ and consequently
\[ m_1 \cdot m_2 \equiv_R m_1^\prime \cdot m_2^\prime. \]
Conversely, $R$ being a congruence implies that the quotient map underlying $\pi$ has the structure of a monoid homomorphism. 
\end{proof}

Note that $R$ being a congruence also provides the quotient set of the underlying set of $M$ by $R$ with the structure of a monoid, leading to the construction of a quotient monoid of $M$ by $R$.

We can also characterize quotient monoids by a universal property in the category of monoids if we add a characterisation of equality in the quotient by the equivalence relation:
\begin{thm}  \label{univ-monoid-quot-thm}
A monoid $Q$ with a monoid homomorphism $\pi : M \rightarrow Q$ is a quotient monoid of a monoid $M$ by an equivalence relation $R$ if, and only if
\begin{itemize} 
\item[(i)] for all $m_1, m_2 : M$, the relation $m_1 \equiv_R m_2$ is equivalent to $\pi(m_1) = \pi(m_2)$, and
\item[(ii)] for all homomorphisms $f : M \rightarrow L$ from $M$ to a monoid $L$ such that $f(m_1) = f(m_2)$ if $m_1 \equiv_R m_2$ for all $m_1, m_2 : M$, there is a unique homomorphism $g : Q \rightarrow L$ such that $f = g \circ q$.
\end{itemize}
\end{thm}
\begin{proof}
First, let $Q$ be a quotient monoid of $M$ by the equivalence relation $R$, with projection homomorphism $\pi : M \rightarrow Q$, and let $f : M \rightarrow L$ be a monoid homomorphism such that $m_1 \equiv_R m_2$ implies $f(m_1) = f(m_2)$ for all $m_1, m_2 : M$. Then the map from the underlying set of $Q$ to the underlying set of $L$ given by the universal property of quotient sets factorizes the underlying map of $f$ by the underlying map of $\pi$, and is unique. It also has the structure of a monoid homomorphism. Furthermore, the underlying map of $\pi$ and hence $\pi$ satisfy (i).

Condition (i) together with the quotient map being a monoid homomorphism implies that $R$ is a congruence. Hence we can construct a quotient monoid as above, and we have shown that it satisfies the universal property. So if $Q$ with projection homomorphism $p : M \rightarrow Q$ is another monoid satisfying the universal property, then the standard arguments show that they are isomorphic, hence the projection map to the underlying set of $Q$ is surjective and $m_1 \equiv_R m_2$ is equivalent to $p(m_1) = p(m_2)$ for all $m_1, m_2 : M$. Therefore, the underlying set of $Q$ together with the underlying map of $p$ is a quotient set of the underlying set of $M$ by $R$, and $Q$ is a quotient monoid of $M$ by $R$.
\end{proof}

Note that $M$ with the identity homomorphism $M \rightarrow M$ satisfies condition (ii).

If $R$ is a congruence, the universal property codified in this definition is the higher recursion principle of the higher inductive type $M/R$ constructed by
\begin{itemize}
\item a quotient function $q : M \rightarrow M/R$;
\item a multiplication $m : M/R \rightarrow M/R \rightarrow M/R$;
\item a neutral element $e : M/R$;
\item associativity: for each $x, y, z : M/R$, an equality 
\[ m(x,m(y, z)) = m(m(x, y), z); \]
\item neutrality : for each $x : M/R$, equalities $m(x, e) = x$ and $m(e, x) = x$;
\item neutral invariance: $q(e_M) = e$;
\item multiplicative compatibility: for all $x, y : M$, an equality 
\[ m(x \cdot y) = m(q(x), q(y)); \] 
\item quotient identifications : for all $x, y : M$, we have $q(x) = q(y)$ if $x \equiv_R y$;  
\item the $0$-truncation constructor: for any $x, y : M/R$ and $p, p^\prime : x = y$, we have $p = p^\prime$.
\end{itemize}

This construction automatically produces a quotient monoid of $M$ by $R$ and thus seems a perfect starting point to discuss quotients of monoids. However, in ordinary Homotopy Type Theory, HITs are still depending on axioms, for example postulating the existence of truncations and type quotients. In those theories and their formalisations it is more efficient to construct general HITs from the few axiomatically given ones, as we will do from now on. In Cubical Type Theory, HITs are constructive and can be introduced directly.

As for set quotients, we can define monoid quotients by arbitrary relations $R : M \rightarrow M \rightarrow \mathrm{Prop}$ if we take the monoid quotient by the congruence $\overline{R}$ \textit{generated} by $R$.
We proceed inductively as follows:
\begin{itemize}  
\item $\overline{R}(m, n)$ holds if $R(a, b)$ is true.
\item $\overline{R}(a, a)$ holds, for all $a : A$;
\item if $\overline{R}(a, b)$ is true then also $\overline{R}(b, a)$, for all $a, b : A$; 
\item $\overline{R}(a, b)$ and $\overline{R}(b, c)$ imply $\overline{R}(a, c)$, for all $a, b, c : A$.
\item If $\overline{R}(m, m^\prime)$ and $\overline{R}(n, n^\prime)$ hold, then $\overline{R}(m \cdot n, m^\prime \cdot n^\prime)$ is true.
\item  $\overline{R}(m, n)$ holds if we can merely produce this relation following the rules above.
\end{itemize}

$\overline{R}$ can also be characterized as the \textit{smallest} congruence extending the equivalence relation $R$. 
Note that even if $R$ already is an equivalence relation, we must include the rules for reflexivity, symmetry and transitivity in the construction of $\overline{R}$, because the multiplication rule may contribute new instances of them. Also,
in contrast to set quotients by equivalence relations we cannot construct a monoid quotient $M/R$ without the multiplication rule for $\overline{R}$, because it is not a consequence of Martin-L\"of's construction of equalities and free generation.

\subsection{Groups}  \label{groups-ssec}

We now extend the algebraic structure of monoids to that of a group: A monoid $G$ is a group if it is furthermore provided with
\begin{itemize}
\item an inverse $g^{-1} : G$, for each $g : G$;
\item inverseness: for all $g : G$, we have $g \cdot g^{-1} = g^{-1} \cdot g = e_G$.
\end{itemize}

A group $G$ is \textit{commutative}, or \textit{abelian} if the underlying monoid is commutative.

The most basic example of a (commutative) group is again the trivial group $\mathbbm{1}$, whose only object will be identified with the neutral element, inverse to itself.

A homomorphism between groups $G, H$ is a homomorphism $f : G \rightarrow H$ of the underlying monoids. In categorical terms, this means that the forgetful functor from groups to monoids is fully faithful.
Note in particular that a homomorphism $f : G \rightarrow H$ of groups preserves inverses:
\[ f(g^{-1}) = f(g)^{-1}\ \mathrm{for\ all\ } g \in G. \]
Thus, we can easily construct the \textit{category} of groups as a concrete category over the category of monoids.
Also, a group homomorphism is a monomorphism if, and only if the underlying monoid homomorphism is a monomorphism if, and only if the underlying map of sets is injective.

As with monoids, let us define a \textit{free group} by its standard universal property:
\begin{Def}   
A group $F$ with a map $h : A \rightarrow F$ is a \textbf{free group} on a set $A$ if for all groups $G$ and maps $h^\prime : A \rightarrow G$ there is a unique homomorphism $f : F \rightarrow G$ such that $h^\prime = f \circ h$.
\end{Def}

Again, this universal property can be interpreted as the higher recursive principle of a higher inductive type. But following our general strategy, we construct a free group $F(A)$ over a set $A$ by using only a few basic HITs:
As with monoids, we can \textit{present} the elements of a free group on a set $A$ by words concatenating arbitrarily many elements of $A$ and their inverses. Such words can be encoded as lists with entries in $A + A$, a sum type of two copies of $A$ where the first summand denotes elements $a : A$ and the second summand their inverses $a^{-1}$. By Prop.~\ref{list-free-monoid-prop} $\mathrm{List}(A + A)$ is a free monoid on $A + A$, but to make it into a group we need to divide out the inverseness equalities.
\begin{prop} [{\cite[Thm.6.11.7]{HoTTBook}}] \label{free-word-group-prop}
The quotient $Q(A)$ of the monoid $\mathrm{List}(A + A)$ by the congruence relation generated by all the word congruences $a \cdot a^{-1} = a^{-1} \cdot a = 1$ is a free group $F(A)$ on $A$, with map $h : A \rightarrow Q(A)$ induced by the left inclusion $A \rightarrow A + A$.
\end{prop}
\begin{proof}
$F(A)$ is a monoid, by Prop.~\ref{monoid-cong-quot-prop}. The inverse of a word is obtained by reverting it and exchanging elements of $A$ and their inverses; this process maps word congruences to word congruences, hence descends to the quotient $F(A)$, and the concatenation of a word and its inverse cancels down to the empty word using word congruences. So $F(A)$ is a group. 

A map $g : A \rightarrow G$ induces a map $\overline{g} : A + A \rightarrow G$ mapping $a$ from the left copy of $A$ to $a$ and from the other copy to $a^{-1}$. The unique monoid homomorphism from $h^\prime : \mathrm{List}(A + A) \rightarrow G$ to a group $G$ associated to $\overline{g}$ maps $a \cdot a^{-1}$ and $a^{-1} \cdot a$ to $e_G$, hence $h$ factorizes through $Q(A)$. The resulting group homomorphism $\overline{h}^\prime : Q(A) \rightarrow G$ satisfies $g = h \circ \overline{h}^\prime$, and it is the unique such homomorphism. 
\end{proof}

A \textit{subgroup} $H \subseteq G$ of a group $G$ is a subobject of $G$ in the category of groups, that is, a group $H$ together with a monomorphisms $\iota : H \rightarrow G$. 
As in Prop.~\ref{mono-mon-inj-prop}, this is equivalent to the underlying map of $\iota$ being injective. 
Hence as in Prop.~\ref{submon-subset-prop}, we can characterize subgroups $H \subseteq G$ as subsets respecting the structure of $G$; in particular, $e_G \in H$, and for all $g, h \in H$, we have $g \cdot h \in H$ and $g^{-1} \in H$. 

The \textit{kernel} and the \textit{image} of a group homomorphism $f : G \rightarrow H$ are subgroups of $G$ and $H$, uniquely maximal among those mapping to $\{e_H\}$ and uniquely minimal among those containing all image elements of the map $f$, respectively. 

Kernel and images can be used to construct subgroups from a given group homomorphism. For example, the
\textit{conjugated subgroup} $g \cdot H \cdot g^{-1}$ of $H$ in $G$ is the image of $H$ under the conjugation homomorphism $\iota_g : G \rightarrow G, h \mapsto ghg^{-1}$.
 
We can generate a subgroup of $G$ from a subset of elements $A \subseteq G$ as the image $H$ of the homomorphism from the free group on this set to $G$ induced by the inclusion. 
Then, by the construction of a free group in Prop.~\ref{free-word-group-prop}, every element of $H$ can be represented by a word
\[  g_1^{\epsilon_1} \cdot \cdots \cdot g_n^{\epsilon_n},\ g_i \in A, \epsilon_i \in \{-1, 1\}, i = 1, \ldots, n. \]
From this, we can deduce that for a group homomorphism $f : G \rightarrow G^\prime$, the subgroup $\mathrm{im}(f_{|H}) \subseteq G^\prime$ is generated by $f(A)$, as for any $g \in G$,
\[ f(g_1^{\epsilon_1} \cdot \cdots \cdot g_n^{\epsilon_n}) =
   f(g_1)^{\epsilon_1} \cdot \cdots \cdot f(g_n)^{\epsilon_n}. \] 
Note that this cannot be shown just using the minimality of images -- for example, we need in addition that epimorphisms are strong.

Kernels in turn allow to characterize monomorphisms in a more direct way:
\begin{prop} \label{ker-mono-prop}  
A group homomorphism $f : G \rightarrow H$ between two groups $G, H$ is a monomorphism if, and only if $\mathrm{ker}(f) = \{e_G\}$, the trivial group. 
\end{prop} 
\begin{proof}
One direction is obvious. For the other direction, assume that $f(g) = f(g^\prime)$ for $g, g^\prime : G$. Then
\[ e_H = f(e_G) = f(g^{-1} \cdot g) = f(g^{-1}) \cdot f(g) = f(g^{-1}) \cdot f(g^\prime) = f(g^{-1} \cdot g^\prime),   \]
hence $g^{-1} \cdot g^\prime = e_G$ because $\mathrm{ker}(f) = \{e_G\}$. This implies $g = g^\prime$ and therefore the injectivity of $f$.
\end{proof}

A monoid quotient $Q$ of a group $G$ by a congruence relation $R : G \rightarrow G \rightarrow \texttt{Prop}$, with projection homomorphism $q : G \rightarrow Q$, inherits the group structure from $G$ by setting $q(g)^{-1} := q(g^{-1})$. Therefore, we can also call $Q$ a \textit{group quotient} by $R$.
However, to characterize group quotients in the category of groups we can use that the extra structure of a group allows to produce the congruence relation from subgroups of $G$:  

\begin{Def}
A subgroup $H \subseteq G$ of a group $G$ is \textbf{normal} if the conjugated subgroups $g \cdot H \cdot g^{-1}$ are equal to $H$ for all $g : G$. 
\end{Def}

To show that a subgroup $H \subseteq G$ is normal it actually suffices to verify that $H \subseteq g \cdot H \cdot g^{-1}$ for all $g : G$ since then $H \subseteq g^{-1} \cdot H \cdot g$ also implies
\[ g \cdot H \cdot g^{-1} \subseteq g \cdot g^{-1} \cdot H \cdot g \cdot g^{-1} = H. \]

\begin{prop} \label{ker-norm-prop}    
The kernel $\mathrm{ker}(f) \subseteq G$ of a homomorphism $f : G \rightarrow H$ between two groups $G, H$ is normal.
\end{prop}
\begin{proof}
Let $h : G$ satisfy $f(h) = e_H$. Then $f(g \cdot h \cdot g^{-1}) = f(g) \cdot f(h) \cdot f(g^{-1}) = e_H$ for $g : G$, so $g \cdot \mathrm{ker}(f) \cdot g^{-1} \subseteq \mathrm{ker}(f)$. Since the other inclusion holds by the same argument, with roles of $g$ and $g^{-1}$ reversed, $g \cdot \mathrm{ker}(f) \cdot g^{-1} = \mathrm{ker}(f)$, and $\mathrm{ker}(f)$ is normal.
\end{proof}

\begin{prop} \label{comm-norm-prop}
The \textbf{commutator subgroup} $[G,G] \subseteq G$ of a group $G$ generated by the commutators $aba^{-1}b^{-1}$, $a, b \in G$, is normal.
\end{prop}
\begin{proof}
For any $c \in G$, 
\[ c \cdot aba^{-1}b^{-1} \cdot c^{-1} = 
   cac^{-1} \cdot cbc^{-1} \cdot (cac^{-1})^{-1} \cdot (cbc^{-1})^{-1} \]
is a generator of $[G,G]$, too. Since the conjugations of generators generate the conjugated subgroup, this implies that $c \cdot [G,G] \cdot c^{-1} = [G,G]$, hence $[G,G]$ is normal.
\end{proof}

\begin{prop}
A subgroup $H \subseteq G$ of a group $G$ defines an equivalence relation $R_H : G \rightarrow G \rightarrow \mathtt{Prop}$ by letting $g \equiv_R g^\prime$ exactly if $g^{-1}g^\prime \in H$.

$R_H$ is a congruence if, and only if $H$ is normal. 

A congruence $R : G \rightarrow G \rightarrow \mathtt{Prop}$ defines a normal subgroup $H_R \subseteq G$ collecting all the elements $g \in G$ such that $g \equiv_R 1_G$.
\end{prop}
\begin{proof}
Reflexivity, symmetry and transitivity of $R_H$ are easy to check. If $H$ is normal then 
\[ (g_1h_1)^{-1} \cdot g_2h_2 = (h_1^{-1} \cdot (g_1^{-1}g_2) \cdot h_1) \cdot h_1^{-1}h_2 \in H  \]
when $g_1 \equiv_{R_H} g_2$ and $h_1 \equiv_{R_H} h_2$. If $R_H$ is a congruence, then $h \in H$, or equivalently $h^{-1} \equiv_{R_H} 1$, implies $g \cdot h^{-1} \cdot g^{-1} \equiv_{R_H} g \cdot 1 \cdot g^{-1}$, hence $(gh^{-1}g^{-1})^{-1} = ghg^{-1} \in H$. The same argument shows that $H_R$ is normal if $R$ is a congruence.
\end{proof}

Note that an equivalence relation on a group does not necessarily define a subgroup because it may not be compatible with multiplication.

These correspondences between relations and subgroups lead to a definition of quotient groups analogous to, but more compact than Def.~\ref{quot-monoid-def}:
\begin{Def}  \label{quot-group-def}    
A group $Q$ with a group homomorphism $\pi : G \rightarrow Q$ is called a \textbf{quotient group} of the group $G$ by the normal subgroup $H$ if
\begin{itemize}
\item[(i)] $\pi$ is surjective, and
\item[(ii)] $\ker \pi = H$.
\end{itemize}
\end{Def}

In particular, the monoid quotient of $G$ by the congruence associated to $H$, with inverses induced by the surjective quotient homomorphism, is a group quotient.
Conversely, the underlying monoid of a group quotient of $G$ by the normal group $H$ is a monoid quotient of $G$ by the congruence associated to $H$, as in Def.~\ref{quot-monoid-def}. 

It is straightforward to check the \textit{First Isomorphism Theorem of Groups}:
\begin{thm} \label{first-iso-grp-thm}   
If $f : G \rightarrow H$ is a homomorphism between groups $G$ and $H$, the image $\mathrm{im}(f) \subseteq H$ of $f$ (considered as a group) is a group quotient of $G$ by $\mathrm{ker}(f)$. 
\end{thm}
\begin{proof}
By definition, $f : G \rightarrow H$ is factorized by the group homomorphism $\iota : \mathrm{im}(f) \hookrightarrow H$, and the factorizing group homomorphism $f^\prime : G \rightarrow \mathrm{im}(f)$ is surjective. Furthermore, $f^\prime$ has kernel $\ker f$, as $\iota$ is a monomorphism.
\end{proof}

Usually, this theorem is stated as $G/\ker(f) \cong \mathrm{im}(f)$ together with the assertion that the isomorphism is unique (or \textit{natural}). Showing that $\mathrm{im}(f)$ is a group quotient makes these statements exact because group quotients also satisfy a universal property:
\begin{prop}  
A group $Q$ is a quotient group of a group $G$ by a normal subgroup $H \subseteq G$, with quotient homomorphism $q : G \rightarrow Q$, if, and only if
\begin{itemize}
\item[(i)] $\ker q = H$;
\item[(ii)] for all group homomorphisms $f : G \rightarrow Q^\prime$ to a group $Q^\prime$ such that $H \subseteq \mathrm{ker}(f)$, there is a unique group homomorphism $g : Q \rightarrow Q^\prime$ such that $f = g \circ q$.
\end{itemize}
\end{prop}
\begin{proof}
First, let $P$ be a group and $f : G \rightarrow P$ a group homomorphism such that $H \subseteq \ker f$. If $R_H$ is the congruence associated to the normal subgroup $H$, the subgroup inclusion shows that $g_1 \equiv_{R_H} g_2$ implies $f(g_1^{-1}g_2) = e_P$ and hence $f(g_1) = f(g_2)$. By the universal property of monoid quotients we have a unique monoid and therefore group homomorphism $g : Q \rightarrow P$ such that $f = g \circ q$.

Conversely, the above arguments show that the group quotient $G/H$ of $G$ by $H$ with quotient homomorphism $\pi : G \rightarrow G/Q$ constructed from the monoid quotient satisfies the universal property. The usual universal yoga and condition (i) yield an isomorphism $i : G/H \rightarrow Q$, whose underlying map of sets is bijective. Consequently, the underlying map of $q = i \circ \pi$ is surjective. 
\end{proof}

In particular, well-definedness for maps from group quotients can be checked on generators of~$H$. And once again, $G$ with the identity homomorphism $G \rightarrow G$ satisfy condition (ii), so (i) is needed.

Note that in an (additive) abelian group $A$, all subgroups $B \subseteq A$ are abelian 
and normal, 
and a quotient group of $A$ by $B$ is automatically abelian, too. 
A quotient group of a group $G$ by its normal commutator subgroup $[G,G]$ (see Prop.~\ref{comm-norm-prop}), with quotient map $\pi$, is abelian because $\pi(aba^{-1}b^{-1}) = \pi(e_G)$ implies $\pi(a)\pi(b) = \pi(b)\pi(a)$, for all $a,b \in G$. 
The abelian quotient also satisfies a modified universal property:
\begin{prop} \label{commutator-quot-prop}
To all group homomorphisms $f : G \rightarrow A$ from an arbitrary group $G$ to an abelian group $A$, there exists a unique homomorphism $g : G/[G,G] \rightarrow A$ factorizing the quotient homomorphism $\pi : G \rightarrow G/[G,G]$, that is, $f = g \circ \pi$.
\end{prop}
\begin{proof}
An equivalence class $[a] \in G/[G,G]$ must be mapped to $f(a) \in A$ by $g$, and this map is well-defined:
$f(a \cdot b \cdot a^{-1} \cdot b^{-1}) = e_A$, as $A$ is abelian. It is easy to verify that $g$ is a group homomorphism.
\end{proof}

As above, the universal property can be interpreted as the recursion principle of a higher inductive type $G/H$ , but we will not discuss this further  -- with one exception: In contrast to standard usage of group quotients, the HIT construction, and therefore the universal property also works for non-normal subgroups and allows us to define group quotients of a group $G$ by arbitrary subgroups $H \subseteq G$. 

Of course, Prop.~\ref{ker-norm-prop} excludes that the kernel of the quotient homomorphism $q : G \rightarrow G/H$ is $H$, if $H$ is not normal. Instead, the kernel of the quotient homomorphism will always be the \textit{normal closure} $\overline{H}_N$ of $H$, the smallest normal subgroup containing $H$. Equivalently, the underlying monoid of $G/H$ is the monoid quotient of $G$ by the congruence closing up the equivalence relation associated to $H$, as discussed at the end of \ref{monoid-ssec}.

The normal closure of $H$ exists because the intersection of normal groups is normal again, and the normal closure is equal to $H$ if, and only if, $H$ is normal. There is a dual way to construct the normal closure:
\begin{prop}    
The subgroup generated by all elements of all conjugated subgroups $g \cdot H \cdot g^{-1}$, $g : G$, is the normal closure $\overline{H}_N$ of $H$ in $G$.
\end{prop} 
\begin{proof}
Let $S := \{g \cdot h \cdot g^{-1} : g \in G, h \in H\}$ be the set of generators. Then obviously, $g \cdot S \cdot g^{-1} = S$ for all $g : G$. Hence the subgroup $\langle S \rangle$ generated by $S$ is contained in $g \cdot \langle S \rangle \cdot g^{-1}$, since $S = g \cdot S \cdot g^{-1} \subseteq g \cdot \langle S \rangle \cdot g^{-1}$. This implies that $\langle S \rangle$ is normal.

Since $H \subseteq S$ we have $H \subseteq \langle S \rangle$. If $N \subseteq G$ is another normal subgroup containing $H$, then $g \cdot H \cdot g^{-1} \subseteq g \cdot N \cdot g^{-1} = N$ for all $g : G$, hence $S \subseteq N$ and therefore $\langle S \rangle \subseteq N$ by the minimality of $\langle S \rangle$.
\end{proof}

\subsection{Rings and Modules} \label{ring-mod-ssec}

\textit{Commutative rings with one} are sets $R$ provided with
\begin{itemize}
\item addition: a commutative group structure, with composition results denoted by $r + s$ for $r, s : R$, the neutral element by $0$ and the element negative to $r : R$ by $-r$;
\item multiplication: a commutative monoid structure, with composition results denoted by $r \times s$ for $r, s : R$ and the neutral element by $1$;
\item a \textit{distributive law} asserting $(r + s) \times t = r \times t + s \times t$, for all $r, s, t : R$.
\end{itemize}

An element $u : R$ is a \textit{unit} of the commutative ring $R$ with one if there exists a $v : R$ such that $u \times v = 1$. The \textit{inverse} $v$ is actually uniquely determined.

A homomorphism $f : R \rightarrow S$ between two commutative rings $R, S$ with one is a group homomorphism on the underlying additive groups of $R$ and $S$, which is at the same time a monoid homomorphism on the underlying multiplicative monoids of $R$ and $S$. 
As before, equality of ring homomorphisms is decided by equality of the underlying set maps. Thus, commutative rings with one and their homomorphisms can be realized as a concrete category over the category of abelian groups, with fibers consisting of commutative monoid structures.

A \textit{module} over the ring $R$, for short \textit{$R$-module} is a set $M$ provide with 
\begin{itemize}
\item a commutative group structure, with composition results denoted by $m + n$ for $m, n : M$, the neutral element by $0$ and the element negative to $m : M$ by $-m$;
\item \textit{scalar multiplication}: a function $R \rightarrow M \rightarrow M$, composition results denoted by $r \cdot m$ for $r : R$, $m : M$.
\end{itemize}
Furthermore, there are several laws ensuring the compatibility of addition and multiplication in the ring and scalar multiplication:
\begin{itemize}
\item $r \cdot (s \cdot m) = (r \times s) \cdot m$ for all $r, s : R$, $m : M$;
\item $(r + s) \cdot m = r \cdot m + s \cdot m$ for all $r, s : R$, $m : M$;
\item $r \cdot (m + n) = r \cdot m + r \cdot n$ for all $r : R$, $m, n : M$;
\item $1 \cdot m = m$ for all $m : M$.
\end{itemize}

A homomorphism $f : M \rightarrow N$ of $R$-modules $M, N$ is a group homomorphism compatible with the scalar multiplication:
\[ f(r \cdot m) = r \cdot f(m)\ \mathrm{for\ all\ } r : R,\ m : M. \]
As a group homomorphism, an $R$-module homomorphism is completely determined by its underlying map of sets. It is straightforward to check that the identity map $M \rightarrow M$ is always an $R$-module homomorphism, and that the composition of two $R$-module homomorphisms is an $R$-module homomorphism. 
This gives the $R$-modules together with their homomorphisms the structure of a category, with a faithful forgetful functor to the category of sets and their maps. 

\begin{exm}
The trivial $R$-module consists of exactly one element and carries the group structure of the trivial group, with scalar multiplication the only possible choice. It is denoted by $(0)$, or even shorter, $0$. 

Since a homomorphism of $R$-modules always maps $0$ to $0$, for each $R$-module $M$ and the zero-module $0$ there are exactly one homomorphism $0 \rightarrow M$ and one homomorphism $M \rightarrow 0$. Categorically speaking, this makes $0$ into a zero object, that is, both initial and terminal. 
\end{exm}

\begin{exm} \label{direct-prod-ex}
Given $R$-modules $(M_i)_{i \in I}$ indexed by a set $I$ we can define their \textit{(direct) product} $\Pi_{i \in I} M_i$ as the set product of the $M_i$ provided with componentwise addition and scalar multiplication.
$\Pi_{i \in I} M_i$ actually is a product in the category of $R$-modules, with homomorphisms $\pi_i : \Pi_{i \in I} M_i \rightarrow M_i$ given by the projection to the $i$th factor, that is, $\Pi_{i \in I} M_i$ together with the $\pi_i$ satisfy the universal property of a product of the $R$-modules $M_i$. 
\end{exm}

\begin{exm} \label{direct-sum-ex}
To construct a direct sum $\bigoplus_{i \in I} M_i$ of $R$-modules $M_i$ indexed by a set $I$ we follow the same strategy as for the construction of a free group (see Prop.~\ref{free-word-group-prop}): We represent elements of $\bigoplus_{i \in I} M_i$ as finite sums of elements from any of the $M_i$, make these \textit{additive combinations} (in analogy to linear combinations) into a monoid by using concatenation and the empty sum $0$, and divide out the congruence relation generated by
\begin{itemize}
\item $m_1 + m_2 = m$ if $m_1, m_2$ and hence $m = m_1 + m_2 \in M_i$,
\item $0_i = 0$, where $0_i$ is the zero in $M_i$, and
\item $m_1 + m_2 = m_2 + m_1$ for arbitrary $m_1, m_2$ in the $M_i$.
\end{itemize} 
Addition is induced by concatenation, the zero element by the empty sum, negatives by $-(m_1 + \cdots + m_k) = (-m_1) + \cdots + (-m_k)$ and scalar multiplication by $r \cdot (m_1 + \cdots + m_k) = r \cdot m_1 + \cdots + r \cdot m_k$. Then $\bigoplus_{i \in I} M_i$ together with the obvious homomorphisms $\iota_i: M_i \rightarrow \bigoplus_{i \in I} M_i$ satisfy the universal property of a \textit{coproduct} (or \textit{sum}) of the $R$-modules $M_i$, $i \in I$: If we have homomorphisms $\phi_i : M_i \rightarrow Q$ to an $R$-module $Q$ for all $i \in I$, then $\bigoplus_{i \in I} \phi_i: \bigoplus_{i \in I} M_i \rightarrow Q$ mapping
$m_1 + \cdots + m_k$ with $m_j \in M_{i_j}$ to $\phi_{i_1}(m_1) + \cdots + \phi_{i_k}(m_k)$ is the unique $R$-module homomorphism such that $\bigoplus_{i \in I} \phi_i \circ \iota_i = \phi_i$.

As with universal constructions involving groups we can use the module structure to condense the requirements for a module to satisfy the universal property of a direct sum:
\begin{prop}
An $R$-module $M$, together with $R$-module homomorphisms $\phi_i : M_i \rightarrow M$, is a direct sum of the $R$-modules $M_i$, $i \in I$, if, and only if,
\begin{itemize}
\item[(i)] all elements of $M$ can be written as a finite sum $\sum_i \phi_i(m_i)$, $m_i \in M_i$, and 
\item[(ii)] $\sum_i \phi(m_i) = 0$ implies $\phi_i(m_i) = 0$, for all of the finitely many $i \in I$.
\end{itemize}  
\end{prop}
\begin{proof}
The direct sum $\bigoplus_i M_i$ constructed above, together with the homomorphisms $\phi_i$ naturally given by the quotient construction, satisfy (i) and (ii) because the $\phi_i$ are injective. Conversely, if $M$ and the $\phi_i$ satisfy (i) and (ii) and $Q$ is a $R$-module with homomorphisms $\psi_i : M_i \rightarrow Q$ then $\psi : M \rightarrow Q$ given by $\sum_i \phi_i(m_i) \mapsto \sum_i \psi_i(m_i)$ is a well-defined $R$-module homomorphism, and we have $\psi_i = \psi \circ \phi_i$.
\end{proof}

This characterisation is the analogue of Strickland's characterisation of localisations (see Prop.~\ref{Strickland-prop}) for direct sums.

Note that a more standard way to represent elements of $\bigoplus_{i \in I} M_i$ is as finite tuples of elements of \textit{disjoint} modules $M_i$. But working with such tuples requires to decide the equality of indices $i, j$, which in general is only possible if we assume LEM.
\end{exm}

\begin{exm}
The ring structure on a commutative ring $R$ with one also provides $R$ with the structure of an $R$-module.

More generally, if $f : R \rightarrow S$ is a homomorphism of commutative rings with one then $S$ can be seen as an $R$-module, with scalar multiplication given by $r \cdot s := f(r) \times s$, for all $r : R$ and $s : S$. We call $S$ an \textit{$R$-algebra}. 

A homomorphism $f : R \rightarrow S$ of commutative rings with one can then be interpreted as an $R$-module homomorphism between the $R$-module $R$ and the $R$-algebra $S$. Furthermore, $S$-modules can be \textit{pulled back} to $R$-modules, with the same underlying set and scalar multiplication induced by the ring homomorphism and the multiplication by scalars in $S$.
\end{exm}

\begin{exm} \label{Z-mod-ex}
$\mathbb{Z}$-modules are exactly the abelian groups, with scalar multiplication given by $0 \cdot a := 0$ for $a$ an element of the abelian group, $n \cdot a := \sum_1^n a$ if $n \geq 1$ and $n \cdot a := \sum_1^n (-a)$ if $n \leq -1$. 
By choosing $1 : R$ for $a$, we can construct a unique homomorphism from $\mathbb{Z}$ to every commutative ring $R$ with one in this way, that is, the commutative ring $\mathbb{Z}$ is an \textit{initial object} in the category of commutative rings, and a commutative ring $R$ always is a $\mathbb{Z}$-algebra.
\end{exm}

\begin{exm}
By adding the values of two $R$-module homomorphisms $f, g : M \rightarrow N$ we obtain a homomorphism $f + g : M \rightarrow N$ as their sum. We can also multiply each value of $f : M \rightarrow N$ with a given $r : R$ and obtain a scalar product homomorphism $r \cdot f : M \rightarrow N$. Thus we provide the set of homomorphisms $\mathrm{Hom}_R(M, N)$ between two $R$-modules $M, N$ with the structure of an $R$-module. Composition of homomorphisms becomes bilinear with respect to this structure.
\end{exm}

\begin{exm} \label{hom-R-ex}
A homomorphism $f: R \rightarrow M$ from the $R$-module $R$ to an $R$-module $M$ is completely determined by $m := f(1_R) : M$, setting $f(r) := rm$. Furthermore, $m$ can be chosen arbitrarily. In different words, $\mathrm{Hom}_R(R, M) \cong M$. 
\end{exm}

A \textit{monomorphism} of $R$-modules is an injective $R$-module homomorphism and can be characterized by the usual universal property.

A \textit{submodule} $N \subseteq M$ of an $R$-module $M$ is a subgroup of (the underlying group of) $M$ such that scalar multiplication is closed on $N$: 
\[ r \cdot n \in N\ \mathrm{for\ all\ } r : R,\ n \in N. \]
Alternatively, a submodule of an $R$-module $M$ can be given as an $R$-module $N$ together with a module monomorphism $\iota : N \hookrightarrow M$.

An \textit{ideal} of a commutative ring $R$ with one is a submodule of the $R$-module $R^1$.

\begin{exm}
Every module $M$ contains the trivial submodule $(0)$. 
\end{exm}

\begin{exm} \label{ann-exm}
Let $R$ be a commutative ring with one, $M$ an $R$-module and $m \in M$. Then the \textit{annihilator} $\mathrm{Ann}(m)$ of $m$ in $R$ consists of all elements $r : R$ such that $r \cdot m = 0$. It is an ideal, since $r_1 \cdot m = 0$ and $r_2 \cdot m = 0$ implies $(r_1 + r_2) \cdot m = 0$ and $r \cdot m = 0$ implies $r^\prime r \cdot m = 0$.

If $S \subseteq R$ is a submonoid of the multiplicative monoid $R$ then the \textit{annihilator} $\mathrm{Ann}(S)$ of $S$ in $R$ consists of all elements $r : R$ such that there exists an $s \in S$ with $r \cdot s = 0$. This annihilator is an ideal, too: $r_1 \cdot s_1 = 0$ and $r_2 \cdot s_2 = 0$ implies $(r_1 + r_2) \cdot s_1s_2 = 0$, and $s_1s_2 \in S$ as $S$ is a monoid. Furthermore, $r \cdot s = 0$ implies $r^\prime r \cdot s = 0$.
\end{exm}

\begin{exm}
Let $R$ be a commutative ring with one and $f : M \rightarrow N$ a homomorphism of $R$-modules $M, N$. The \textit{kernel} $\ker(f)$ of $f$ is the kernel of the underlying homomorphism of groups. It is a submodule of $M$: If $f(m) = 0$ then $f(r \cdot m) = r \cdot f(m) = r \cdot 0 = 0$.

As in Prop.~\ref{ker-mono-prop}, a homomorphism of $R$-modules is a monomorphism iff its kernel is trivial, that is equal to $(0)$.

If $f: R \rightarrow S$ is a homomorphism of commutative rings with one, then interpreting $f$ as an $R$-module homomorphism we can describe the kernel of $f$ as an ideal $\ker(f)$ in $R$.
\end{exm}

\begin{exm}
Let $R$ be a commutative ring with one and $f : M \rightarrow N$ a homomorphism of $R$-modules $M, N$. The \textit{image} $\mathrm{im}(f)$ of $f$ is the image of the underlying homomorphism of groups. It is a submodule of $N$: Since $r \cdot f(m) = f(r \cdot m)$, the elements of $\mathrm{im}(f)$ are closed under scalar multiplication.
\end{exm}

\begin{exm} \label{mod-complex-exm}
A \textit{complex} of $R$-modules is a sequence $(f_n)_{n \in \mathbb{Z}}$ of $R$-modules homomorphisms, 
\[ \cdots \rightarrow M_{n+1} \stackrel{f_{n+1}}{\rightarrow} M_n \stackrel{f_n}{\rightarrow} M_{n-1} \rightarrow \cdots , \]
such that $\mathrm{Im}(f_{n+1}) \subseteq \mathrm{ker}(f_n) \subseteq M_n$ for all $n \in \mathbb{Z}$. 

A complex is called \textit{exact} if $\mathrm{Im}(f_{n+1}) = \mathrm{ker}(f_n)$ for all $n \in \mathbb{Z}$.

If $M_n = 0$ for all $n \leq N_-$ and/or all $n \geq N_+$, we can cut off the (exact) complex at $N_-$ resp.\ $N_+$ because $\mathrm{Im}(f_{n+1}) = \mathrm{ker}(f_n)$ is automatically satisfied if $M_{n+1} = M_n = M_{n-1} = 0$, and we obtain a \textit{bounded (exact) complex} (to the left resp.\ to the right).  

A \textit{short exact sequence} is a bounded exact complex of the form
\[ 0 \rightarrow M^\prime \stackrel{f}{\rightarrow} M \stackrel{g}{\rightarrow} M^{\prime\prime} \rightarrow 0, \]
that is, a complex with $5$ terms where $f$ is a monomorphism, $g$ is an epimorphism, and $\mathrm{Im}(f) = \mathrm{ker}(g)$.
\end{exm}

\begin{exm}
Since direct sums of $R$-modules are coproducts (see Ex.~\ref{direct-sum-ex}), a set $A \subseteq M$ of elements in the $R$-module $M$ induces a unique $R$-module homomorphism $f_A: \bigoplus_{m \in A} R \rightarrow M$ such that $f_m := f_A \circ \iota_m: R \rightarrow M$ is the map determined by $f_m(1_R) = m$, as in Ex.~\ref{hom-R-ex}. The image of $f_A$ is called the $R$-(sub)module (of $M$) \textit{generated} by $A$ and denoted by $\sum_{m \in A} Rm$. If the image of $f_A$ is equal to $M$, $A$ is said to \textit{generate} the module $M$, and the elements of $A$ are called \textit{generators} of $M$. If $A$ is a finite set, $M$ is called a \textit{finitely generated} module. But even if $M$ is not finitely generated, generators always exist, for example the set of all elements of $M$.

More generally, let $M_k \subseteq M$ be a set of submodules of an $R$-module $M$ indexed by the set $K$. As before, the inclusion homomorphisms induce a module homomorphism $\bigoplus_{k \in K} M_k \rightarrow M$, and $\sum_{k \in K} M_k$ denotes the $R$-(sub)module (of $M$) \textit{generated} by the submodules $M_k$. 
\end{exm}

\begin{exm} \label{free-pres-exm}
Given an arbitrary $R$-module $M$, we can find a (highly non-unique) \textit{free presentation} of $M$, 
\[ F \stackrel{h_1}{\rightarrow} G \stackrel{h_0}{\rightarrow} M \rightarrow 0, \]
that is, an exact complex of $R$-modules, where $F$ and $G$ are free, $h_0$ is an epimorphism, and   
$\mathrm{Im}(h_1) = \mathrm{ker}(h_0)$. For example, for any set of generators $\{m_i\}_{i \in I}$ of $M$ we can construct the epimorphism $h_0 : G := \bigoplus_{i \in I} R \cdot g_i \rightarrow M$, $g_i \mapsto m_i$, and for any set of generators $(h_j)_{j \in J}$ of $\mathrm{ker}(h_0) \subseteq G$ we can construct a $R$-homomorphism $h_1 : F := \bigoplus_{j \in J} R \cdot f_j \rightarrow G$, $f_j \mapsto h_j$ which satisfies $\mathrm{im}(h_1) = \mathrm{ker}(h_0)$. 
\end{exm}

Let $N \subseteq M$ be a submodule of the $R$-module $M$. Since $M$ is an (additive) abelian group, $N \subseteq M$ is a normal subgroup, and we can construct the group quotient $M/N$ with quotient homomorphism $q: M \rightarrow M/N$ such that $\ker(q) = N$, as in Def.~\ref{quot-group-def}. Scalar multiplication descends from $M$ to $M/N$ since $m - m^\prime \in N$ implies $rm - rm^\prime = r(m - m^\prime) \in N$, hence $M/N$ is an $R$-module, too, satisfying the universal properties of a quotient of $R$-modules. 

If $i: N \rightarrow M$ is a monomorphism we can interpret $N$ as a submodule of $M$, and $N$ is (uniquely isomorphic or equal to) the kernel of the quotient map $q: M \rightarrow M/N$.

The \textit{cokernel} of a homomorphism $f: M \rightarrow N$ of $R$-modules is the $R$-module quotient $\mathrm{coker}(f) = N/\mathrm{im}(f)$. Reformulating the universal property of $R$-module quotients we see that the quotient map $N \rightarrow N/\mathrm{im}(f)$ uniquely factorizes all $R$-module homomorphisms $g : N \rightarrow Q$ such that $g \circ f = 0$.

 An \textit{epimorphism} of $R$-modules is a surjective $R$-module homomorphism and can be characterised by the dual universal property. The map $p: M \rightarrow \mathrm{im}(f)$ factorizing $f : M \rightarrow N$ is surjective, hence an epimorphism. Consequently, every homomorphism $f : M \rightarrow N$ is the composition of an epimorphism (to $\mathrm{im}(f)$) and a monomorphism (embedding $\mathrm{im}(f)$).

If $f: M \rightarrow Q$ is an epimorphism of $R$-modules, then $Q$ is uniquely isomorphic or equal to the $R$-module quotient $M/\ker(f)$, see Cor.~\ref{first-iso-grp-thm}.

\begin{lem}
Let $f: M \rightarrow M^\prime$ be an epimorphism of $R$-modules, and let $g: M \rightarrow N$ be a homomorphism to an $R$-module $M$ such that $\ker(f) \subseteq \ker(g)$. Then there is a unique module homomorphism $g^\prime: M^\prime \rightarrow N$ such that $g = g^\prime \circ f$.
\end{lem}
\begin{proof}
This is nothing else than the universal property of the $R$-module quotient $M/\ker(f) \cong M^\prime$.
\end{proof}

\begin{prop} \label{map-sum-prop}
Let $N_1, N_2 \subseteq N$ be two submodules of an $R$-module $N$. Then two module homomorphisms $f_1: N_1 \rightarrow M$ and $f_2: N_2 \rightarrow M$ to an $R$-module $M$ such that $f_{1|N_1 \cap N_2} = f_{2|N_1 \cap N_2}$ define a unique module homomorphism $f: N_1 + N_2 \rightarrow M$ such that $f_{|N_1} = f_1$ and $f_{|N_2} = f_2$.
\end{prop}
\begin{proof}
$f_1$ and $f_2$ induce a unique module homomorphism $\overline{f}: N_1 \oplus N_2 \rightarrow M$ such that $f_i = \overline{f} \circ \iota_i$, $i = 1, 2$, since direct sums are coproducts. The kernel of $p: N_1 \oplus N_2 \rightarrow N_1 + N_2$ consists of pairs $(n_1, n_2)$ such that $n_1 + n_2 = 0_N$, or $n_1 = -n_2$. But this immediately implies $n_1, n_2 \in N_1 \cap N_2$, hence 
\[ \overline{f}(n_1, n_2) = f_1(n_1) + f_2(n_2) = f_1(n_1) + f_1(n_2) = f_1(n_1 + n_2) = 0, \]
and consequently $\ker(p) \subseteq \ker(\overline{f})$, and we can apply the lemma above. 
\end{proof}

From now on, let $S$ be a submonoid of the multiplicative monoid of a commutative ring $R$ with one. We first characterize \textit{localisations} of $R$ in $S$ by their universal property:
\begin{Def}
A commutative ring $T$ with one is a \textbf{localisation} of $R$ in $S$, with localisation homomorphism $q : R \rightarrow T$ if for all commutative rings $T^\prime$ and ring homomorphisms $f : R \rightarrow T^\prime$ such that $S$ is mapped to units in $T^\prime$, there is a unique ring homomorphism $g : T \rightarrow T^\prime$ such that $f = g \circ q$.
\end{Def}

As before, the universal property is the recursion principle of a higher inductive type $R[S^{-1}]$. It has the following constructors:
\begin{itemize}
\item all the (higher) constructors making $R[S^{-1}]$ into a commutative ring with one: addition, multiplication, $0$, negative, $1$, associativity and commutativity of addition and multiplication, neutrality of $0$ and $1$, negativeness, distributivity;
\item a function $q : R \rightarrow R[S^{-1}]$;
\item an inverting function $i : S \rightarrow R[S^{-1}]$;
\item all the equalities ($=$ higher constructors) making $q$ and $i$ into ring homomorphisms;
\item inverseness: $q(s) \times i(s) = 1$ for all $s \in S$.
\end{itemize} 

\begin{prop}
$R[S^{-1}]$, together with the constructor function $q : R \rightarrow R[S^{-1}]$, is a localisation of $R$ in $S$.
\end{prop}
\begin{proof}
The existence of $g$ is a consequence of the recursion principle of the HIT $R[S^{-1}]$, the equality $f = g \circ q$ is one of the calculation rules. Uniqueness can be shown by the induction principle of $R[S^{-1}]$, for example by verifying that the composition $- \circ q$ is a right inverse of the construction of $g$ from $f$ by the recursion principle. 
\end{proof}

A more explicit construction of a localisation represents its elements by \textit{fractions} and then describes the equalities between these fractions.
\begin{prop}
Declaring $(r, s)$ and $(r^\prime, s^\prime)$ to be related if there (merely) is a $t \in S$ such that $t(rs^\prime - r^\prime s) = 0$ defines an equivalence relation on $R \times S$ whose set quotient $R\{S^{-1}\}$ together with the map $q : R \rightarrow R\{S^{-1}\}$ given by $r \mapsto \frac{r}{1}$ (where $\frac{r}{s}$ denotes the equivalence class of the pair $(r,s)$) is a localisation of $R$ in $S$.
\end{prop}
\begin{proof}
Reflexivity and symmetry of the relation are immediate. For the transitivity, assume that $(r,s) \sim (r^\prime,s^\prime)$ and $(r^\prime,s^\prime) \sim (r^{\prime\prime},s^{\prime\prime})$ or equivalently, $t(rs^\prime - r^\prime s) = 0$ and $t^\prime(r^\prime s^{\prime\prime} - r^{\prime\prime} s^\prime) = 0$ for some $t, t^\prime \in S$. Then
\[ tt^\prime s^\prime(r s^{\prime\prime} - r^{\prime\prime} s) = t^\prime s^{\prime\prime} \cdot t(rs^\prime - r^\prime s) + ts \cdot t^\prime(r^\prime s^{\prime\prime} - r^{\prime\prime} s^\prime) = 0, \]
so $(r,s) \sim (r^{\prime\prime},s^{\prime\prime})$, since $tt^\prime s^\prime$ is an element in the monoid $S$.

The usual calculation rules for fractions, $(r,s) + (r^\prime, s^\prime) = (r s^\prime + r^\prime s, ss^\prime)$ and $(r,s) \times (r^\prime, s^\prime) = (rr^\prime, ss^\prime)$ provide addition and multiplication on the set quotient, as these operations preserve the equivalence relation: If $(r,s) \sim (q,t)$, that is, there is $u \in S$ such that $u(rt - qs) = 0$, then 
\[ (r,s) + (r^\prime, s^\prime) = (rs^\prime + r^\prime s, ss^\prime) \sim (qs^\prime + r^\prime t, ts^\prime) = (q,t) + (r^\prime, s^\prime) \]
because 
\[ u((rs^\prime + r^\prime s)ts^\prime - (qs^\prime + r^\prime t)ss^\prime) = u(s^\prime)^2(rt - qs) = 0, \]
and 
\[ (r,s) \times (r^\prime, s^\prime) = (rr^\prime, ss^\prime) \sim (qr^\prime, ts^\prime) = (q,t) \times (r^\prime, s^\prime) \]
because
\[ u(rr^\prime ts^\prime - qr^\prime ss^\prime) = r^\prime s^\prime (rt - qs) = 0. \]
It is straightforward to verify additive and multiplicative associativity and commutativity. $0$ will be the equivalence class of $(0,1)$, and $1$ will be the equivalence class of $(1,1)$. The negative of (the equivalence class of) $(r,s)$ will be (the equivalence class of) $(-r,s)$: 
\[ (r,s) + (-r,s) = (rs - rs,s^2) = (0,s^2) \sim (0,1). \]
This is actually the only law which fails when trying to define the ring structure directly on $R \times S$ - this is only possible on the quotient $R\{S^{-1}\}$.

The map $R \rightarrow R\{S^{-1}\}$ given by $r \mapsto \frac{r}{1}$ is a ring homomorphism: $0$ and $1$ are mapped to $0$ and $1$, and the map is compatible with addition and multiplication.

So it remains to show that $R\{S^{-1}\}$ is a localisation. Let $f: R \rightarrow Q$ be a ring homomorphism to a commutative ring $Q$ with one that maps elements of $S$ to units of $Q$. The map (of sets) $\overline{f}: R \times S \rightarrow Q$ given by $(r,s) \mapsto f(r)f(s)^{-1}$ factorizes $f$ through the map $\lambda: r \mapsto (r,1)$. If $(r,s) \sim (r^\prime,s^\prime)$, that is, there is $u \in S$ such that $u(rs^\prime - r^\prime s) = 0$, we have $f(u)(f(r)f(s^\prime - f(r^\prime)f(s)) = 0$, and by multiplying with the inverses of $f(u)$, $f(s)$ and $f(s^\prime)$ we can conclude that
$\overline{f}(r,s) = \overline{f}(r^\prime,s^\prime)$. 

Therefore, the universal property of set quotients yields a unique map $g : R\{S^{-1}\} \rightarrow Q$ factorising the quotient map $q : R \times S \rightarrow R\{S^{-1}\}$, and it is straightforward to show that the composed map $g$ is a ring homomorphism such that $f = g \circ q \circ \lambda$. It is also the unique such ring homomorphism, as $g(\frac{r}{s}) = f(r)f(s)^{-1}$ is determined by $f(r)$ and $f(s)$.  
\end{proof}

This construction can be turned into another characterisation of localisations, condensing the description of equalities to specifying what the kernel of the map from the ring into a localisation is. This is nothing else than \textit{Strickland's characterisation}, extensively discussed in \cite{BHL+22, Buz24}.
\begin{prop} \label{Strickland-prop}
A commutative ring $T$ with one and a ring homomorphism $q : R \rightarrow T$ is a localisation of $R$ in $S$ if, and only if
\begin{itemize}
\item $q$ maps every element $S$ to a unit of $T$;
\item every element of $T$ can be written as a fraction $\frac{q(r)}{q(s)} = q(r)q(s)^{-1}$ for some $r : R$ and $s : S$ (which merely exist);
\item the kernel of $q$ is the annihilator ideal $\mathrm{Ann}(S)$ of $S$ in $R$ defined in Ex.~\ref{ann-exm} for the multiplicative submonoid $S$ of $R$.
\end{itemize} 
\end{prop}
\begin{proof}
First assume that $T$ is a ring satisfying the properties above. Then the map $q_S : R \times S \rightarrow T$ mapping $(r,s)$ to $q(r)q(s)^{-1}$ is surjective, and pairs mapping to the same fraction in $R\{S^{-1}\}$ are mapped to the same element in $T$ by $q_S$: $q(r)q(s)^{-1} = q(r^\prime)q(s^\prime)^{-1}$ if, and only if $q(rs^\prime - r^\prime s) = 0$. Since $\ker(q) = \mathrm{Ann}(S)$, this is equivalent to the existence of a $t \in S$ such that 
\[ t(rs^\prime - r^\prime s) = 0. \]
By the construction of well-defined functions in Ex.~\ref{well-def-exm} this means that we have a bijective map from $R\{S^{-1}\}$ to $T$, and it is straightforward to check that this bijection actually is a ring isomorphism. Hence, $T$ is a localisation.

Conversely,  we only need to check that the kernel of $q: R \rightarrow R\{S^{-1}\}$ is equal to $\mathrm{Ann}(S)$. But $\frac{r}{1} = 0 = \frac{0}{1}$ if, and only if there is a $s \in S$ such that $s(r \times 1 - 0 \times 1) = sr = 0$ which means exactly that $r \in \mathrm{Ann}(S)$. 
\end{proof}

\subsection{Tensor products of modules} \label{tensor-ssec}

To characterise \textit{tensor products} by their universal property we first need to introduce \textit{bilinear maps}:
\begin{Def}
A map $\phi: M \times N \rightarrow L$ between (the underlying sets of) the (set-theoretic) product of two $R$-modules $M, N$ and an $R$-module $L$ is called \textbf{bilinear} if
\begin{itemize}
\item[(i)] $\phi(rm, n) = \phi(m, rn) = r \cdot \phi(m,n)$ for all $m \in M$, $n \in N$, $r \in R$;
\item[(ii)] $\phi(m_1 + m_2, n) = \phi(m_1, n) + \phi(m_2,n)$, $\phi(m, n_1 + n_2) = \phi(m,n_1) + \phi(m,n_2)$ for all $m,m_1, m_2 \in M$, $n, n_1, n_2 \in N$.
\end{itemize}
\end{Def}
\begin{Def}
An $R$-module $T$ is a \textbf{tensor product} of two $R$-modules $M, N$, with $R$-bilinear map $\phi: M \times N \rightarrow T$, if for all bilinear maps $\psi: M \times N \rightarrow S$ to an $R$-module $S$ there is a unique $R$-module homomorphism $f: T \rightarrow S$ such that $\psi = f \circ \phi$. 
\end{Def}

As before, the universal property is the recursion principle of a higher inductive type $M \otimes_R N$. It has the following constructors:
\begin{itemize}
\item all the (higher) constructors making $M \otimes_R N$ into an $R$-module: addition, associativity and commutativity of addition, $0$, negative, scalar multiplication, associativity and distributivity of scalar multiplication;
\item a function $t: M \times N \rightarrow M \otimes_R N$;
\item the equalities making $t$ into an $R$-bilinear map.
\end{itemize}
\begin{prop}
$M \otimes_R N$, together with the $R$-bilinear map 
\[ t: M \times N \rightarrow M \otimes_R N, \] 
is a tensor product of $M$ and $N$.
\end{prop}
\begin{proof}
For an $R$-module $T$ with bilinear map $s: M \times N \rightarrow T$, the $R$-module homomorphism $\phi: M \otimes_R N \rightarrow T$ exists because of the recursion principle of the HIT $M \otimes_R N$, and the equality $s = \phi \circ t$ is one of the calculation rules. Uniqueness can be shown by using the induction principle of the HIT $M \otimes_R N$, for example by verifying that the composite $- \circ t$ is a right inverse of the construction of $s$ from $\phi$ by the recursion principle.
\end{proof}

The definition of a tensor product can be directly used to tensorize $R$-module homomorphisms:
\begin{Def}
Let $f : M^\prime \rightarrow M$ be a $R$-module homomorphism and $N$ a $R$-module. Then the \textbf{tensorized} $R$-module homomorphism $f \otimes_R N$ is given by the $R$-bilinear map
\[ (m^\prime, n) \mapsto f(m) \otimes n. \]
\end{Def}

A more explicit construction of a tensor product represents its elements by finite sums of \textit{tensors} and then describes the equalities between these tensors by an $R$-module quotient:
\begin{prop} \label{tensor-const-prop}
The quotient $T$ of a free $R$-module $\bigoplus_{(m,n) \in M \times N} R \cdot (m \otimes n)$ by the $R$-submodule generated by 
\[ (m_1 + m_2) \otimes n - m_1 \otimes n - m_2 \otimes n,\ 
     m \otimes (n_1 + n_2) - m \otimes n_1 - m \otimes n_2,\] 
\[ r(m \otimes n) - rm \otimes n,\  r(m \otimes n) - m \otimes rn, \]
together with the $R$-bilinear map $M \times N \rightarrow T$ determining the $R$-module quotient homomorphism $\bigoplus_{(m,n) \in M \times N} R \cdot (m \otimes n) \rightarrow T$, is a tensor product of $M$ and $N$.
\end{prop}
\begin{proof}
If $T^\prime$ is an $R$-module with a bilinear map $\psi : M \times N \rightarrow T^\prime$ then
\[ \sum_i m_i \otimes n_i \mapsto \sum_i \psi(m_i, n_i) \]
defines an $R$-module homomorphism $\Psi : T \rightarrow T^\prime$ because the bilinearity of $\psi$ implies that the generators of the kernel of the quotient homomorphism $\bigoplus_{(m,n) \in M \times N} R \cdot (m \otimes n) \rightarrow T$ are mapped to $0$. It is also the unique such homomorphism because the homomorphism $\Psi \circ q: \bigoplus_{(m,n) \in M \times N} R \cdot (m \otimes n) \rightarrow T^\prime$ is uniquely determined by $\psi$.
\end{proof}

By further condensing the description of the kernel of the quotient homomorphism to a version of the equational criterion for the vanishing of tensors (see \cite[Lem.6.4]{ECA}) we can turn this construction into another characterisation of tensor products, which will be an analogue of Strickland's characterization of localisations.
\begin{prop} \label{eq-char-tensor-prop}
Let $M$ and $N$ be $R$-modules, and $T$ an $R$-module with a bilinear map $\phi: M \times N \rightarrow T$. Let $\{n_i\}$ be a set of generators of $N$. Then $T$ is a tensor product of $M$ and $N$ if
\begin{itemize} 
\item every element of $T$ can be written as a finite sum $\sum_i \phi(m_i, n_i)$, with $m_i \in M$ arbitrary, and
\item for any finite set of pairs $(m_i, n_i) \in M \times N$ with $m_i \in M$ arbitrary, $\sum_i \phi(m_i, n_i) = 0$ implies the existence of elements $m_j^\prime \in M$, $a_{ij} \in R$ such that
\[ \sum_j a_{ij}m_j^\prime = m_i\ \mathrm{for\ all\ } i \]
and
\[ \sum_i a_{ij}n_i = 0\ \mathrm{for\ all\ } j. \] 
\end{itemize}
\end{prop}
\begin{proof}
If $T$ satisfies the two conditions above, the kernel of the surjective $R$-module homomorphism
\[ \Phi : \bigoplus_{(m,n) \in M \times N} R \cdot (m \otimes n) \rightarrow T,\ m \otimes n \mapsto \phi(m,n), \]
will be generated by the bilinearity relations listed in Prop.~\ref{tensor-const-prop}: First of all, the bilinear relations are mapped to $0$ by $\Phi$, as $\phi$ is bilinear. 

Next, every element $t \in \bigoplus_{(m,n) \in M \times N} R \cdot (m \otimes n)$ is a finite sum $\sum_k r_k(m_k^\prime \otimes n_k^\prime)$, $r_k \in R$, $m_k^\prime \in M$, $n_k^\prime \in N$, and is equal to a finite sum $\sum_i m_i \otimes n_i$, $m_i \in M$, modulo the bilinear relations, because there are $b_{ki} \in R$ such that $n_k^\prime = \sum_i b_{ki}n_i$, and
\[ \sum_k r_k(m_k^\prime \otimes n_k^\prime) = \left[ \sum_k r_k(m_k^\prime \otimes \sum_i b_{ki} n_i)
    - \sum_k(r_km_k^\prime \otimes \sum_i b_{ki} n_i) \right] \]
\[ \makebox[5cm]{} + \left[ \sum_k(r_km_k^\prime \otimes \sum_i b_{ki} n_i) - \sum_k \sum_i
     (r_km_k^\prime \otimes b_{ki}n_i) \right] \]
\[ \makebox[5cm]{} + \left[ \sum_k \sum_i (r_km_k^\prime \otimes b_{ki}n_i) - \sum_k \sum_i (r_kb_{ki}
    m_k^\prime \otimes n_i) \right] \]
\[ \makebox[7cm]{} + \sum_k \sum_i (r_kb_{ki} m_k^\prime \otimes n_i), \]
where the terms in the square brackets are $R$-linear combinations of the bilinear relations, so we can set $m_i := \sum_k r_kb_{ki} m_k^\prime$.
  
Now, $\Phi$ maps an element $t = \sum_i m_i \otimes n_i \in \bigoplus_{(m,n) \in M \times N} R \cdot (m \otimes n)$ to a finite sum $\sum_i \phi(m_i, n_i) \in T$, so $t \in \mathrm{ker}(\Phi)$ iff $\sum_i \phi(m_i, n_i) = 0 \in T$. Using $m_j^\prime \in M$, $a_{ij} \in R$ such that $\sum_j a_{ij}m_j^\prime = m_i$ for all $i$ and $\sum_i a_{ij} n_i = 0$ for all $j$, we obtain
\[ \sum_i m_i \otimes n_i = \left[ \sum_i (\sum_j a_{ij} m_j^\prime \otimes n_i) - 
             \sum_i \sum_j a_{ij} (m_j^\prime \otimes n_i) \right] \]
\[ \makebox[4cm]{} + \left[ \sum_i \sum_j a_{ij} (m_j^\prime \otimes n_i) - \sum_j (m_j^\prime \otimes \sum_i a_{ij} n_i) \right] \]  
\[ \makebox[4cm]{} + \left[ \sum_j (m_j^\prime \otimes 0) - \sum_j 0 \cdot (m_j^\prime \otimes 0) \right] + \sum_j 0. \]
The terms in the brackets are $\mathbb{R}$-linear combinations of the $\mathbb{R}$-bilinear relations, whereas the last term is $0$.

Therefore, $T$ is isomorphic to $M \otimes_R N$ as constructed in Prop.~\ref{tensor-const-prop} and also a tensor product.
\end{proof}

To prove the converse of this characterisation we need to show several identifications of tensor products with differently constructed $R$-modules. In the process, it turns out that the equational criterion work is not so easy to verify in many cases, and we mainly use the defining the property of being a tensor product; but this also works efficiently.
\begin{prop} \label{tensor-prod-ident-prop}
Let $M$, $N$ and $N_i$, $i \in I$, be $R$-modules.
\begin{itemize}
\item[(a)] $M$, together with the $R$-bilinear map $R \times M \rightarrow M$, $(r,m) \mapsto rm$, is a tensor product of the $R$-modules $R$ and $M$.
\item[(b)] $N \otimes_R M$, together with the $R$-bilinear map 
\[ M \times N \rightarrow N \otimes_R M,\ (m,n) \mapsto n \otimes m, \] 
is a tensor product of the $R$-modules $M$ and $N$.
\item[(c)] $\bigoplus_i \left( M \otimes_R N_i\right)$, together with the $R$-bilinear map 
\[ M \times \bigoplus_i N_i \rightarrow \bigoplus_i \left( M \otimes_R N_i\right),\ 
   (m, \sum_i n_i) \mapsto \sum_i m \otimes n_i,\] 
is a tensor product of the $R$-modules $M$ and $\bigoplus_i N_i$.
\item[(d)] \label{tensor-right-ex-prop}
Taking tensor products is \textbf{right-exact}, that is, for every exact complex 
\[ M^\prime \stackrel{f}{\rightarrow} M \stackrel{g}{\rightarrow} M^{\prime\prime}  
      \rightarrow 0 \] 
of $R$-modules and every $R$-module $N$,
\[ M^\prime \otimes_R N \stackrel{f \otimes_R N}{\longrightarrow} M \otimes_R N 
                \stackrel{g \otimes_R N}{\longrightarrow} M^{\prime\prime} \otimes_R N \longrightarrow 0 \]
is also an exact complex. 
\end{itemize}
\end{prop}
\begin{proof} 
\begin{itemize}
\item[(a)] If $\sum r_i m_i = 0$, then take $m_i^\prime := r_im_i$ and $a_{ii} = r_i$. Thus, the equational criterion is satisfied.
\item[(b)] $N \otimes_R M$, together with the bilinear map $\phi$ given in the claim, satisfies the universal property of a tensor product of $M$ and $N$: If $\psi : M \times N \rightarrow T$ is an $R$-bilinear map to another $R$-module $T$, then we can factorize $\psi$ by $\phi$ and the $R$-module homomorphism 
\[ N \otimes_R M \rightarrow T\ \mathrm{uniquely\ given\ by\ } (n, m) \mapsto \psi(m,n), \]
using the universal property of the tensor product $N \otimes_R M$.
\item[(c)] $\bigoplus_i (M \otimes_R N_i)$, together with the bilinear map $\phi$ given in the claim, satisfies the universal property of a tensor product of $M$ and $\bigoplus_i N_i$: If $\psi : M \times \bigoplus_i N_i \rightarrow T$ is an $R$-bilinear map to another $R$-module $T$, then we can factorize $\psi$ by $\phi$ and the $R$-module homomorphism 
\[ \bigoplus_i (M \otimes_R N_i) \rightarrow T\ \mathrm{uniquely\ given\ by\ } (m, n_i) \mapsto \psi(m,n_i), \]
where we interpret $n_i$ as an element in $\bigoplus_i N_i$  and use the universal properties of a direct sum and of a tensor product $M \otimes_R N_i$.
\item[(d)] By assumption, $M^{\prime\prime}$ is a cokernel of $f$, with projection homomorphism $g$, and we need to show that $M^{\prime\prime} \otimes_R N$ is the cokernel of $f \otimes_R N$, with projection homomorphism $g \otimes_R N$. 

To this purpose, take an $R$-module homomorphism $h_N : M \otimes_R N \rightarrow Q$ to an $R$-module $Q$ such that $h_N \circ (f \otimes_R N) = 0$. Then there is a $R$-bilinear map $M^{\prime\prime} \times N \rightarrow Q$ mapping $(m^{\prime\prime},n)$ to $h_N(m \otimes n)$, where $m \in M$ satisfies $g(m) = m^{\prime\prime}$. This is well-defined, since any choice of $m$ yields the same element in $Q$ because $h_N \circ (f \otimes_R N) = 0$. The induced $R$-module homomorphism $M^{\prime\prime} \otimes_R N \rightarrow Q$ is the only one factorising $h_N$ by $g \otimes_R N$, showing that $M^{\prime\prime} \otimes_R N$ is the cokernel of $f \otimes_R N$.  
\end{itemize}
\end{proof}

We can now show the converse of the equational criterion:
\begin{prop} \label{tensor-eq-crit-prop}
If $T$ is a tensor product of the $R$-modules $M$ and $N$, with an $R$-bilinear map $\phi : M \times N \rightarrow T$ and a set $\{n_i\}_{i \in I}$ of generators of $N$, then for any finite set of pairs $(m_i, n_i) \in M \times N$ with $m_i \in M$ arbitrary, $\sum_i \phi(m_i, n_i) = 0$ if, and only if, there exist elements $m_j^\prime \in M$, $a_{ij} \in R$, such that
\[ \sum_j a_{ij}m_j^\prime = m_i\ \mathrm{for\ all\ } i \]
and
\[ \sum_i a_{ij}n_i = 0\ \mathrm{for\ all\ } j. \] 
\end{prop}
\begin{proof}
Choose a free presentation of $N$, 
\[ \bigoplus_{j \in J} R \cdot f_j \stackrel{\psi}{\rightarrow}  
   \bigoplus_{i \in I} R \cdot e_i \stackrel{\phi}{\rightarrow} N \rightarrow 0, \]
such that $\phi(e_i) = n_i$ (see Ex.~\ref{free-pres-exm}) and $\psi(f_j) = \sum_{i} a_{ij}e_i$ for some $a_{ij} \in R$. In particular, $0 = \phi(\psi(f_j)) = \sum_{i} a_{ij}n_i$ for all $j$. If we tensorize this presentation with $M$ and use the identifications of Prop.~\ref{tensor-prod-ident-prop}, we obtain the following exact complex:
\[ \bigoplus_{j \in J} M \cdot f_j \stackrel{M \otimes \psi}{\longrightarrow}  
   \bigoplus_{i \in I} M \cdot e_i \stackrel{M \otimes \phi}{\longrightarrow} M \otimes_R N \rightarrow 0, \]
where $(M \otimes \psi)(mf_j) = \sum_i a_{ij} me_i$ and $(M \otimes \phi)(me_i) = m \otimes n_i$. Then, if $\sum_i m_ie_i$ is in the kernel of $M \otimes \phi$, it also is in the image of $M \otimes \psi$, hence there are $m_j^\prime \in M$ such that
\[ \sum_i m_ie_i = (M \otimes \psi)(\sum_j m_j^\prime f_j) = 
    \sum_j \left( \sum_i a_{ij} m_j^\prime e_i \right) = \sum_i \left( \sum_i a_{ij} m_j^\prime \right) e_i. \]
This implies $\sum_i a_{ij} m_j^\prime = m_i$, using the properties of direct sums.  
\end{proof}

\section{Canonical Isomorphisms of Non-Canonical Choices} \label{choice-sec}

We now investigate the issue of non-canonical choices in constructions appearing in homological algebra. If the resulting objects satisfy a universal property, the different choices lead to uniquely equivalent and therefore canonically equal objects, by univalence: we show this for left derived functors in the category of $R$-modules.

Also note that the construction may involve the Axiom of Choice when done uniformly for all objects of a certain type, but more explicit construction are possible, too.

To avoid a full-blown elaboration of the theory of abelian categories we restrict ourselves to constructing the chomomology of $R$-modules over a commutative ring $R$ with one. This special case is important, also underlying the construction of cohomology of sheaves on algebraic varieties. Also, the possibility of sign choices in the construction of boundary maps is even better hidden than in the group cohomology case discussed by Buzzard. Furthermore, we need the Axiom of Choice to show that there are enough injective $R$-modules, and it will be interesting to see how the presence of the axiom influences the efficiency of the construction of cohomology. This approach is in contrast to the purely constructive use of HoTT in the construction of cohomology groups (more specifically, Ext groups) presented by \cite{CTF25}.

We will closely follow the presentation in \cite[III.1]{Hart:AG}, always highlighting the places where homotopy type theory plays a role. On the other hand, we omit details that simply follow the standard arguments.

\subsection{Abelian Category of $R$-modules}

We need from Sec.~\ref{ring-mod-ssec} that the $R$-modules with their homomorphisms form a \textit{category}, and that homomorphisms between two $R$-modules can be added and multiplied with scalars from $R$, thus giving them the structure of an $R$-module. Furthermore, $R$-modules have
\begin{itemize}
\item the zero object $(0)$;
\item finite direct sums, which are both products and coproducts;
\item a kernel and a cokernel for each module homomorphism.
\end{itemize}

Also, monomorphisms are kernels of an $R$-module homomorphism, and epimorphisms are cokernels, and every homomorphism can be factored into an epimorphism followed by a monomorphism. Thus, $R$-modules and their homomorphisms form an \textit{abelian category} $R\mathrm{-Mod}$.

Note that in homotopy type theory (or at least in \cite[9.1]{HoTTBook}), categories associate \textit{sets} of homomorphisms to two objects. Furthermore, a full-blown category, in contrast to a precategory, asserts univalence in the framework of the category: isomorphisms are equivalent to equalities. Both properties hold for $R\mathrm{-Mod}$, as the type of all maps between sets is a set again, and standard univalence can be applied on two $R$-modules with isomorphisms inverse to each other between them.

We will consider \textit{additive functors} 
\[ F: R\mathrm{-Mod} \rightarrow R\mathrm{-Mod}, \] 
that is, functors on the category $R\mathrm{-Mod}$ to itself whose map on module homomorphisms between two $R$-modules $M, N$ is an $R$-module homomorphism $\mathrm{Hom}_R(M, N) \rightarrow \mathrm{Hom}_R(F(M), F(N))$.

\subsection{Homological Algebra} \label{hom-alg-ssec}

The $i^\mathrm{th}$ \textit{cohomology module} $h^i(M^\bullet)$ of a complex $M^\bullet = (M^i, f^i)_{i \in \mathbb{Z}}$ of $R$-modules $M^i$ and module homomorphisms $f^i: M^i \rightarrow M^{i-1}$ (see Ex.~\ref{mod-complex-exm}) is the quotient $\ker{f^i}/\mathrm{im}(f^{i+1})$. 

A \textit{homomorphism} $\phi: M^\bullet \rightarrow N^\bullet$ \textit{of complexes of $R$-modules} is a collection $\phi^i: M^i \rightarrow N^i$ of $R$-module homomorphisms commuting with the coboundary maps $f^i$ and $g^i$ of $M^\bullet$ and $N^\bullet$, respectively. Such a homomorphism induces a ``natural" $R$-module homomorphism $h^i(\phi): h^i(M^\bullet) \rightarrow h^i(N^\bullet)$ --- natural in the sense that a cohomology class represent by $m_i \in \ker(f^i) \subseteq M^i$ is mapped by $h^i(\phi)$ to a cohomology class represented by $\phi^i(m_i) \in \ker(g^i) \subseteq N^i$. Also, $h^i$ respects composition of homomorphisms of complexes, hence behaves ``functorial".

Two homomorphisms of complexes $\phi, \psi: M^\bullet \rightarrow N^\bullet$ are \textit{homotopic}, denoted by $\phi \sim \psi$, if there is a collection of $R$-module homomorphisms $k^i: M^i \rightarrow N^{i+1}$ such that $\phi^i - \psi^i = k^{i-1} \circ f^i + g^{i+1} \circ k^i: M^i \rightarrow N^i$, for all $i$. Two homotopic homomorphisms of complexes induce the same maps on cohomology modules. Here, the equality of functions $h^i(\phi) = h^i(\psi)$ follows from elementwise equality by function extensionality. Note that since in homotopy type theory, the cohomology modules are sets, equalities between elements of these sets are unique, too.

A functor $F: R\mathrm{-Mod} \rightarrow R\mathrm{-Mod}$ is \textit{exact} if for all short exact sequences as in Ex.~\ref{mod-complex-exm},
\[ 0 \rightarrow F(M^\prime) \stackrel{F(f)}{\rightarrow} F(M) \stackrel{F(g)}{\rightarrow} F(M^{\prime\prime}) \rightarrow 0 \]
is also exact. It is \textit{left exact} if
\[ 0 \rightarrow F(M^\prime) \stackrel{F(f)}{\rightarrow} F(M) \stackrel{F(g)}{\rightarrow} F(M^{\prime\prime}) \]
is exact, and \textit{right exact} if
\[ F(M^\prime) \stackrel{F(f)}{\rightarrow} F(M) \stackrel{F(g)}{\rightarrow} F(M^{\prime\prime}) \rightarrow 0 \]
is exact.

\begin{exm} \label{hom-left-ex-ex}
The functor $\mathrm{Hom}_R(M, -)$, where $M$ is an $R$-module, is left exact: Let $0 \rightarrow N^\prime \stackrel{i}{\rightarrow} N \stackrel{q}{\rightarrow} N^{\prime\prime} \rightarrow 0$ be a short exact sequence of $R$-modules. Then $\mathrm{Hom}_R(M, N^\prime) \rightarrow \mathrm{Hom}_R(M, N)$ composing $h \in \mathrm{Hom}_R(M, N^\prime)$ with $i$ is a monomorphism, too, and if $h \in \mathrm{Hom}_R(M, N)$ composed with $q$ is the 0-homomorphism, then $\mathrm{im}(h) \subseteq \ker(q) = N^\prime$, hence $h$ is in the image of the homomorphism $\mathrm{Hom}_R(M, N^\prime) \rightarrow \mathrm{Hom}_R(M, N)$. 
\end{exm}

\begin{exm} \label{tensor-right-ex-ex}
Prop.~\ref{tensor-right-ex-prop}(d) shows that the functor $- \otimes_R M$ is right exact, for any $R$-module $M$.
\end{exm}

A \textit{short exact sequence of complexes}
\[ 0 \rightarrow A^\bullet \stackrel{\phi}{\rightarrow} B^\bullet \stackrel{\psi}{\rightarrow} C^\bullet \rightarrow 0 \]
consists of short exact sequences
\[ 0 \rightarrow A^i \stackrel{\phi_i}{\rightarrow} B^i \stackrel{\psi_i}{\rightarrow} C^i \rightarrow 0 \]
commuting with the coboundary maps $f_i$, $g_i$ and $h_i$ of the complexes $A^\bullet$, $B^\bullet$ and $C^\bullet$, respectively.

To such a sequence of complexes we can associate $R$-module homomorphisms 
\[ \delta^i: h^i(C^\bullet) \rightarrow h^{i-1}(A^\bullet) \] 
for all $i$ such that we obtain a \textit{long exact sequence} of cohomology modules:
\[ \cdots \rightarrow h^i(A^\bullet) \stackrel{h^i(\phi)}{\rightarrow} h^i(B^\bullet) \stackrel{h^i(\psi)}{\rightarrow} h^i(C^\bullet) 
             \stackrel{\delta^i}{\rightarrow} h^{i-1}(A^\bullet) \rightarrow \cdots .\]
We can construct these \textit{boundary maps} $\delta^i$ using the following diagram chase, with some twists necessary in Homotopy Type Theory: If $c$ represents a cohomology class in $h^i(C^\bullet)$, then for all $b \in B^i$ such that $\psi_i(b) = c$, the value $g_i(b) \in B^{i-1}$ is mapped to $0$ by $\psi^{i-1}$, hence there is a unique $a \in A^{i-1}$ such that $\phi^{i-1}(a) = g_i(b)$. A different choice $b^\prime$ leads to an element $a^\prime \in A^{i-1}$ differing by $f_i(b - b^\prime)$ from $a$ (where $b - b^\prime \in A^i$ because $\psi^i(b-b^\prime) = 0$), hence representing the same cohomology class of $h^{i-1}(A^\bullet)$ as $a$. Thus, we can construct the cohomology class represented by $a$ from the mere existence of $b$ with $\psi^i(b) = c$, without explicitly constructing $b$ from a specific $c$. Different choices of $c$ also lead to elements of $A^{i-1}$ representing the same cohomology class of $h^{i-1}(A^\bullet)$ as $a$, and the whole construction respects the $R$-module structure of the cohomology groups. Finally, the resulting long exact sequences commute with homomorphisms on cohomology groups induced by homomorphisms of short exact sequences of $R$-module complexes. 

So far, so ``natural" (or  ``canonical"). But what if a colleague in the office next door in John Conway's imagined math department (see \cite[Sec.6]{Buz24}) loves theoretical jokes and declares that they prefer to choose $-a$ instead of $a$, for each $c \in \ker(h^i)$? After some deliberation we must admit that nothing is wrong with this choice, as we still obtain a long exact sequence. From a Homotopy Type Theory perspective, we could simply state that there \textit{merely exist} boundary maps that produce a long exact sequence of cohomology modules functorially from short exact sequences; the ``merely" not indicating that we cannot explicitly construct them but that we have several such constructions. And this suffices as long as we only aim to prove propositions, by induction on truncation. On the other hand, we actually \textit{have} explicitly constructed the boundary maps, and the explicit results of such a construction may be useful in other calculations. So it can pay off to record the choices made for the construction, but the results will not depend on them if they are propositions.

\subsection{$\delta$-Functors}

The choices that can be made when constructing boundary maps between cohomology objects, as discussed above and (more obviously) in group cohomology as presented by Buzzard, is the reason why we need to introduce the data of boundary maps into $\delta$-functors devised by Grothendieck to define cohomology.

\begin{Def}
A \textbf{$\delta$-functor} from $R\mathrm{-Mod}$ to $R\mathrm{-Mod}$ is a collection of additive functors $T^i: R\mathrm{-Mod} \rightarrow R\mathrm{-Mod}$, together with module homomorphisms $\delta^{i+1}: T^{i+1}(M^{\prime\prime}) \rightarrow T^i(M^\prime)$ for each short exact sequence $0 \rightarrow M^\prime \rightarrow M \rightarrow M^{\prime\prime} \rightarrow 0$ and $i \geq 0$ such that
\begin{itemize}
\item[(1)] for each short exact sequence as above, there is a long exact sequence
\[ 0 \leftarrow T^0(M^{\prime\prime}) \leftarrow T^0(M) \leftarrow T^0(M^\prime) 
             \stackrel{\delta^1}{\leftarrow} T^1(M^{\prime\prime}) \leftarrow \cdots \]
\[ \cdots \leftarrow T^i(M) \leftarrow T^i(M^{\prime}) 
             \stackrel{\delta^{i+1}}{\leftarrow} T^{i+1}(M^{\prime\prime}) \leftarrow \cdots ;\]
\item[(2)] for each homomorphism from one short exact sequence (as above) to another $0 \rightarrow N^\prime \rightarrow N \rightarrow N^{\prime\prime} \rightarrow 0$ and $i \geq 0$, the homomorphisms $T^i(M^{\prime}) \rightarrow T^i(N^{\prime})$ and $T^{i+1}(M^{\prime\prime}) \rightarrow T^{i+1}(N^{\prime\prime})$ commute with the $\delta^{i+1}$.
\end{itemize} 
\end{Def}

\begin{Def}
A $\delta$-functor $T = (T^i: R\mathrm{-Mod} \rightarrow R\mathrm{-Mod})$ is \textbf{universal} if, given another $\delta$-functor $T^\prime = (T^{\prime i}: R\mathrm{-Mod} \rightarrow R\mathrm{-Mod})$ and a natural transformation $F^{\prime 0}: T^0 \Rightarrow T^0$, there is a unique sequence of natural transformations $F^i: T^{\prime i} \Rightarrow T^i$ for $i \geq 0$ starting with $F^0$ and commuting with the $\delta^{i+1}$ on each short exact sequence.
\end{Def}

If $T$ and $T^\prime$ are two universal $\delta$-functors sharing the same (right-exact) additive functor $F = T^0 = T^{\prime 0}$ as their zeroth component the universal yoga will provide us with \textit{natural equivalences} between the functors $T^i$ and $T^{\prime i}$ for all $i > 0$. In particular, the $R$-modules $T^i(M)$ and $T^{\prime i}(M)$ will be isomorphic, and $\mathrm{Hom}_R(T^i(M),T^i(N))$ will be in bijection to $\mathrm{Hom}_R(T^{\prime i}(M),T^{\prime i}(N))$, for all $R$-modules $M, N$ and $i \geq 0$. As in the end, a functor $R\mathrm{-Mod} \rightarrow R\mathrm{-Mod}$ consists of a function on the type of $R$-modules to itself, a family of functions between hom-sets and equalities between homomorphisms in $R\mathrm{-Mod}$, in homotopy type theory these natural equivalences yield \textit{equalities} $T^i = T^{\prime i}$ via function extensionality and univalence (and the fact that equalities between homomorphism are unique since hom-sets are sets). In turn, these equalities together with the commutativity of the functors with the boundary maps show equality of the $\delta^{i+1}$, too. And finally, as the natural equivalences are uniquely determined by definition of \textit{universal} $\delta$-functors, all these equalities are uniquely determined (that is, \textit{canonical}).

This means that different choices of boundary maps will still lead to equal cohomology theories on $R\mathrm{-Mod}$ as long as the choices are universal. Note that we can use the equality to retrieve the different choices, as univalence is an equivalence. As emphasized above, recording the choices made in a specific construction may help to calculate with cohomology modules, but as long as the result of the calculation is only used to prove a proposition, it is valid independent of the choices. In the end, all these considerations can be seen as a justification of the practice of current standard mathematics, with the caveat of keeping track of choices. 

Another characterisation of $\delta$-universal functors is needed:
\begin{Def}
An additive functor $F: R\mathrm{-Mod} \rightarrow R\mathrm{-Mod}$ is coeffaceable if for each $R$-module $M$ there is an $R$-module $P$ and an epimorphism $u: P \rightarrow M$ such that $F(u) = 0$.
\end{Def}

\begin{thm} \label{eff-univ-delta-thm}
Let $T = (T^i: R\mathrm{-Mod} \rightarrow R\mathrm{-Mod})$ be a $\delta$-functor. If the $T^i$ are coeffaceable functors for each $i > 0$ then $T$ is universal.
\end{thm}
\begin{proof}
We follow the proof of \cite[Thm.XX.7.1]{Lang-Alg}, fill in some gaps as outlined in the lecture notes of Mazur \cite{Maz} and (very slightly) twist the proof to Homotopy Type Theory. 

We first check the uniqueness statement:
Let $T^\prime = (T^{\prime i}: R\mathrm{-Mod} \rightarrow R\mathrm{-Mod})$ be a second $\delta$-functor, and $F: T^{\prime 0} \Rightarrow T^0$ a natural transformation. Suppose we have two sequences of natural transformations $F^i, G^i: T^{\prime i} \Rightarrow T^i$, $i \geq 0$, commuting with the $\delta^i$ for each short exact sequence, such that $F^0 = F = G^0$.

We show $F^n = G^n$ by induction on $n \geq 0$, where $F^0 = G^0$ is already one of our assumptions. Suppose that $F^k = G^k$ for all $0 \leq k < n$. We want to prove that $F^n(M) = G^n(M): T^{\prime n}(M) \rightarrow T^n(M)$ for all $R$-modules $M$. This equality is a proposition, hence we can choose a coeffacing epimorphism $u: P \twoheadrightarrow M$ from an $R$-module $P$. The short exact sequence
\[ 0 \rightarrow \mathrm{ker}(u) =: K \stackrel{i}{\rightarrow} P \stackrel{u}{\twoheadrightarrow} M \rightarrow 0 \]
leads to the following commutative diagram with exact rows, as part of the associated long exact sequence:

\begin{center}
\begin{tikzpicture}
  \matrix (m) [matrix of math nodes,row sep=3em,column sep=4em,minimum width=2em]
  { T^{\prime n-1}(K) & T^{\prime n}(M) & T^{\prime n}(P)  \\
     T^{n-1}(K) & T^n(M) & T^n(P). \\};
  \path[-stealth]
    (m-1-1) edge node [left] {\tiny $F^{n-1}(K) = G^{n-1}(K)$} (m-2-1)
    (m-1-2) edge node [above] {\tiny $\delta^{n-1}$} (m-1-1)
    (m-2-2) edge node [below] {\tiny $\delta^{n-1}$} (m-2-1)
    (m-1-2) edge [transform canvas={xshift=-1.5mm}] node [left] {\tiny $F^n(M)$} (m-2-2)
            edge [transform canvas={xshift=1.5mm}] node [right] {\tiny $G^n(M)$} (m-2-2)
    (m-1-3) edge node [above] {\tiny $T^{\prime n}(u)$} (m-1-2)
    (m-2-3) edge node [below] {\tiny $T^n(u)$} (m-2-2);
\end{tikzpicture}
\end{center}

Commutativity shows that
\[ \delta^{n-1} \circ F^n(M) = F^{n-1}(K) \circ \delta^{n-1} = G^{n-1}(K) \circ \delta^{n-1} = \delta^{n-1} \circ G^n(M), \]
and $T^n(u) = 0$ implies that $\delta^{n-1}$ in the lower row is a monomorphism. Hence $F^n(M) = G^n(M)$.

Next, we show the existence of natural transformations $F^n: T^{\prime n} \Rightarrow T^n$ as requested for given $T^\prime$ and $F = F^0$. Suppose that we have already constructed all the $F^k$ for $k < n$, commuting with the $\delta^{k-1}$. We can use the uniqueness shown in the first part to choose a coeffacing monomorphism $u_M: P_M \twoheadrightarrow M$ from an $R$-module $P_M$ for each $R$-module $M$, by induction on propositional truncation. Considering the same exact sequence as above we obtain the commutative diagram with exact rows

\begin{center}
\begin{tikzpicture}
  \matrix (m) [matrix of math nodes,row sep=3em,column sep=4em,minimum width=2em]
  {  T^{\prime n-1}(P) & T^{\prime n-1}(K) & T^{\prime n}(M) &   \\
     T^{n-1}(P) & T^{n-1}(K) & T^n(M) & 0 \\};
  \path[-stealth]
    (m-1-1) edge node [left] {\tiny $F^{n-1}(P)$} (m-2-1)
    (m-1-2) edge node [above] {\tiny $T^{\prime n-1}(i)$} (m-1-1)
    (m-2-2) edge node [below] {\tiny $T^{n-1}(i)$} (m-2-1)
    (m-1-2) edge node [left] {\tiny $F^{n-1}(K)$} (m-2-2)
    (m-1-3) edge node [above] {\tiny $\delta^{n-1}$} (m-1-2)
    (m-2-3) edge node [below] {\tiny $\delta^{n-1}$} (m-2-2)
    (m-1-3) edge [dashed] node [left] {\tiny $F_{u_M}$} (m-2-3)
    (m-2-4) edge (m-2-3);
\end{tikzpicture}
\end{center}

Then 
\[ T^{n-1}(i) \circ F^{n-1}(K) \circ \delta^{n-1} = F^{n-1}(P) \circ T^{\prime n-1}(i) \circ \delta^{n-1} = 0 \circ F^{n-1}(P) = 0, \]
hence $\mathrm{im}(F^{n-1}(K) \circ \delta^{n-1}) \subseteq \ker(T^{n-1}(i)) = \mathrm{im}(\delta^{n-1})$, and we have a unique homomorphism $F_{u_M}$ commuting the diagram above because the homomorphism $\delta^{n-1} : T^n(M) \rightarrow T^{n-1}(K)$ is a monomorphism.

The $F_{u_M}$ form a natural transformation $T^{\prime n} \Rightarrow T^n$: Let $f: M \rightarrow N$ be an $R$-module homomorphism, $u_M: P_M \rightarrow M$ and $u_N: P_N \rightarrow N$ two effacing epimorphisms with kernels $K_M$ and $K_N$, respectively. First, assume that we have a commutative diagram of short exact sequences:

\begin{center}
\begin{tikzpicture}
  \matrix (m) [matrix of math nodes,row sep=3em,column sep=4em,minimum width=2em]
  {
     0 & K_M & P_M & M & 0 \\
     0 & K_N & P_N & N  & 0. \\};
  \path[-stealth]
    (m-1-1) edge (m-1-2)
    (m-2-1) edge (m-2-2)
    (m-1-2) edge (m-1-3)
                edge node [left] {\tiny $g$} (m-2-2)
    (m-2-2) edge (m-2-3)
    (m-1-3) edge node [above] {\tiny $u_M$} (m-1-4)
                edge (m-2-3)
    (m-2-3) edge node [below] {\tiny $u_N$} (m-2-4)
    (m-1-4) edge (m-1-5)
                edge node [right] {\tiny $f$} (m-2-4)
    (m-2-4) edge (m-2-5);
\end{tikzpicture}
\end{center}

From this we obtain a cube diagram

\begin{center}
\begin{tikzpicture}
  \matrix (m) [matrix of math nodes, row sep=3em, column sep=3em]{
    & T^{n-1}(K_M) & & T^n(M) \\
    T^{\prime n-1}(K_M) & & T^{\prime n}(M) & \\
    & T^{n-1}(K_N) & & T^n(N) \\
    T^{\prime n-1}(K_N) & & T^{\prime n}(N), & \\};
  \path[-stealth]
    (m-1-4) edge node [above] {\tiny $\delta^{n-1}$} (m-1-2) 
    (m-2-1) edge node [left,xshift=-1ex] {\tiny $F^{n-1}(K_M)$} (m-1-2)
    (m-1-2) edge node [right,yshift=3ex] {\tiny $T^{n-1}(g)$} (m-3-2)
    (m-1-4) edge node [right] {\tiny $T^n(f)$} (m-3-4) 
    (m-2-3) edge node [right,xshift=1ex] {\tiny $F_{u_M}$} (m-1-4)
    (m-2-3) edge node [above,xshift=-3ex] {\tiny $\delta^{n-1}$} (m-2-1) 
    (m-2-1) edge node [right] {\tiny $T^{\prime n-1}(g)$} (m-4-1)
    (m-2-3) edge node [right,yshift=3ex] {\tiny $T^{\prime n}(f)$} (m-4-3)
    (m-3-4) edge node [above,xshift=-3ex] {\tiny $\delta^{n-1}$} (m-3-2)
    (m-4-1) edge node [right,xshift=1ex] {\tiny $F^{n-1}(K_N)$} (m-3-2)
    (m-4-3) edge node [right,xshift=1ex] {\tiny $F_{u_N}$} (m-3-4)
    (m-4-3) edge node [above] {\tiny $\delta^{n-1}$} (m-4-1);
\end{tikzpicture}
\end{center}

\noindent whose faces but possibly the right-hand one commute by what we assume or have previously proven. Since the $\delta^{n-1}$ in the diagram are monomorphisms,
\[ \delta^{n-1} \circ T^n(f) \circ F_{u_M} = \delta^{n-1} \circ F_{u_N} \circ T^{\prime n}(f) \]
implies the commutativity of the right-hand side square. But this equality follows by taking other commuting paths in the cube.

In general, we cannot expect $f$ to lift to a homomorphism $P_M \rightarrow P_N$ which then automatically restricts to a homomorphism $K_M \rightarrow K_N$. But we can use an intermediate step, using that $u_M + (u_N \circ f)$ also is an epimorphism: As above, we can complete the two right-hand vertical homomorphisms in the following diagram to make it commutative:

\begin{center}
\begin{tikzpicture}
  \matrix (m) [matrix of math nodes,row sep=3em,column sep=4em,minimum width=2em]
  {
     0 & M & P_M & K_M & 0 \\
     0 & M & P_M \oplus P_N & Q & 0 \\
     0 & N & P_N & K_N  & 0. \\};
  \path[-stealth]
    (m-1-2) edge (m-1-1)
    (m-2-2) edge (m-2-1)
    (m-3-2) edge (m-3-1)
    (m-1-3) edge node [above] {\tiny $u_M$} (m-1-2)
    (m-2-2) edge node [left] {\tiny $\mathrm{id}$} (m-1-2)
                edge node [left] {\tiny $f$} (m-3-2)
    (m-2-3) edge node [above] {\tiny $u_M + (u_N \circ f)$} (m-2-2)
    (m-3-3) edge node [above] {\tiny $u_N$} (m-3-2)
    (m-1-4) edge (m-1-3)
    (m-2-4) edge (m-2-3)
    (m-2-3) edge node [right] {\tiny $p_M$} (m-1-3)
                edge node [right] {\tiny $p_N$} (m-3-3)
    (m-3-4) edge (m-3-3)
    (m-1-5) edge (m-1-4)
    (m-2-5) edge (m-2-4)
    (m-2-4) edge (m-1-4)
                edge (m-3-4)
    (m-3-5) edge (m-3-4);
\end{tikzpicture}
\end{center}

Since $T^n$ is additive,
\[ T^n(u_M + (u_N \circ f)) = T^n(u_M) + T^n(u_N)T^n(f) = 0, \]
so $u_M + (u_N \circ f)$ is a coeffacing epimorphism for $T^n$. Applying the previous arguments to the top two and to the bottom two rows of this diagram, we obtain 
\[ F_{u_M + (u_N \circ f)} \circ T^n(\mathrm{id}) = T^{\prime n}(\mathrm{id}) \circ F_{u_M}\ \mathrm{and\ }
  F_{u_M + (u_N \circ f)}  \circ T^n(f) = T^{\prime n}(f) \circ F_{u_N}, \]
so $F_{u_M}  \circ T^n(f) = T^{\prime n}(f) \circ F_{u_N}$, which shows the naturality of $F^n$.

Furthermore, choosing $N := M$ and $f := \mathrm{id}$, this also shows the independence of $F_u$ from the coeffacing homomorphism $u: P \rightarrow M$. 

Finally, we show that the $F^n$ commute with the $\delta^{n-1}$. This is a proposition, so we can freely use induction on propositional truncation. Start with a short exact sequence $0 \rightarrow M^\prime \stackrel{i}{\rightarrow} M \stackrel{j}{\rightarrow} M^{\prime\prime} \rightarrow 0$ of $R$-modules. If $j$ is a coeffacing epimorphism, the diagram

\begin{center}
\begin{tikzpicture}
  \matrix (m) [matrix of math nodes,row sep=3em,column sep=4em,minimum width=2em]
  {  T^{\prime n-1}(M^{\prime\prime}) & T^{\prime n}(M^\prime)  \\
     T^{n-1}(M^{\prime\prime}) & T^n(M^\prime). \\};
  \path[-stealth]
    (m-1-1) edge node [left] {\tiny $F^{n-1}(M^{\prime\prime})$} (m-2-1)
    (m-1-2) edge node [above] {\tiny $\delta^{n-1}$} (m-1-1)
    (m-2-2) edge node [below] {\tiny $\delta^{n-1}$} (m-2-1)
    (m-1-2) edge node [right] {\tiny $F^n(M^\prime)$} (m-2-2);
\end{tikzpicture}
\end{center}

\noindent commutes by construction of $F^n(M^\prime) := F_j$ as above. 

If $j$ is not coeffacing, let $u: P_M \rightarrow M$ be a coeffacing epimorphism. Then $j \circ u: P_M \rightarrow M^{\prime\prime}$ is also a coeffacing epimorphism, and we have a commutative diagram

\begin{center}
\begin{tikzpicture}
  \matrix (m) [matrix of math nodes,row sep=3em,column sep=4em,minimum width=2em]
  {  0 & K_{M^{\prime\prime}} & P_M & M^{\prime\prime}  & 0 \\
     0 & M^\prime & M & M^{\prime\prime} & 0. \\};
  \path[-stealth]
    (m-1-1) edge (m-1-2)
    (m-2-1) edge (m-2-2)
    (m-1-2) edge  (m-1-3)
                edge node [left] {\tiny $w$} (m-2-2)
    (m-2-2) edge node [below] {\tiny $i$}  (m-2-3)
    (m-1-3) edge node [below] {\tiny $j \circ u$} (m-1-4)
                edge node [left] {\tiny $u$} (m-2-3)
    (m-2-3) edge node [above] {\tiny $j$} (m-2-4)
    (m-1-4) edge (m-1-5)
                edge node [left] {\tiny $\mathrm{id}$} (m-2-4)
    (m-2-4) edge (m-2-5);
\end{tikzpicture}
\end{center}

This leads to the cube

\begin{center}
\begin{tikzpicture}
  \matrix (m) [matrix of math nodes, row sep=3em, column sep=3em]{
    & T^{n-1}(M^\prime) & & T^n(M^{\prime\prime}) \\
    T^{\prime n-1}(M^\prime) & & T^{\prime n}(M^{\prime\prime}) & \\
    & T^{n-1}(K_{M^{\prime\prime}}) & & T^n(M^{\prime\prime}) \\
    T^{\prime n-1}(K_{M^{\prime\prime}}) & & T^{\prime n}(M^{\prime\prime}), & \\};
  \path[-stealth]
    (m-1-4) edge node [above] {\tiny $\delta^{n-1}$} (m-1-2) 
    (m-2-1) edge node [left,yshift=1ex] {\tiny $F^{n-1}(M^\prime)$} (m-1-2)
    (m-3-2) edge node [right,yshift=3ex] {\tiny $T^{n-1}(w)$} (m-1-2)
    (m-3-4) edge node [right] {\tiny $T^n(\mathrm{id})$} (m-1-4) 
    (m-2-3) edge node [right,xshift=1ex] {\tiny $F^n(M^{\prime\prime})$} (m-1-4)
    (m-2-3) edge node [above,xshift=-3ex] {\tiny $\delta^{n-1}$} (m-2-1) 
    (m-4-1) edge node [left] {\tiny $T^{\prime n-1}(w)$} (m-2-1)
    (m-4-3) edge node [right,yshift=3ex] {\tiny $T^{\prime n}(\mathrm{id})$} (m-2-3)
    (m-3-4) edge node [above,xshift=-3ex] {\tiny $\delta^{n-1}$} (m-3-2)
    (m-4-1) edge node [right,xshift=1ex] {\tiny $F^{n-1}(K_{M^{\prime\prime}})$} (m-3-2)
    (m-4-3) edge node [right,xshift=1ex] {\tiny $F^n(M^{\prime\prime})$} (m-3-4)
    (m-4-3) edge node [above] {\tiny $\delta^{n-1}$} (m-4-1);
\end{tikzpicture}
\end{center}

\noindent where all the faces but the top one are commutative by what we have assumed or shown before. The commutativity of the top square, and hence the commuting of $F^{n-1}$, $F^n$ and the $\delta^{n-1}$, follows by extending the paths $F^{n-1}(M^\prime) \circ \delta^{n-1}$ and $\delta^{n-1} \circ F^n(M^{\prime\prime})$ with $T^{\prime n}(\mathrm{id}) = \mathrm{id}_{T^{\prime n}(M^{\prime\prime})}$, and replacing the resulting paths with commuting paths in the cube. 
\end{proof}

Effaceability is actually a necessary condition for universality of $\delta$-functors, too. This is an immediate consequence of the uniqueness of universal $\delta$-functors, once we have constructed one such $\delta$-functor with effaceable functors in Sec.~\ref{der-funct-ssec}.

\subsection{Projective $R$-modules and Projective Resolutions} \label{proj-mod-res-ssec}

To construct universal $\delta$-functors from $R\mathrm{-Mod}$ to itself, one standard way is to use \textit{projective resolutions} of $R$-modules by \textit{projective $R$-modules} and apply the machinery of homological algebra developed above on it.

\begin{Def}
An $R$-module $P$ is \textbf{projective} if for all module epimorphisms $p : N \twoheadrightarrow M$ and module homomorphisms $f: P \rightarrow M$ there merely is a module homomorphism $\overline{f}: P \rightarrow N$ such that $f = p \circ \overline{f}$.
\end{Def}  

Projective resolutions of $R$-modules exist because $R\mathrm{-Mod}$ has \textit{enough projectives}, that is, to every $R$-module there merely exists an epimorphism from a projective $R$-module. This is an immediate consequence of the following fact and the existence of a free presentation for each $R$-module, as discussed in Ex.~\ref{free-pres-exm}.
\begin{prop}
A free module $\bigoplus_{i \in I} R \cdot e_i$ is a projective module.
\end{prop}
\begin{proof}
Let $p : N \twoheadrightarrow M$ be an $R$-module epimorphism and $f : \bigoplus_{i \in I} R \cdot e_i \rightarrow M$ any $R$-module homomorphism. The property of being projective is a proposition, hence for each $i \in I$, we can find $n_i \in N$ such that $p(n_i) = f(e_i)$. Then the universal property of a free module yields a (unique) $R$-module homomorphism $\overline{f} : \bigoplus_{i \in I} R \cdot e_i \rightarrow N$ given by $e_i \mapsto n_i$, and $f = p \circ \overline{f}$, as requested. 
\end{proof}

\begin{cor}
For each $R$-module $M$ there merely exists a \textbf{projective resolution}, that is, a complex of $R$-modules
\[ P^\bullet_M: 0 \leftarrow P^0_M \stackrel{d_M^1}{\leftarrow} P^1_M \stackrel{d_M^2}{\leftarrow} \cdots \stackrel{d_M^n}{\leftarrow} P^n_M \leftarrow \cdots \]
everywhere exact but at $P_M^0$, such that $\mathrm{coker}(d_M^1) = M$ and all the $R$-modules $P^i_M$, $i \geq 0$, are projective.
\end{cor}
\begin{proof}
By extending a free presentation of $M$ in the same way as we constructed the free presentation itself, we obtain the desired projective resolution, as a free resolution. 
\end{proof}

The mere existence once again indicates that there will be many different projective resolutions for the same $R$-module $M$. The situation, however, is worse than before: there is no explicit construction of projective resolutions working in all cases. For example, in the proof above the construction depends on choosing preimages under epimorphisms, which cannot be done explicitly in all cases.

Similarly, it is possible to lift an $R$-module homomorphism $f : M \rightarrow N$ to a homomorphism of projective resolutions $f^\bullet : P^\bullet_M \rightarrow P^\bullet_N$, but only in a highly non-unique way. At least, two such liftings of the same homomorphism to the same projective resolutions are homotopic --- this will lead to the uniqueness of derived functors. Also, a short exact sequence of $R$-modules can be lifted to a short exact sequence of projective resolutions --- this allows to construct boundary maps for the derived functors.
\begin{prop} \label{proj-res-lift-prop}
Let $f: M \rightarrow N$ be a homomorphisms of $R$-modules $M, N$ and $P^\bullet_M, P^\bullet_N$ projective resolutions of $M$ and $N$, respectively. Then $f$ merely lifts to a homomorphism $f^\bullet: P^\bullet_M \rightarrow P^\bullet_N$ of $R$-module complexes, that is, we have a commutative diagram

\begin{center}
\begin{tikzpicture}
  \matrix (m) [matrix of math nodes,row sep=3em,column sep=4em,minimum width=2em]
  {
     0 & M & P^0_M & P^1_M & \cdots \\
     0 & N & P^0_N & P^1_N  & \cdots. \\};
  \path[-stealth]
    (m-1-2) edge (m-1-1)
    (m-2-2) edge (m-2-1)
    (m-1-3) edge (m-1-2)
    (m-1-2) edge node [left] {\tiny $f$} (m-2-2)
    (m-2-3) edge (m-2-2)
    (m-1-4) edge node[above] {\tiny $d_M^1$} (m-1-3)
    (m-1-3) edge node [left] {\tiny $f^0$}(m-2-3)
    (m-2-4) edge node[above] {\tiny $d_N^1$} (m-2-3)
    (m-1-5) edge node[above] {\tiny $d_M^2$} (m-1-4)
    (m-1-4) edge node [left] {\tiny $f^1$} (m-2-4)
    (m-2-5) edge node[above] {\tiny $d_N^2$} (m-2-4);
\end{tikzpicture}
\end{center}

Furthermore, two liftings $f_1^\bullet, f_2^\bullet: P_M^\bullet \rightarrow P_N^\bullet$ of $f$ to the same projective resolutions are homotopic.

Finally, if $0 \rightarrow M^\prime \stackrel{f}{\rightarrow} M \stackrel{g}{\rightarrow} M^{\prime\prime} \rightarrow 0$ is a short exact sequence of $R$-modules, then it merely lifts to a short exact sequence
\[ 0 \rightarrow P_{M^\prime}^\bullet \stackrel{f^\bullet}{\rightarrow} P_M^\bullet \stackrel{g^\bullet}{\rightarrow} P_{M^{\prime\prime}}^\bullet \rightarrow 0 \]
of suitably chosen projective resolutions of $M^\prime$, $M$ and $M^{\prime\prime}$, where $f^\bullet$ and $g^\bullet$ lift $f$ and $g$.
\end{prop}
\begin{proof}
The (mere) existence of a lifting $f^\bullet: P^\bullet_M \rightarrow P^\bullet_N$ follows from the universal property of projective $R$-modules and a split up of the projective resolutions into short exact sequences $0 \leftarrow M \leftarrow P^0_M \leftarrow K^0_M \leftarrow 0$, $0 \leftarrow K^0_M \leftarrow P^1_M \leftarrow P^1_M \leftarrow 0$, \ldots\ and $0 \leftarrow N \leftarrow P^0_N \leftarrow K^0_N \leftarrow 0$, $0 \leftarrow K^0_N \leftarrow P^1_N \leftarrow K^1_N \leftarrow 0$, \ldots, respectively.

If $f_1^\bullet, f_2^\bullet: P_M^\bullet \rightarrow P_N^\bullet$ are two liftings of $f$, then $f_1^\bullet - f_2^\bullet: P_M^\bullet \rightarrow P_N^\bullet$ is a lifting of the zero map $o: M \rightarrow N$, and we want to show that this lifting is homotopic to $0$. First, $f^0-g^0$ maps into $\mathrm{ker}(d_N^0) = \mathrm{im}(d_N^1)$, so the projectivity of $P_M^0$ yields a homomorphism $k^0 : P_M^0 \rightarrow P_N^1$ such that $d_N^1 \circ k^0 = f^0-g^0$. In the next step, note that
\[ d_N^1 \circ ((f^1 - g^1) - k^0 \circ d_M^1) = d_N^1 \circ (f^1 - g^1) - (f^0-g^0) \circ d_M^1 = 0, \]
hence $(f^1 - g^1) - k^0 \circ d_M^1$ maps $P_M^1$ into $\mathrm{ker}(d_N^1) = \mathrm{im}(d_N^2)$. Again, the projectivity of $P_M^1$ yields a homomorphism $k^1 : P_M^1 \rightarrow P_N^2$ such that
\[ f^1 - g^1 = k^0 \circ d_M^1 + d_N^2 \circ k^1. \]
In that way, we can construct inductively the requested homotopy.

The construction of a short exact sequence of complexes lifting a short exact sequence $0 \rightarrow M^\prime \stackrel{f}{\rightarrow} M \stackrel{g}{\rightarrow} M^{\prime\prime} \rightarrow 0$ starts with arbitrary projective resolutions $P_{M^\prime}^\bullet$ and $P_{M^{\prime\prime}}^\bullet$ of $M^\prime$ and $M^{\prime\prime}$. Then we can set $P_M^n := P_{M^\prime}^n \oplus P_{M^{\prime\prime}}^n$, obviously a projective module for all $n \geq 0$, and construct the differential $d^n_M : P_M^n \rightarrow P_M^{n-1}$ inductively as follows: For $n = 0$, we take the direct sum of the composition of the epimorphism $P_{M^\prime}^0 \rightarrow M^\prime$ and the inclusion $M^\prime \hookrightarrow M$ with the lifting of $P_{M^{\prime\prime}}^0 \rightarrow M$ through the epimorphism $M \rightarrow M^{\prime\prime}$, possible because $P_{M^{\prime\prime}}^0$ is projective. It is straightforward to show that the resulting homomorphism $P_{M^\prime}^0 \oplus P_{M^{\prime\prime}}^0 \rightarrow M$ is surjective. In the next step we replace $M^\prime$, $M$ and $M^{\prime\prime}$ with the kernels of the epimorphisms $P_{M^\prime}^0 \rightarrow M^\prime$,  $P_M^0 \rightarrow M$ and $P_{M^{\prime\prime}}^0 \rightarrow M^{\prime\prime}$, and proceed as before. This can be continued inductively.  
\end{proof}

\subsection{Derived Functors} \label{der-funct-ssec}

We finally construct a universal $\delta$-functor 
\[ T = (T^n: R\mathrm{-Mod} \rightarrow R\mathrm{-Mod}) \] 
extending a right-exact additive functor $F = T^0: R\mathrm{-Mod} \rightarrow R\mathrm{-Mod}$, as the \textit{left derived functors} $L^n F$, $n \geq 0$, of $F$. Since such a universal $\delta$-functor is unique, we can freely use induction on propositional truncation, that is, make choices of merely existing objects.

In particular, we can choose a projective resolution
\[ P^\bullet_M: 0 \leftarrow M \leftarrow P^0_M \leftarrow \cdots \leftarrow P^n_M \leftarrow P^{n+1}_M \rightarrow \cdots  \]
for each $R$-module $M$. Setting $L^n F(M) := h^n(F(P^\bullet_M))$ and $L^n F(f) := h^n(F(P^\bullet_f))$ as constructed in homological algebra (see Sec.~\ref{hom-alg-ssec}), a straightforward check shows that $L^n F$ is an additive functor. The homotopy statement in Prop.~\ref{proj-res-lift-prop} shows the independence of the construction of $L^n F$ of the choice of projective resolutions, up to \textit{canonical}, that is, unique isomorphism or equivalently, equality.

The right-exactness of $F$ shows that $L^0 F = F$. As in homological algebra, there exist boundary morphisms $\delta^n: L^n F(M^{\prime\prime}) \rightarrow L^{n-1} F(M^\prime)$ closing the gaps in long exact sequences associated to short exact sequence $0 \rightarrow M^\prime \rightarrow M \rightarrow M^{\prime\prime} \rightarrow 0$ and commmuting with homomorphisms $L^n F(f)$ and $L^{n-1} F(f)$. 

All this shows: The $L^n F$, $n \geq 0$, are a $\delta$-functor.
\begin{thm}
The $L^n F$, $n \geq 0$, are a universal $\delta$-functor.
\end{thm}
\begin{proof}
This follows from Thm.~\ref{eff-univ-delta-thm} because the $L^n F$ are coeffaceable functors: Projective $R$-modules $P$ have a projective resolution $0 \leftarrow P \leftarrow P \leftarrow 0$, hence $L^n F(P) = 0$ for $n > 0$, and every $R$-module has an epimorphism from a projective module.
\end{proof}

\section{Calculating with Cohomology Classes} \label{calc-cohom-sec}
 
The constructions and considerations of the previous section show that in HoTT, ``calculations" with arbitrary cohomology groups will be different from calculations with numbers or more advanced algebraic objects like polynomials: The cohomology classes with which we calculate are usually chosen by induction on propositional truncation, and therefore we can only show proper theorems (propositions in HoTT terminology). In particular, for such purposes cohomology classes need not be explicitly constructed, but in special cases, an explicit construction, or even a particular choice of a construction, may nevertheless be possible and useful, without the theorem depending on the choice.

A typical application of cohomology groups is to deduce properties from their \textit{vanishing} -- generations of students have wondered why to introduce such a complicated machinery just to show that it produces nothing else than $0$. We present one instance of such an application by proving a criterion for flatness, a subtle commutative-algebraic concept capturing how ``smoothly" modules vary over the spectrum of a ring, and use it for a specific example:

\begin{prop}[{\cite[Ex.6.4]{ECA}}] \label{flat-hyper-prop}
Let $R$ be a Noetherian ring, $S = R[x_1, \ldots, x_r]$ the polynomial ring in $r$ variables over $R$ and $f$ a non-zerodivisor. Then $S/(f)$ is a flat $R$-module if, and only if the coefficients of $f$ generate the unit ideal in $R$.
\end{prop}

\begin{cor}
Let $k$ be a field. Then $k[x,y]/(1+xy)$ is a flat $k[x]$-module.
\end{cor}
\begin{proof}
Apply Prop.~\ref{flat-hyper-prop} on $R = k[x]$, $S = k[x,y]$ and $f = 1 + xy$: The cofficients of $f$ in $k[x]$ are $1$ and $x$ and generate the unit ideal in $k[x]$.
\end{proof}

To prove the proposition we need to introduce some further commutative-algebraic concepts and some of their properties. The material can be found in standard monographs (e.g. \cite{Hart:AG} or \cite{ECA}); the fact that our presentation looks very similar to theirs actually sends an encouraging message: Formalisation of these topics in the framework of HoTT does not require a complete redevelopment but can follow well-established patterns of mathematical practice put on a solid logical foundation.

We assume that every ring $R$ in this section is commutative with one.

\subsection{Noetherian rings and modules}

A ring $R$ is called \textit{Noetherian} if all the ideals of $R$ are finitely generated. Thus, the Noetherian property is a special case of the Axiom of Choice: it allows to pick finitely many generators of a given ideal -- but without providing a recipe how to do so, that is, \textit{merely}.

There are equivalent definitions:
\begin{prop} \label{Noether-submodule-prop}
Let $R$ be a Noetherian ring. Then submodules of a finitely generated $R$-module $M$ are finitely generated.
\end{prop}
\begin{proof}
As $M$ is a finitely generated $R$-module there exists a surjective $R$-module homomorphism $R^n \rightarrow M$ for some $n \in \mathbb{N}$. Hence it is enough to show that every submodule of $R^n$ is finitely generated.

This follows from $R$ Noetherian by induction on $n$, using the fact that an $R$-module $N$ is finitely generated if in a short exact sequence $0 \rightarrow N^\prime \rightarrow N \rightarrow N^{\prime\prime} \rightarrow 0$, the $R$-modules $N^\prime$ and $N^{\prime\prime}$ are finitely generated. But every element of $N$ is the sum of an $R$-linear combination of generators of $N^\prime$ and an $R$-linear combination of chosen preimages of generators of $N^{\prime\prime}$.
\end{proof}
\begin{prop}[Ascending Chain Condition]
Let $R$ be a Noetherian ring and $M$ a finitely generated $R$-module. Then any chain 
\[ N_1 \subseteq N_2 \subseteq \cdots N_i \subseteq \cdots \subseteq M \]
of submodules $N_i$ of $M$ has a \textit{maximal} element, that is, there is a $k \in \mathbb{N}$ such that
\[ N_k = N_{k+1} = \ldots. \]
\end{prop}
\begin{proof}
$N = \bigcup_{i=1}^\infty \subseteq M$ is a submodule, hence finitely generated by Prop.~\ref{Noether-submodule-prop}, say by $n_1, \ldots, n_l \in N$. Each $n_j$ must be contained in some $N_{k_j}$, hence all the $n_j$ are contained in $N_k$ with $k = \max \{k_j : j = 1,\ldots, l\}$. Then $N = \bigoplus R \cdot n_j \subseteq N_k$, hence
\[ N_k = N_{k+1} = \cdots = N, \]
as requested.
\end{proof}

Note that the Ascending Chain Condition only implies finite generatedness of ideals if we can decide whether an ideal $(f_1, \ldots, f_k) \subseteq R$ generated by elements $f_1, \ldots, f_k$ in an ideal $I \subseteq R$ is equal to $I$ or not. If we assume the Law of Excluded Middle this implication of course holds in any case.

\subsection{Polynomial rings over a coefficient ring $R$}

The polynomials 
\[ r_nX^n + r_{n-1}X^{n-1} + \cdots + r_1X + r_0,\ r_n \neq 0, \] 
in a variable $X$ with coefficients $r_n, \ldots, r_0$ in the coefficient ring $R$ form a ring, using the standard addition and multiplication of polynomials. The polynomial has degree $n$ if $r_n \neq 0$, and then $r_nX^n$ is called the initial term with initial coefficient $r_n$. Such a ring satisfies a universal \textit{adjoint} property:
\begin{prop}
Let $f: R \rightarrow S$ be a ring homomorphism. Then the set of ring homomorphisms $\phi: R[X] \rightarrow S$ is in one-to-one correspondence with the elements $s \in S$, associating $s$ to the unique ring homomorphism $R[X] \rightarrow S$ mapping $X$ to $s$ (and $r$ to $f(r)$). \hfill $\Box$
\end{prop}

By induction, we can construct polynomial rings $R[X_1, \ldots, X_n]$ in finitely many variables $X_i$ over the coefficient ring $R$. 

\begin{thm}[Hilbert's Basis Theorem]
If a ring $R$ is Noetherian, then the polynomial ring $R[X]$ is also Noetherian.
\end{thm}
\begin{proof}
Let $I \subseteq R[X]$ be an ideal. We need to show that $I$ is finitely generated.

Let $I_n \subseteq R$ be the ideal of leading coefficients of polynomials of degree $n$ in $I$. Since $R$ is Noetherian, the $I_n$ are finitely generated, say by $r_1^{(n)}, \ldots, r_{k_n}^{(n)} \in R$ which are leading coefficients of $f_1^{(n)}, \ldots, f_{k_n}^{(n)} \in I$. The Ascending Chain Condition implies that there is $N$ such that $\sum_{n=0}^\infty I_n = \sum_{n=0}^N I_n \subseteq R$. Then we can show that the $f_1^{(0)}, \ldots, f_{k_N}^{(N)}$ generate $I$.

We do this by induction on the degree $n$ of a polynomial $f \in I$: If $\deg f = 0$ then $f \in I_0$, hence by construction it is an $R$-linear combination of the $f_i^{(0)} = r_i^{(0)} \in I_0$. If $\deg f = n \leq N$, the leading coefficient $r_n \in R$ of $f$ is in $I_n$, hence is an $R$-linear combination $r_n = \sum_{j=1}^{k_n} u_j r_j^{(n)}$. Then $f - \sum_{j=1}^{k_n} u_j f_j^{(n)}$ has degree $< n$, so the claim follows from the induction hypothesis.

If $\deg f = m > N$, we write the leading coefficient $r_m \in R$ of $f$ as an $R$-linear combination of all the generators $r_j^{(n)}$ of the $I_0, \ldots, I_N$ introduced above:
\[ r_m = \sum_{n=0}^N \sum_{j=1}^{k_n} u_j^{(n)} r_j^{(n)}. \]
Then $f - \sum_{n=0}^N \sum_{j=1}^{k_n} u_j^{(n)} f_j^{(n)} \cdot X^{m - n}$ once again has degree $< \deg f$, and the claim follows by induction. 
\end{proof}

The argument follows Gordan's constructive proof of Hilbert's Basis Theorem (see \cite[Ex.15.15]{ECA}). Hilbert's original proof (see \cite[Thm.1.2]{ECA}) uses the Law of Excluded Middle: It constructs a sequence $f_1, \ldots, f_m, f_{m+1}, \ldots \in R[X]$ adding in each step a new polynomial $f_{m+1}$ if $(f_1, \ldots, f_m) \neq I$. Thus, we need to decide whether $(f_1, \ldots, f_m) = I$ or $(f_1, \ldots, f_m) \neq I$.

\subsection{Tor groups and flat $R$-modules} 

Tensorizing with an $R$-module $M$ is a right-exact functor $- \otimes_R M$, by Prop.~\ref{tensor-prod-ident-prop}(d). Thus, we can construct the left-derived functors of $- \otimes_R M$ as in Sec.~\ref{der-funct-ssec}, and we denote them by $\mathrm{Tor}_n^R(M): R-\mathrm{Mod} \rightarrow R-\mathrm{Mod}$. In particular, for each short exact sequence of $R$-modules $0 \rightarrow N^\prime \rightarrow N \rightarrow N^{\prime\prime} \rightarrow 0$ there is a long exact sequence starting with
\[ 0 \leftarrow N^{\prime\prime} \otimes_R M \leftarrow N \otimes_R M \leftarrow N^\prime \otimes_R M \leftarrow    
   \mathrm{Tor}_1^R(N,M^\prime) \leftarrow \mathrm{Tor}_1^R(N,M) \leftarrow \cdots . \]
\begin{lem} \label{tor-free-zero-lem}
For all $R$-modules $M$ and sets $I$, $\mathrm{Tor}_n^R(R^I,M) = 0$ if $n \geq 1$.
\end{lem}
\begin{proof}
$0 \rightarrow R^I \rightarrow R^I \rightarrow 0$ is a projective resolution of the free $R$-module $R^I$, hence the claim follows by construction of $\mathrm{Tor}_n^R(R^I,M)$.
\end{proof}

We now turn to flat $R$-modules and their properties:
\begin{Def}
An $R$-module $M$ is flat if tensorizing with $M$ is a left-exact functor, that is, for all monomorphisms $\iota: N^\prime \rightarrow N$, the induced $R$-module homomorphism $\iota \otimes M : N^\prime \otimes_R M \rightarrow N \otimes_R M$ also is a monomorphism.
\end{Def}
\begin{prop}[{\cite[Prop.6.1]{ECA}}] \label{Tor-flat-crit}
An $R$-module $M$ is flat if, and only if 
\[(\ast) \hspace{1cm} \mathrm{Tor}_1^R(R/I,M) = 0\ \mathrm{for\ all\ finitely\ generated\ ideals\ } I \subseteq R. \]
\end{prop}
\begin{proof}
Assume first that $M$ is flat. Then $I \otimes_R M \rightarrow R \otimes_R M \cong M$ is injective for each finitely generated ideal $I \subseteq R$, and the injection is part of the long exact sequence
\[0 \leftarrow R/I \otimes_R M \leftarrow R \otimes_R M \leftarrow I \otimes_R M \leftarrow \mathrm{Tor}_1^R(R/I,M) \leftarrow \mathrm{Tor}_1^R(R,M) \leftarrow \cdots . \]
By Lem.~\ref{tor-free-zero-lem}, $\mathrm{Tor}_1^R(R,M) = 0$, so the exactness of the sequence implies $\mathrm{Tor}_1^R(R/I,M) = 0$. 

Conversely,  if $(\ast)$ is satisfied for all finitely generated ideals $I \subseteq R$, it also holds for arbitrary ideals $I^\prime \subseteq R$: An element $0 \neq x^\prime \in I^\prime \otimes_R M$ is a finite sum of elements $r^\prime \otimes m$, $r^\prime \in I^\prime$, $m \in M$. Thus $x^\prime$ is the image of an element $x \in I \otimes_R M$, where $I$ is a finitely generated ideal containing the $r^\prime$, and necessarily $x \neq 0$. By $(\ast)$ it maps to a non-zero element in $R \otimes_R M$, which is also the image of $x^\prime$, as required.

Next, we can restrict the check of injectivity to finitely generated $R$-modules~$N$: As before, the statement that $x \in N^\prime \otimes_R M$ goes to $0 \in N \otimes_R M$ only involves finitely many elements of $N$.

Thus, we can assume that there exists a finite sequence of submodules
\[ N^\prime = N_0 \subseteq N_1 \subseteq \cdots \subseteq  N_k = N \]
such that each $N_{i+1}/N_i$ is generated by one element, that is, is $\cong R/I$ for an ideal~$I \subseteq R$. By induction, we can assume right away that $N/N^\prime \cong R/I$. The short exact sequence $0 \rightarrow N^\prime \rightarrow N \rightarrow N/N^\prime \rightarrow 0$ gives rise to a long exact sequence containing the terms 
\[ \mathrm{Tor}_1^R(N/N^\prime,M) \rightarrow n^\prime \otimes_R M \rightarrow N \otimes_R M. \]
Since $\mathrm{Tor}_1^R(N/N^\prime,M) \cong \mathrm{Tor}_1^R(R/I,M) = 0$, we are done.
\end{proof}

\subsection{The Flatness Criterion}

We now prove Prop.~\ref{flat-hyper-prop}.

Let $I \subseteq R$ be the ideal generated by the coefficients of $f$. Let $\mu_f : S \rightarrow S$ be the $R$-module homomorphism given by multiplication with $f$. Since $f$ is a non-zero divisor, $\mu_f$ is injective. The short exact sequence $0 \rightarrow S \stackrel{\mu_f}{\longrightarrow} S \rightarrow S/(f) \rightarrow 0$ gives rise to the long exact sequence starting with 
\[ 0 \leftarrow S/(f) \otimes_R R/I \leftarrow S \otimes_R R/I \stackrel{\mu_f \otimes 1}{\longleftarrow} 
   S \otimes_R R/I  \leftarrow \mathrm{Tor}_1^R(S/(f), R/I) \leftarrow \mathrm{Tor}_1^R(S, R/I). \]
Since $S$ is a free $R$-module, $\mathrm{Tor}_1^R(S, R/I) = 0$, hence $\mathrm{Tor}_1^R(S/(f), R/I)$ is the kernel of the homomorphism $ S \otimes_R R/I \stackrel{\mu_f \otimes 1}{\longrightarrow} 
   S \otimes_R R/I$. Tensorizing the short exact sequence $0 \rightarrow I \rightarrow R \rightarrow R/I \rightarrow 0$ with $S$ gives rise to the exact sequence
\[ I \otimes_R S \rightarrow R \otimes_R S \rightarrow R/I \otimes_R S \rightarrow 0, \]
and since the image of $I \otimes_R S$ in $r \otimes_R S \cong S$ is $I \cdot S$, we deduce that 
\[ S \otimes_R R/I \cong R/I \otimes_R S \cong S/I \cdot S. \]
The map $\mu_f \otimes 1$ corresponds to the $R$-module homomorphism $S/I \cdot S \rightarrow S/I \cdot S$ induced by multiplication of elements in $S$ with $f$. On the other hand, $f \in I \cdot S$ by construction, so multiplication with $f$ induces the $0$-map, and the kernel is $\cong S/I \cdot S$. Then Prop.~\ref{Tor-flat-crit} implies that $S/(f)$ is flat if, and only if $S/I \cdot S = 0$, that is, $I \cdot S = S$ is the unit ideal in $S$. But since $S$ is a graded $R$-algebra with degree $0$ part $= R$ this is equivalent to $I \subseteq R$ being the unit ideal.

\bibliographystyle{alpha}

\bibliography{doktor}

\end{document}